\documentclass[ejs, preprint
]{imsart}

\doi{10.1214/154957804100000000}
\pubyear{0000}
\volume{0}
\firstpage{0}
\lastpage{0}

\newcommand{\declarecommand}[1]{\providecommand{#1}{}\renewcommand{#1}} 
\usepackage{amsmath, amssymb, color}
\usepackage{amsthm} 
\usepackage[colorlinks,citecolor=blue,urlcolor=blue]{hyperref}
\usepackage{xcolor,colortbl}
\usepackage{color}
\usepackage{amsfonts, fancybox}
\usepackage{amssymb}
\usepackage[american]{babel}
\usepackage{enumitem}
\usepackage{psfrag,color}
\usepackage{dcolumn}
\usepackage[format=hang, justification=centering, singlelinecheck=false]{subcaption}
\usepackage{flafter, bm, dsfont}
\usepackage[section]{placeins}

\usepackage{multirow}
\usepackage{color}
\usepackage{amsmath}
\usepackage{float}
\usepackage{mathrsfs}
\usepackage{amsfonts}
\usepackage{dsfont}
\usepackage{amssymb}
\usepackage{algorithmicx, algorithm, algpseudocode}
\MakeRobust{\Call}
\usepackage{graphicx}

\usepackage[noadjust]{cite}
\usepackage{tikz}
\usetikzlibrary{positioning}

\newcommand{\bc}{\begin{center}}
\newcommand{\ec}{\end{center}}
\newcommand{\ba}{\begin{array}}
\newcommand{\ea}{\end{array}}
\newcommand{\be}{\begin{eqnarray}}
\newcommand{\ee}{\end{eqnarray}}
\newcommand{\bel}{\begin{eqnarray}\label}
\newcommand{\eel}{\end{eqnarray}}
\newcommand{\bes}{\begin{eqnarray*}}
\newcommand{\ees}{\end{eqnarray*}}
\newcommand{\bn}{\begin{enumerate}}
\newcommand{\en}{\end{enumerate}}

\definecolor{MIT}{cmyk}{.24, 1.00, .78, .17} 
\definecolor{pink}{cmyk}{0, 1, 0, 0} 
\definecolor{darkgreen}{cmyk}{1,0, 1, 0}

\newcommand{\iid}{i.i.d.\ }

\newtheorem{theorem}{Theorem}
\newtheorem*{theorem*}{Theorem}
\newtheorem{lemma}[theorem]{Lemma}
\newtheorem{definition}[theorem]{Definition}

\newtheorem*{proposition*}{Proposition}

\newtheorem{remark}[theorem]{Remark}

\usepackage[utf8]{inputenc}
\newcommand{\leadeq}[2][4]{\MoveEqLeft[#1] #2 \nonumber}

\usepackage{stmaryrd}
\usepackage{bbm}
\usepackage[abbreviations]{foreign}

\newcommand{\mockalph}[1]{}

\algnewcommand\algorithmicinput{\textbf{Input:}}
\algnewcommand\Input{\item[\algorithmicinput]}
\algnewcommand\algorithmicoutput{\textbf{Output:}}
\algnewcommand\Output{\item[\algorithmicoutput]}

\numberwithin{equation}{section}
\usepackage{mathtools}
\mathtoolsset{showonlyrefs}

\makeatletter
\newcommand{\subalign}[1]{%
\vcenter{%
\Let@ \restore@math@cr \default@tag
\baselineskip\fontdimen10 \scriptfont\tw@
\advance\baselineskip\fontdimen12 \scriptfont\tw@
\lineskip\thr@@\fontdimen8 \scriptfont\thr@@
\lineskiplimit\lineskip
\ialign{\hfil$\m@th\scriptstyle##$&$\m@th\scriptstyle{}##$\crcr
#1\crcr
}%
}
}
\makeatother

\newcommand{\DS}{\displaystyle}

\newcommand{\cA}{\mathcal{A}}

\newcommand{\cC}{\mathcal{C}}
\newcommand{\cD}{\mathcal{D}}

\newcommand{\cF}{\mathcal{F}}

\newcommand{\cM}{\mathcal{M}}

\newcommand{\cX}{\mathcal{X}}

\newcommand{\R}{\mathbbm{R}}
\newcommand{\p}{\mathbbm{P}}
\newcommand{\E}{\mathbbm{E}}

\newcommand{\N}{\mathbbm{N}}
\newcommand{\1}{\mathbbm{1}}

\newcommand{\pn}{\p_{\kern-0.25em n}}
\newcommand{\pnm}{\p_{\kern-0.25em n,m}}
\newcommand{\psubm}{\p_{\kern-0.25em m}}
\newcommand{\psubp}{\p_{\kern-0.25em p}}
\newcommand{\cfi}{\cF_{\kern-0.25em \infty}}

\newcommand{\argmin}{\mathop{\mathrm{argmin}}}
\newcommand{\argmax}{\mathop{\mathrm{argmax}}}

\newcommand{\ud}{\mathrm{d}}

\newcommand{\eps}{\varepsilon}

\newlength{\minipagewidth}
\declarecommand{\ham}{d_\mathsf{H}}
\declarecommand{\KL}{\operatorname{\mathsf{KL}}}
\declarecommand{\hel}{\operatorname{\mathsf{h}}}
\declarecommand{\dimone}{n_1}
\declarecommand{\dimtwo}{n_2}
\declarecommand{\Done}{D}
\declarecommand{\Dtwo}{\tilde{D}}

\declarecommand{\AOone}{\1}
\declarecommand{\AOtwo}{\tilde{\1}}
\declarecommand{\widesim}[2][2]{\mathrel{\overset{#2}{\scalebox{#1}[1]{$\sim$}}}}
\declarecommand{\widesimiid}{\widesim{\text{i.i.d.}}}
\declarecommand{\subG}{\mathsf{subG}}

\declarecommand{\Poi}{\mathsf{Poi}}
\declarecommand{\notecm}[1]{{\sf{\color{blue}{[CM: #1]}}}}

\declarecommand{\pmin}{p^\ast_{\mathsf{min}}}
\declarecommand{\pmax}{p^\ast_{\mathsf{max}}}
\declarecommand{\defn}{:=}
\declarecommand{\constbox}{\mathcal{C}(Y)}
\declarecommand{\mle}{\hat \theta^{\mathsf{MLE}}}
\declarecommand{\pratio}{L}

\declarecommand{\notejc}[1]{{\sf{\color{orange}{[JC: #1]}}}}
\declarecommand{\premle}{\tilde \theta}
\declarecommand{\groundtruth}{\theta^\ast}
\declarecommand{\pgt}{p^\ast}
\declarecommand{\unnoise}{E}
\declarecommand{\samples}{N}
\declarecommand{\sample}{N}
\declarecommand{\density}{\rho}
\declarecommand{\densitygt}{\rho^\ast}
\declarecommand{\constholder}{R}
\declarecommand{\expholder}{\beta}
\declarecommand{\dimtradeoff}{n}
\declarecommand{\diam}{\operatorname{diam}}
\declarecommand{\dmin}{\density_{\mathrm{\min}}}
\declarecommand{\dmax}{\density_{\mathrm{\max}}}
\declarecommand{\Bin}{\mathsf{Bin}}
\declarecommand{\Multi}{\mathsf{Multi}}
\declarecommand{\classholder}{\cD}
\declarecommand{\classholderbd}{\tilde \cD}

\declarecommand{\constlipsq}{C_1}

\declarecommand{\mtp}{\mathrm{MTP}_2}

\declarecommand{\V}{V}

\begin{document}

\begin{frontmatter}


\title{Optimal Rates for Estimation of Two-Dimensional Totally Positive Distributions}
\runtitle{Estimation of Totally Positive Distributions}



%
%

\begin{aug}

\author{\fnms{Jan-Christian} \snm{H\"utter}\thanksref{a}\ead[label=e1]{jhuetter@broadinstitute.org}},
\author{\fnms{Cheng} \snm{Mao}\thanksref{b}\corref{}\ead[label=e2]{cheng.mao@math.gatech.edu}},
\author{\fnms{Philippe} \snm{Rigollet}\thanksref{c}\ead[label=e3]{rigollet@math.mit.edu}}
\and
\author{\fnms{Elina} \snm{Robeva}\thanksref{d}\ead[label=e4]{erobeva@math.ubc.ca}}


\runauthor{H\"utter, Mao, Rigollet and Robeva}

\address[a]{Broad Institute, 415 Main Street, Cambridge, MA, 02142, USA. \printead{e1}}
\address[b]{School of Mathematics, Georgia Institute of Technology, Suite 117, 686 Cherry Street, Atlanta, GA, 30332-0160, USA. \printead{e2}}
\address[c]{Department of Mathematics, Massachusetts Institute of Technology, Building 2, Room 106, 77 Massachusetts Avenue, Cambridge, MA, 02139-4307, USA. \printead{e3}}
\address[d]{Department of Mathematics, University of British Columbia, Room 121, 1984 Mathematics Road, Vancouver, BC, V6T 1Z2, Canada. \printead{e4}}

\end{aug}

\begin{abstract} 
We study minimax estimation of two-dimensional totally positive distributions. 
Such distributions pertain to pairs of strongly positively dependent random variables and appear frequently in statistics and probability. 
In particular, for distributions with $\beta$-H\"older smooth densities where $\beta \in (0, 2)$, we observe polynomially faster minimax rates of estimation when, additionally, the total positivity condition is imposed. 
Moreover, we demonstrate fast algorithms to compute the proposed estimators and corroborate the theoretical rates of estimation by simulation studies. 
\end{abstract}





\end{frontmatter}

\section{Introduction}

For a set $\cX = \prod_{i=1}^d \cX_i$ where each $\cX_i$ is totally ordered\footnote{A set $\cX$ is totally ordered if it is equipped with a total order, that is, a binary relation which is antisymmetric, transitive and connex. This work is only concerned with $\cX \subseteq \R$ equipped with its natural order.}, a function $p : \cX \to \R$ is called \emph{multivariate totally positive of order 2} ($\mtp$)~\cite{Karlin, KarRin80a} if 
\begin{align}
p(x\wedge y)p(x\vee y) \ge p(x)p(y)\,,  \qquad\forall\, x,y\in \mathcal{X},
\label{eq:mtp2}
\end{align}
where $\land$ and $\lor$ denote the coordinate-wise min and max operators respectively. 
The $\mtp$ condition is also known as the \emph{FKG lattice condition} because of its central role in the FKG inequality~\cite{fortuin1971correlation}. It is sometimes referred to as \emph{log-supermodularity} because of its similarity up to morphism to supermodularity~\cite{Hof63, QueSpiTar98}. 
Throughout, we say that a probability distribution is $\mtp$ if it has an $\mtp$ density. 

A variety of joint distributions are known to be $\mtp$, for example, order statistics of \iid variables, eigenvalues of Wishart matrices~\cite{KarRin80a}, and ferromagnetic Ising models~\cite{Leb72}. Furthermore, Gaussian and binary latent tree models are signed $\mtp$, that is, there exists a sign change of each coordinate making the distribution $\mtp$~\cite{KarRin81, LUZ}. In particular, all these distributions exhibit positive association, a marked feature of $\mtp$ distributions.
As opposed to positive association, however, the $\mtp$ property is preserved after conditioning or marginalization~\cite{KarRin80a}. 
As a result of their frequent appearances, $\mtp$ distributions have long been studied in statistics and probability~\cite{KarRin80a, karlinGaussian, BarFor00, colangelo2005some, slawski2015estimation, LUZ, RobStuTraUhl18}. 

In this paper, we study minimax estimation of an $\mtp$ distribution in dimension two\footnote{In dimension two, $\mtp$ is sometimes simply called  $\mathrm{TP}_2$ for \emph{totally positive of order 2}.} from \iid observations. 
We mainly focus on distributions on the square $[0, 1]^2$ for which density functions exist. 
Since almost surely no four observations from such a distribution form a rectangle, the $\mtp$ constraint~\eqref{eq:mtp2} is inactive on the observations and consequently, the maximum likelihood estimator over this class does not exist (see Lemma \ref{lem:ill-posedness} and Remark \ref{rem:ill-posedness} in Appendix \ref{sec:ill-posedness}). 
Therefore, we further assume that the distribution has a $\expholder$-H\"older smooth density, a widely adopted assumption in nonparametric estimation~\cite{Tsy09}. 

Smooth $\mtp$ distributions have long been studied in the literature. Examples include, but are not limited to, (1) pairwise marginals of Gaussian latent tree models~\cite{Fel73}, such as Brownian motion tree models and factor analysis models, (2) joint distributions of  pairs of time points of a strong Markov process on the real line with continuous paths~\cite{KarMcGre59}, such as a diffusion process, and (3) $\mtp$ transelliptical distributions, such as $\mtp$ multivariate $t$-distributions, which are commonly used in finance~\cite{AgrRoyUhl19}.

\paragraph{Main contribution}
Our main results can be stated informally as follows.
\begin{theorem}[Informal statement of minimax rates]
	\label{thm:informal}
Given \( N \) i.i.d. observations from a two-dimensional distribution with an $\mtp$ and $\expholder$-H\"older smooth density, the minimax rate of estimation in the squared Hellinger distance is (up to a polylogarithmic factor) 
\begin{equation}
\begin{cases}
N^{-\frac{2 \expholder}{2 \expholder + 1}}, &  \text{ if } \  0.5 \le  \expholder < 1,\\
N^{-2/3}, &  \text{ if } \  1 \le \expholder < 2,\\
N^{-\frac{2 \expholder}{2 \expholder + 2}}, & \text{ if } \   \expholder \ge 2.
\end{cases}
\end{equation}
\end{theorem}

It is well known that without the $\mtp$ assumption, the minimax rates for the $\expholder$-H\"older class in dimension $d$ scale as $N^{-\frac{2 \expholder}{2 \expholder + d}}$ up to a polylogarithmic factor, under various comparable models and error metrics (see, for example, \cite{Nem00}). 
Hence, our results show that for $0.5 \le \expholder < 1$, the minimax rate exhibits a one-dimensional behavior thanks to the $\mtp$ constraint; for $1 \le \expholder < 2$, the rate is polynomially faster than that without the $\mtp$ constraint and is independent of the smoothness parameter $\expholder$; for $\expholder > 2$, however, the $\mtp$ constraint has no effect on the minimax rate (see Figure \ref{fig:convergence_comparison} for a visualization). 

Note that results similar to what we obtain for $\mtp$ are expected to arise when $p$ is assumed to be smooth and log-concave. However, $\mtp$ only makes assumptions on the behavior of the function along lattice directions. While this is not directly comparable with log-concavity, it is, in essence, a weaker condition in the sense that it imposes a less stringent structure on the density. Our results indicate that when coupled with smoothness, $\mtp$ makes up for this deficiency and leads to the same rates of convergence. 

Our results for the regime $0 < \expholder < 0.5$ are unfortunately inconclusive, but the upper bounds exhibit polynomial improvement in the rates when $\mtp$ is assumed; see~\eqref{eq:gr} below.

\begin{figure}[H]
	\centering
		\includegraphics[width=0.45\textwidth]{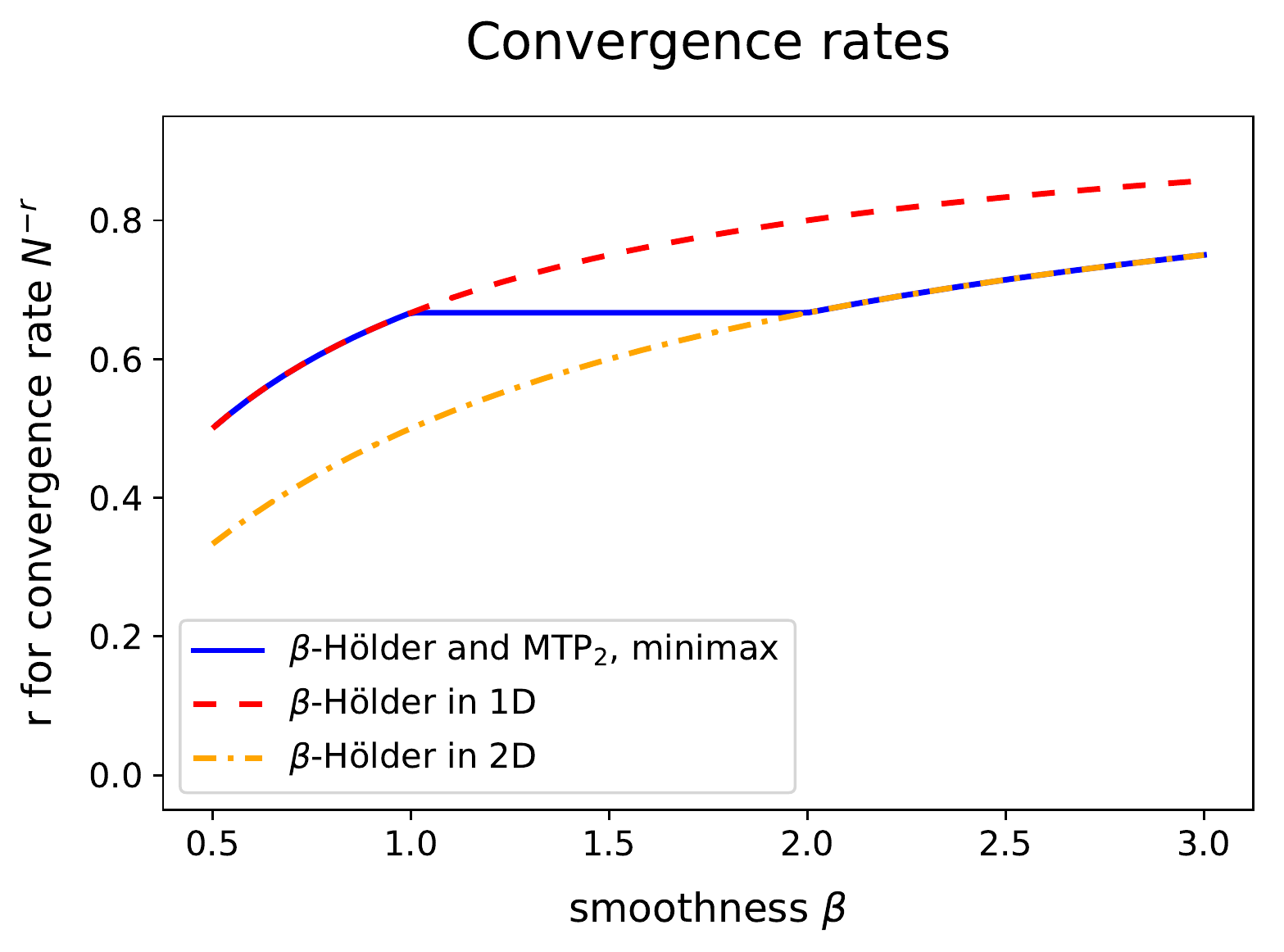}
		\caption{Visual comparison of the estimation rate for \( \beta \)-H\"older smooth \( \mtp \) distributions in Theorem~\ref{thm:informal}, with estimation rates for \( \beta \)-H\"older smooth distributions (without the $\mtp$ constraint) in 1D and 2D, suppressing logarithmic factors.
		}
	\label{fig:convergence_comparison}
\end{figure}



As a stepping stone to this problem, we also consider the following discrete version of $\mtp$. A distribution on the grid $[\dimone] \times [\dimtwo]$ is \( \mtp \) if its probability mass function (PMF) $p$, which is an $n_1\times n_2$ matrix, fulfills \eqref{eq:mtp2}. Thus, $\mtp$ says that all the $2\times 2$ minors of $p$ are non-negative:
\begin{align}
p_{ij}p_{k\ell} \geq p_{i\ell}p_{kj}, \qquad \text{for all } 1\leq i<k\leq n_1, 1\leq j<\ell\leq n_2.
\label{eq:grid-mtp2}
\end{align}
We study estimation of the PMF \( p \) from $N$ independent observations in this discrete setup.

To obtain upper bounds for estimation of a discrete $\mtp$ distribution, we employ a variant of the maximum likelihood estimator (MLE) defined in Section~\ref{sec:dis-est}. For estimating a smooth $\mtp$ density, we first discretize the space $[0, 1]^2$ and then apply the discrete MLE to obtain an estimator (defined in Section~\ref{sec:est-cts}) that achieves near-optimal upper bounds. Both estimators are computationally efficient, with the implementations discussed in more detail in Section~\ref{sec:algo}.

\medskip
\noindent \textbf{Related work.}
There has been a recent surge of interest in the estimation of $\mtp$ distributions. 
The special case of Gaussian $\mtp$ distributions has been studied by~\cite{slawski2015estimation, LUZ} from the perspective of maximum likelihood estimation and optimization. 
Maximum likelihood estimation of \emph{log-concave} $\mtp$ distributions was also analyzed recently~\cite{RSU, RobStuTraUhl18}. 
However, no statistical rate of estimation of $\mtp$ distributions is currently known. 
The present paper establishes the first minimax rates (up to logarithmic factors) of estimation of a smooth $\mtp$ density.

More broadly, our work falls into the scope of nonparametric density estimation which is a fundamental problem in nonparametric estimation. As such it has received considerable attention over the years~\cite{izenman1991review, scott2015multivariate, wasserman2016topological, chen2017tutorial, silverman2018density}. A central paradigm in this literature is to assume smoothness of the underlying density to be estimated. Such an assumption justifies a variety of statistical methods ranging from kernel density estimation to series expansions. Another approach to nonparametric estimation, and in particular to density estimation, is to use shape constraints whereby the (local) smoothness assumption is dropped and favored by a (global) synthetic constraint such as monotonicity~\cite{Grenander, Polonik}, convexity~\cite{Groeneboom, Seregin} and log-concavity~\cite{Walther_2002, Duembgen, CSS, Sam18} (see~\cite{GJ_review} for a recent overview). 
As explained above, the $\mtp$ constraint alone does not make the density estimation problem well-defined and it has been combined with another shape constraint, namely log-concavity, in~\cite{RobStuTraUhl18}. 
Instead, the present work combines $\mtp$ with smoothness to obtain a faster statistical rate than with smoothness alone, thus demonstrating compatibility of the local and the global approach. 


As we have discussed above, $\mtp$ is also called log-supermodular. In the recent paper \cite{HutMaoRigRob19}, we studied estimation of supermodular matrices (also known as \emph{anti-Monge} matrices) under sub-Gaussian noise. 
%
We note that the proof techniques used in~\cite{HutMaoRigRob19} are the starting point for the proofs in this paper, but are extended to the context density estimation.
In a parallel work \cite{FanGunSen19}, the authors study a related but slightly different model under Gaussian noise, and their proof techniques could potentially be extended to yield rates similar to the ones found in this paper.

\medskip
\noindent \textbf{Organization.}
We present the main results of the paper: upper and lower bounds for the discrete case in Section~\ref{sec:grid}, followed by the continuous case in Section~\ref{sec:cts}.
All proofs are postponed to Section~\ref{sec:proofs}.
The implementation of our estimators is discussed in Section~\ref{sec:algo}. 
Our theoretical results are complemented by numerical experiments on synthetic data in Section~\ref{sec:numerics}.
Finally, Section~\ref{sec:conclusion} includes a conclusion of the paper and a discussion of questions left for future research.

\medskip
\noindent \textbf{Notation.}
For a positive integer $n$, let $[n] = \{1, 2, \ldots, n\}$. For a
finite set $S$, we use $|S|$ to denote its cardinality. For two
sequences $\{a_n\}_{n=1}^\infty$ and $\{b_n\}_{n=1}^\infty$ of real numbers, we write
$a_n \lesssim b_n$ if there is a universal constant $C > 0$ such that $a_n
\leq C b_n$ for all $n \geq 1$. The relation $a_n \gtrsim b_n$ is
defined analogously.  We use $c$ and $C$ (possibly with subscripts) to denote 
%
universal positive constants that may change from line to line.  Given a matrix $M \in \R^{\dimone \times \dimtwo}$, we denote its $i$th row by $M_{i, \cdot}$ and its $j$th column by $M_{\cdot, j}$. 
For an entrywise positive vector $w \in \R^n$ and a vector $v \in \R^n$, we use the notation
\begin{equation}
\label{eq:fr}
\| v \|_{w} \defn \Big( \sum_{i=1}^{n} w_{i} v_i^2 \Big)^{1/2}
\end{equation}
for the $w$-weighted \( \ell_2 \) norm of the vector $v$. Similarly, for an entrywise positive matrix $b \in \R^{\dimone \times \dimtwo}$ and a matrix $a \in \R^{\dimone \times \dimtwo}$, we use $\|a\|_b$ to denote the $b$-weighted Frobenius norm of $a$.
For a reference measure $\mu$ on a (continuous or discrete) space $\mathcal{X}$, and two distributions with probability density or mass functions $p$ and $q$ respectively, we let 
\begin{align*}
\hel(p, q) \defn \Big( \int_{\mathcal{X}} \big( \sqrt{p(x)} - \sqrt{q(x)} \, \big)^2 d \mu(x) \Big)^{1/2}
\quad \text{ and } \quad 
\KL(p, q) \defn \int_{\mathcal{X}} p(x) \log \frac{p(x)}{q(x)} d \mu(x) 
\end{align*}
denote the Hellinger distance and the Kullback-Leibler (KL) divergence between the two distributions respectively.



\section{\texorpdfstring{$\mathbf{MTP_2}$}{MTP2} distribution estimation on a grid}
\label{sec:grid}


Let $p^\ast$ be a probability mass function (PMF) on the grid $[\dimone] \times [\dimtwo]$, where we assume without loss of generality that $\dimone \ge \dimtwo$. In the case where $\dimone \le \dimtwo$, our results and proofs remain valid with the roles of $\dimone$ and $\dimtwo$ swapped. 
Suppose that $p^*$ satisfies the $\mtp$ condition 
\begin{align}
\label{eq:discrete-mtp2}
p^\ast_{i, j} p^\ast_{i+1, j+1} \ge p^\ast_{i, j+1} p^\ast_{i+1, j} 
\qquad \text{for all } i \in [\dimone - 1] , \, j \in [\dimtwo - 1].
\end{align}
Note that this is equivalent to condition~\eqref{eq:grid-mtp2} by a telescoping sum argument.  
%

Suppose that we are given \( N \) i.i.d. observations $\{Z_k\}_{k=1}^N$ from the distribution on $[\dimone] \times [\dimtwo]$ with PMF $p^*$, that is, each $Z_k = (i,j)$ with probability $p^*_{i,j}$ for $(i,j) \in [\dimone] \times [\dimtwo]$. Our goal is to estimate \( p^\ast \). 
The number of observations at each point $(i, j)$ on the grid $[\dimone] \times [\dimtwo]$ is recorded in a matrix $Y = (Y_{i,j})_{i \in [\dimone], \, j \in [\dimtwo]}$, defined by 
\begin{align}
	\label{eq:jy}
Y_{i, j} \defn \sum_{k = 1}^N \1 \{ Z_k = (i, j) \} . 
\end{align}
Then $Y$ can be viewed as a multinomial random variable with distribution denoted by $\Multi(N, p^*)$. 

In addition, we define 
$$
\pmin \defn \min_{i \in [\dimone] , \, j \in [\dimtwo]} p^*_{i, j}
\quad \text{ and } \quad
\pmax \defn \max_{i \in [\dimone] , \, j \in [\dimtwo]} p^*_{i, j} ,
$$
and assume a mild lower bound on the sample size $\samples \ge 12 \log(\dimone \dimtwo / \delta)/\pmin$. 
Then $Y_{i,j}$ concentrates around its expectation as indicated by the next lemma. 
In particular, we have sufficiently many observations per entry on the grid with high probability. 
\begin{lemma} \label{lem:obs-bd}
For any $\delta \in (0, 1/2]$ and $\samples \ge 12 \log(\dimone \dimtwo / \delta)/\pmin$, it holds with probability at least $1 - 2 \delta$ that
$$\frac 12 \samples p^\ast_{i, j} \le Y_{i, j} \le \frac 32 \samples p^\ast_{i, j}$$ 
for all $(i, j) \in [\dimone] \times [\dimtwo]$. 
\end{lemma}

\begin{proof}
Note that marginally $Y_{i,j}$ follows the binomial distribution $\Bin(N, p^*_{i,j})$. 
Hence the result is an immediate consequence of Lemma~\ref{lem:bin-tail} with $q = p^*_{i, j}/2$, together with a union bound over $(i, j) \in [\dimone] \times [\dimtwo]$. 
\end{proof}


\subsection{Estimator}
\label{sec:dis-est}



We begin by describing the MLE  of the log-PMF $\groundtruth \in (-\infty, 0]^{\dimone \times \dimtwo}$ defined by $\groundtruth_{i, j} \defn \log p^\ast_{i, j}$. Owing to the fact that $p^*$ is a totally positive PMF, $\groundtruth$ satisfies the following two constraints:
$$
\DS\sum_{i \in [\dimone] , \,  j \in [\dimtwo]} e^{\groundtruth_{i, j}} = 1\,,\quad \text{and} \quad \Done \groundtruth \Dtwo^\top \geq 0\,,
$$
where the symbol $\geq$ denotes entrywise inequality and the difference operators $\Done \in \R^{(\dimone-1) \times \dimone}$, $\Dtwo \in \R^{(\dimtwo-1) \times \dimtwo}$ are both of the form
\begin{equation}
 \begin{bmatrix}
-1 & 1 & 0 &  \dots & 0 & 0\\
0 & -1 & 1 &  \dots & 0 & 0\\
\vdots & & &  & \vdots\\
0 & 0 & 0 &  \dots & -1 & 1
\end{bmatrix} . \label{eq:diff-op}
\end{equation}

The log-likelihood of a candidate $\theta=\log(p) \in (-\infty,0]^{n_1 \times n_2}$ is given by
\begin{align}
\log \prod_{k = 1}^N p_{Z_k} = 
\log \prod_{i \in [\dimone] , \,  j \in [\dimtwo]} (p_{i, j})^{Y_{i, j}} = \sum_{i \in [\dimone] , \,  j \in [\dimtwo]} Y_{i, j} \theta_{i, j} =\langle Y, \theta \rangle.
\end{align}
Hence the MLE is given by
\begin{align} \label{eq:mle-1}
\mle \defn \argmax_{\substack{\sum_{i, j} e^{\theta_{i, j}} = 1 \\ \Done \theta \Dtwo^\top \ge 0}} \langle Y, \theta \rangle .
\end{align}

Instead of the MLE, we study a constrained variant which is both amenable to analysis and efficiently computable\footnote{The MLE itself can also be efficiently computed; see Appendix~\ref{sec:exist} and Section~\ref{sec:numerics}.}. 
Lemma~\ref{lem:obs-bd} 
implies that with probability at least $1-2 \delta$, the true log-PMF $\groundtruth$ lies in the cube
\begin{align} \label{eq:constraint}
\constbox \defn \Big\{ \theta \in (-\infty, 0]^{\dimone \times \dimtwo} : \log \frac{2 Y_{i,j}}{3 \samples} \le \theta_{i,j} \le \log \frac{2 Y_{i,j}}{\samples} \text{ for all } i \in [\dimone], j \in [\dimtwo] \Big\} .
\end{align} 
This motivates the constrained optimization problem
\begin{align} \label{eq:est-1}
\tilde \theta \defn \argmax_{\substack{ \Done \theta \Dtwo^\top \ge 0 \\ \theta \in \constbox }} \frac{1}{\samples} \langle Y, \theta \rangle - \sum_{i \in [\dimone] , \,  j \in [\dimtwo]} e^{\theta_{i, j}} .
\end{align}
Note that the objective is concave and there are $O(\dimone \dimtwo)$ inequality constraints, so the program can be solved efficiently. However, since the constraint $\sum_{i, j} e^{\theta_{i, j}} = 1$ is replaced by a penalty term, it is not necessarily true that $\sum_{i, j} e^{\tilde \theta_{i, j}} = 1$. Hence we define the estimator of interest $\hat \theta \in \R^{\dimone \times \dimtwo}$ by normalizing $\tilde \theta$:
\begin{align} \label{eq:est-2}
\qquad \qquad \qquad 
\hat \theta_{i, j} \defn \tilde \theta_{i, j} - \log \sum_{r \in [\dimone], \, s \in [\dimtwo]} e^{\tilde \theta_{r, s}} \quad \text{ for } i \in [\dimone], \, j \in [\dimtwo] .
\end{align}
It is clear then that $\hat \theta$ is a supermodular log-PMF.
Finally, we define our estimator $\hat p = \hat p ( Y )$ by $\hat p_{i, j} \defn e^{\hat \theta_{i, j}}$, which is therefore a properly defined $\mtp$ PMF.

\subsection{Upper and lower bounds}

We measure the performance of our estimator $\hat p$ using the Hellinger distance $\hel(p^*, \hat p)$. 
%
%
For any PMF $p$ on the grid, define
\begin{equation} 
\pratio(p) \defn \frac{ p_{1,1} p_{\dimone, \dimtwo} }{ p_{\dimone, 1} p_{1, \dimtwo} }. \label{eq:ratio}
\end{equation}
The quantity $\log \big( \pratio(p) \big)$ is a seminorm of the log-PMF $\theta = \log (p)$ (see \eqref{eq:seminorm}), which measures the complexity of $\theta$. As a result, the following upper bound for our estimator $\hat p$ depends on $\log \big( \pratio(p^*) \big)$. 

\begin{theorem}[Upper bounds for estimation of discrete $\mtp$ distributions] \label{thm:upper}
Fix $\delta \in (0, 1/4]$ and suppose that we are given $N \ge  12 \log(\dimone \dimtwo / \delta)/\pmin$ independent observations from a distribution with an $\mtp$ PMF $p^\ast$ on the grid $[\dimone] \times [\dimtwo]$ where $\dimone \ge \dimtwo$. Then the estimator $\hat p$ defined above satisfies 
$$
\hel^2 ( p^\ast, \hat p ) \le \frac 12 \KL( p^\ast, \hat p ) \lesssim 
\frac{ \dimone \log(\dimone / \delta) }{\samples } 
+ ( \pmax \, \dimone \dimtwo)^{1/3}   
\Big( \log \big( \pratio(p^\ast) \big) + 1 \Big)^{2/3}  
\Big( \frac{ \log(\dimone / \delta) \log(\dimtwo) }{ \samples } \Big)^{2/3} 
$$
with probability at least $1- 4 \delta$. 
\end{theorem}

In particular, in the 
case where $\pmax \asymp 1/(\dimone \dimtwo)$, the bound in Theorem~\ref{thm:upper} reduces to 
\begin{equation}
	\label{eq:gb}
	\hel^2 ( p^\ast, \hat p ) \lesssim   \frac{\dimone}{\samples} +
	\frac{1}{\samples^{2/3}} 
\end{equation}
up to logarithmic factors. 
The term $\frac{1}{\samples^{2/3}}$ results from the $\mtp$ shape constraint, while the term $\frac{\dimone}{\samples}$ is present even if the PMF $p^*$ has constant rows. Technically, the two terms follow from a decomposition of the noise in the proof. 
The following theorem shows that this upper bound is, in fact, optimal in the minimax sense up to logarithmic factors.

\begin{theorem}[Lower bounds for estimation of discrete $\mtp$ distributions] \label{eq:lower} \label{thm:lower}
Let $\p_{p^\ast}$ denote the probability with respect to $\sample$ independent observations from the distribution with an $\mtp$ PMF $p^\ast$ on the grid $\dimone \times \dimtwo$. For $\dimone \le \sample \le \dimone^3 \dimtwo^3$, there exists a universal constant $c>0$ such that 
$$
\inf_{\tilde p} \sup_{ p^\ast \, \mtp } \p_{  p^\ast } \Big\{ \hel^2 ( p^\ast, \tilde p ) \ge c \Big( \frac{\dimone}{\sample} + \frac{1}{\sample^{2/3}} \Big) \Big\} 
\ge \frac 13 ,
$$
where the infimum is over all estimators $\tilde p$ measurable with respect to the observations.
For $\sample \le \dimone$, we have the vacuous lower bound of constant order. 
For $\sample \ge \dimone^3 \dimtwo^3$, we have the lower bound of order $\frac{\dimone \dimtwo}{\sample}$, which is the trivial rate of estimation.
\end{theorem}

Note that in the regime with an enormous sample size $\sample \ge \dimone^3 \dimtwo^3$, the lower bound $\frac{\dimone \dimtwo}{N}$ is achieved by the empirical frequency matrix $Y/N$ (see Appendix~\ref{sec:emp}), so there is no need to exploit the $\mtp$ constraint. 
In fact, it can be seen from the proof of Theorem~\ref{thm:upper} in Section \ref{sec:pf-ub} that the estimator \( \hat p \) also attains this rate up to logarithmic factors (see Remark~\ref{rem:ml-vs-emp}), a behavior that can be observed in the numerical experiments as well (see Figure~\ref{fig:pmf_large_V_var_N} in Section~\ref{sec:numerics}).

While our upper and lower bounds match in terms of the sample size $N$ and dimensions $(\dimone, \dimtwo)$, there are two potential improvements that can be made. 
First, the assumption $N \ge  12 \log(\dimone \dimtwo / \delta)/\pmin$ in Theorem~\ref{thm:upper} is necessary to guarantee that we have sufficient observations at each point on the grid $[\dimone] \times [\dimtwo]$, so that the (box-constrained) MLE can be properly defined and efficiently computed. 
There may exist other estimation procedures that apply in the regime where the sample size is smaller. 
Second, the upper bound contains the parameter $\pmax$ which is not present in the lower bound. 
This is likely an artifact of our proof of the upper bound and could potentially be mitigated.

\section{Smooth \texorpdfstring{$\mathbf{MTP_2}$}{MTP2} density estimation}
\label{sec:cts}

We turn to estimation of a probability distribution with a smooth $\mtp$ density $\densitygt$ on $[0, 1]^2$ with respect to the Lebesgue measure. 
Recall that $\mtp$ requires that for any $x, y \in [0, 1]^2$, 
\begin{align}
\densitygt (x \land y) \densitygt (x \lor y) \ge \densitygt (x) \densitygt (y) .
\label{eq:tp}
\end{align}
In addition, we assume that $\densitygt$ is \( \expholder \)-H\"older smooth, defined more precisely as follows.  
\begin{definition}
For \( \expholder, \constholder > 0 \), we define \( \classholder(\expholder, \constholder) \) to be the set of probability densities \( \density \) on \( [0,1]^2 \) such that \( \density \) is \( \lceil \expholder - 1 \rceil \) times continuously differentiable with
\begin{align}
	\label{eq:go}
	| \partial^{\alpha} \density(x) |
	\le {} & \constholder, \quad \text{for all } | \alpha | \le \lceil \expholder - 1 \rceil, \, x \in [0,1]^2 , \quad \text{and} \\
	| \partial^\alpha \density(x) - \partial^\alpha \density(y) |
	\le {} & \constholder \, \| x - y \|_2^{\expholder - \lceil \expholder - 1 \rceil}, \quad \text{for all } | \alpha | = \lceil \expholder - 1 \rceil \text{ and all } x, y \in [0,1]^2.
	\label{eq:fv}
\end{align}
Moreover, for $\dmin, \dmax > 0$, we define \( \classholderbd(\expholder, \constholder, \dmin, \dmax) \) to be the subset of \( \classholder(\expholder, \constholder) \) consisting of densities $\density$ such that
\begin{equation}
	\label{eq:fz}
	\dmin \le \density(x) \le \dmax, \quad \text{for all } x \in [0,1]^2 .
\end{equation}
\end{definition}

Equipped with the above definition, we assume that \( \densitygt \in \classholderbd(\expholder, \constholder, \dmin, \dmax ) \). 
Given \( N \) \iid observations from the distribution with density $\densitygt$ where $\samples > 0$, we aim to estimate the distribution.  

\subsection{Estimator}
\label{sec:est-cts}

To define an estimator of \( \densitygt \), we make use of both smoothness and the $\mtp$ assumption.
Namely, smoothness allows us to discretize the space $[0,1]^2$ into grid cells and group observations together in each cell, after which we are able to employ the $\mtp$ shape constraint. 

More precisely, for a positive integer \( \dimtradeoff \) to be determined later, we consider the equidistant discretization on \( [0,1]^2 \) with \( \dimtradeoff \) subdivisions on each dimension, that is, with grid cells 
\begin{equation}
	\label{eq:fx}
	S_{i, j} 
	\defn \Big[ \frac{i - 1}{\dimtradeoff}, \frac{i}{\dimtradeoff} \Big) \times \Big[ \frac{j - 1}{\dimtradeoff} , \frac{j}{\dimtradeoff} \Big) , 
	\quad i, j \in [\dimtradeoff] .
\end{equation}
Denote by \( Y \) the (unnormalized) histogram estimator with grid \( (S_{i, j})_{i,j=1}^{\dimtradeoff} \) for a sample \( \{ X_1, \dots, X_{N} \} \), that is,
\begin{equation}
	\label{eq:gq}
	Y_{i, j} \defn \sum_{k = 1}^{N} \1 \{ X_k \in S_{i, j} \} .
\end{equation}
Moreover, we define 
\begin{equation}
	\label{eq:fy}
	\pgt_{i, j} \defn \int_{S_{i, j}} \densitygt(x) \, \ud x .
\end{equation}
Since $\densitygt$ is $\mtp$, it is easily verified that the discrete density $\pgt$ is $\mtp$ in the sense of \eqref{eq:discrete-mtp2}. 

Given the matrix $Y$ with entries specified by \eqref{eq:gq}, we compute the estimator \( \hat p = \hat p(Y) \) defined in Section~\ref{sec:dis-est}, 
and define an estimator $\hat \density$ of the density $\densitygt$ by 
\begin{equation}
	\label{eq:gd}
	\hat \density(x) \defn \dimtradeoff^2 \hat p_{i, j} , \quad \text{ for } x \in S_{i, j} ,
\end{equation}
which is a piecewise constant estimator on the grid \( (S_{i, j})_{i, j=1}^{\dimtradeoff} \).

\subsection{Upper and lower bounds}

The performance of our estimator $\hat \density$ with respect to the Hellinger distance is characterized by the following theorem. 

\begin{theorem}[Upper bounds for estimation of smooth $\mtp$ distributions]
	\label{thm:upper-cts}
	Suppose that we are given \( N \) independent observations from an $\mtp$ distribution with density \( \densitygt \in \classholderbd(\expholder, \constholder, \dmin, \dmax) \). 
	With $\tilde \expholder \defn \expholder \land 1$ and the choice 
	\begin{equation}
		\dimtradeoff = \Big\lfloor \Big( \frac{\constholder^2 \samples}{\dmin}  \Big)^{1/(2 \tilde \expholder + 1)} \land \Big( \frac{\dmin \samples }{ \log ( \dmin \samples ) } \Big)^{1/2} \Big\rfloor  , 
	\end{equation}
we define the estimator \( \hat \density \) as in \eqref{eq:gd}. 
Moreover, suppose that $\samples$ is larger than a constant depending on $\beta, R$ and $\dmin$.  
	Then with probability at least $1 - \samples^{-4}$, the following holds. If $\expholder > 0.5 $, then
	\begin{equation}
		\hel^2(\hat \density, \densitygt) 
		\lesssim 
		\frac{  \log \samples  }{ \samples^{2 \tilde \expholder/(2 \tilde \expholder + 1)} }  + 
		\frac{ (\log \samples )^{4/3} }{ \samples^{2/3} } ,
	\end{equation} 
	and if $0 < \expholder \le 0.5$, then
		\begin{align}
			\hel^2(\hat \density, \densitygt) &\lesssim \Big( 
			\frac{ \log \samples  }{ \samples } \Big)^{ \tilde \expholder } 
			+ \frac{ (\log \samples )^{4/3} }{ \samples^{2/3} } ,
		\end{align}
where the suppressed constants depend on the quantities $\beta, R, \dmin$ and $\dmax$. 
\end{theorem}

Note that the size of discretization $n$ can be viewed as a tuning parameter in smooth density estimation---the larger $n$ is, the smaller bias and larger variance the piecewise constant estimator has. As the proof of Theorem~\ref{thm:upper-cts} suggests, the above choice of $n$ achieves the optimal bias-variance trade-off, thereby yielding near-optimal upper bounds. 

As in the discrete setting, each of the above upper bounds contains two terms. The term involving $\tilde{\expholder}$ in the exponent originates from the smoothness of the density, while the term $\frac{ (\log \samples )^{4/3} }{ \samples^{2/3} }$ is due to the $\mtp$ shape constraint. 


More precise versions of the bounds in Theorem~\ref{thm:upper-cts} are given in \eqref{eq:ub1} and \eqref{eq:ub2} with explicit dependencies on $\dmin$, $\dmax$, and $\constholder$. 
In particular, treating these quantities as constants, the above bounds yield (up to logarithmic factors) that 
\begin{equation}
	\label{eq:gr}
	\hel^2(\hat \density, \densitygt)
	\lesssim \left\{
	\begin{aligned}
	&\samples^{- \beta}, & \quad & 0 < \expholder < 0.5, \\
		&\samples^{- \frac{2 \expholder}{2 \expholder + 1}}, &\quad &0.5 \le \expholder < 1,\\
		&\samples^{-2/3}, &\quad &\expholder \ge 1 ,
	\end{aligned}
	\right.
\end{equation}
with high probability. 
These upper bounds are complemented by the following lower bounds. 

\begin{theorem}[Lower bounds for estimation of smooth $\mtp$ distributions]
	\label{thm:lower-cts}
	Let \(  \p_{\densitygt} \) denote the probability with respect to \( \sample \) independent observations from the distribution with density \( \densitygt \in \classholder(\expholder, R) \), for \( \expholder > 0 \) and \( R \ge 1 \).
	Then there exists a universal constant $c>0$ such that 
	\begin{equation}
		\label{eq:hv}
		\inf_{\tilde \density} \sup_{\densitygt \in \classholder(\expholder, R)} \p_{\densitygt}\left( \hel^2(\tilde \density, \densitygt) \ge c \phi_\expholder(\sample) \right) \ge \frac{1}{3},
	\end{equation}
	where the infimum is taken over all estimators \( \tilde \density \) measurable with respect to the observations and 
\begin{equation}
\label{eq:hx}
\phi_{\expholder}(\sample)
= \left\{
\begin{aligned}
&\sample^{-\frac{2 \expholder}{2 \expholder + 1}}, & \quad & \text{if } 0 < \expholder < 1,\\
&\sample^{-2/3}, & \quad & \text{if } 1 \le \expholder < 2,\\
&\sample^{-\frac{2 \expholder}{2 \expholder + 2}}, & \quad & \text{if } \expholder \ge 2.
\end{aligned}
\right.
\end{equation}
\end{theorem}

The lower bounds above match the upper bounds up to logarithmic factors in the regime $0.5 \le \expholder \le 2$. 
For $0 < \expholder < 0.5$, the rate \( N^{-\beta} \) coincides with the rate obtained by the nonparametric H\"older-constrained estimator in dimension one \cite{BirMas93}.
This slow rate is known to be suboptimal, as sieve estimators attain the optimal rate \( N^{-2 \beta/(2 \beta + 1)} \).
While we conjecture that up to log factors, the latter should also be the optimal rate in our case, we leave the problem of finding the optimal rate in this regime as an open question for future research.
For $\expholder > 2$, the rate is the same as that for $\expholder$-H\"older smooth density estimation without the $\mtp$ assumption. 
Therefore, while the lower bound is interesting in our setup, to obtain the upper bound, it suffices to use any existing rate-optimal estimator.

\begin{remark}
	The construction of the estimator \( \hat \density \) depends both on the smoothness parameter \( \expholder \) and the H\"older constant \( \constholder \), and it does not match the lower bounds in the case \( \expholder > 2 \).
	Both of these shortcomings can be remedied by considering an ensemble of estimators that include both our estimators \( \hat \density \) for varying parameters \( \tilde \expholder \) and \( \tilde \constholder \) over a discretization of the set of parameters, and, for example, regular kernel density estimators which are rate-optimal for H\"older smooth density estimation. 
	Over such an ensemble, either selection \cite{Mas07} or aggregation procedures \cite{JudRigTsy08} can be used to achieve adaptive rates that match the lower bounds up to logarithmic factors. 
	Since the techniques are standard and yield no new phenomenon, we do not pursue this direction in the current work. 
\end{remark}


\section{Efficient algorithms}
\label{sec:algo}

The optimization problem for finding the constrained MLE in \eqref{eq:est-1} is a convex problem with a polynomial number of constraints and can thus be solved in polynomial time with a general purpose solver for convex problems such as SCS \cite{DonChuPari16, scs} or ECOS \cite{DomChuBoy13}.
However, since the number of constraints is of the order \( \dimone \dimtwo \), solving the linear systems in each iteration step of these solvers can take a long time without specialized solvers.
We address this issue by employing a proximal Newton method, whose main step consists in a projection onto the set of constraints, which in turn can be solved by a variant of Dykstra's algorithm, as discussed in \cite{HutMaoRigRob19}.
In this section, in order to emphasize the connection to computing projections, we think about \eqref{eq:est-1} as a minimization problem instead of a maximization problem by changing the sign in front of the objective.

First, we derive the outer iteration of our algorithm as a proximal Newton method.
These methods are intended to solve nonlinear optimization problems by successively solving local quadratic approximations to the objective functions.
For a more thorough introduction to this class of methods, see \cite{LeeSunSau14a}.
Briefly, for \( d \in \mathbb{N} \), to minimize a composite function of the form
\begin{equation}
	\label{eq:iv}
	\min_{\theta \in \R^d} f(\theta), \quad f(\theta) = g(\theta) + h(\theta),
\end{equation}
one starts with an initialization \( x^{(0)} = x_0 \in \R^d \) and computes updates by solving
\begin{align}
	\label{eq:ik}
	\rho^{(k)} = {} & \argmin_{\tilde \rho \in \R^d} \nabla g(\theta^{(k-1)})^\top (\tilde \rho - \theta^{(k-1)}) + \frac{1}{2}(\tilde \rho - \theta^{(k-1)})^\top \nabla^2 g(\theta^{(k-1)})(\tilde \rho - \theta^{(k-1)}) + h(\tilde \rho)\\
	\theta^{(k)} = {} & \theta^{(k-1)} + t_k (\eta^{(k)} - \theta^{(k-1)}),
\end{align}
where \( t_k \) is usually chosen by a line-search technique.
In the case of \eqref{eq:est-1}, we set
\begin{align}
	\label{eq:iw}
	g(\theta) = {} & -\frac{1}{\samples} \langle Y , \theta \rangle + \sum_{i, j} e^{\theta_{i,j}}\\
	h(\theta) = {} & 
	\left\{
	\begin{aligned}
		& 0, & \quad & \theta \in \cM \cap \cC(Y),\\
		& +\infty, & \quad & \theta \notin \cM \cap \cC(Y),
	\end{aligned}
	\right.
\end{align}
where \( \cC(Y) \) corresponds to the box constraints defined in \eqref{eq:constraint}, and
\begin{align}
	\label{eq:ix}
	\cM := \{ \theta \in \R^{\dimone \times \dimtwo} : \Done \theta \Dtwo^\top \ge 0 \}.
\end{align}

Further computation then shows that the Hessian \( \nabla^2 g(\theta) \) has the structure of a diagonal operator, which makes the subproblem of computing \( \rho^{(k)} \) equivalent to finding a projection with respect to a weighted Frobenius norm.
Namely, computing the first and second derivatives yields
\begin{align}
	\label{eq:iy}
	(\nabla g(\theta))_{i_1, i_2} = {} & -\frac{1}{\samples} Y_{i,j} + \exp(\theta_{i,j}) , \\
	(\nabla^2 g(\theta))_{(i_1, i_2), (j_1, j_2)} = {} & 
	\left\{
	\begin{aligned}
		& \exp(\theta_{i_1, i_2}), & \quad & (i_1, i_2) = (j_1, j_2), \\
		& 0, & \quad & \text{otherwise}.
	\end{aligned}
	\right.
\end{align}
Hence, writing
\begin{equation}
	\label{eq:iz}
	\Lambda_{i_1, i_2} = \exp(\theta^{(k)}_{i_1, i_2}),
\end{equation}
computing \( \rho^{(k)} \) is equivalent to
\begin{align}
	\label{eq:ja}
	\rho^{(k)}
	= {} & \argmin_{\tilde \rho \in \cM \cap \cC(Y)} \Big\langle -\frac{1}{\samples} Y + \Lambda, \tilde \rho \Big\rangle + \frac{1}{2} \Big\| \tilde \rho - \theta^{(k-1)} \Big\|_{\Lambda}^2\\
	= {} & \argmin_{\tilde \rho \in \cM \cap \cC(Y)} \frac{1}{2} \Big\| \tilde \rho - \Big( \theta^{(k-1)} + \frac{1}{\samples} Y \oslash \Lambda - \1 \Big) \Big\|_{\Lambda}^2 
	+ \Big\langle \theta^{(k-1)} , \Lambda - \frac{1}{\samples} Y \Big\rangle - \frac{1}{2} \Big\| \Lambda - \frac{1}{\samples} Y \Big\|_{1 \oslash \Lambda}^2\\
	= {} & \argmin_{\tilde \rho \in \cM \cap \cC(Y)} \frac{1}{2} \Big\| \tilde \rho - \Big( \theta^{(k-1)} + \frac{1}{\samples} Y \oslash \Lambda - \1 \Big) \Big\|_{\Lambda}^2,
	\label{eq:jh}
\end{align}
the projection of \( \theta^{(k-1)} + \frac{1}{\samples} Y \oslash \Lambda - \1 \) onto \( \cM \cap \cC(Y) \) with respect to the Frobenius norm weighted by \( \Lambda \).

Second, problem \eqref{eq:jh} can be efficiently solved by a variant of Dykstra’s algorithm, as shown in \cite{HutMaoRigRob19}.
The idea is to split up the projection onto \( \cM \cap \cC(Y) \) into the projection onto \( \cC(Y) \) and the sets
\begin{equation}
	\label{eq:jb}
	\cM_{i_1, i_2} = \Big\{ \theta \in \R^{\dimone \times \dimtwo} : \sum_{j_1 \in \{0,1\}, j_2 \in \{0,1\}} (-1)^{j_1 + j_2} \theta_{i_1 + j_1, i_2 + j_2} \ge 0 \Big\}, \quad i_1 \in [\dimone - 1], i_2 \in [\dimtwo - 1],
\end{equation}
where additional correction terms are applied to the vectors to ensure convergence.
The basic Dysktra algorithm for projecting a vector \( y \in \R^d \) onto a general collection of sets \( \cM_1, \dots, \cM_m \) is listed as Algorithm~\ref{alg:dykstra}.

\begin{algorithm}[H]
\normalsize
	\caption{{Dykstra algorithm}}
	\label{alg:dykstra}
	\begin{algorithmic}
		\Input \( y \in \R^d \)
		\Output \( \theta \approx \Pi_{\cM}(y) \)
		\Function{ProjectDykstra}{y}
		\For{\( i = 1, \dots, m \)}
			\State{\( p_i = 0_d \)}
			\Comment Initialize residuals
		\EndFor
		\State{\( \theta_m = y \)}
		\Comment Initialize iterates
		\While{not converged}
		\For{\( i = 1, \dots, m \)}
		\State \( \theta_i \gets \Pi_{\cM_i}(\theta_{(i - 2) \% m + 1} + p_i) \)
		\Comment Project shifted iterates
		\State \( p_i \gets \theta_{(i - 2) \% m + 1} + p_i - \theta_i \)
		\Comment Compute new residual
		\EndFor
		\EndWhile
		\State \Return \( \theta \)
		\EndFunction
	\end{algorithmic}
\end{algorithm}

In our case, the projection onto \( \cM_{i_1, i_2} \) with weight matrix \( \Lambda \), written as \( \Pi_{\cM_{i_1, i_2}, \Lambda} \), has the following closed form solution.
For \( i_1 \in [\dimone - 1], i_2 \in [\dimtwo - 1]\), let \( \Lambda \) be given is in \eqref{eq:iz} and set
\begin{equation}
	\label{eq:je}
	\Gamma_{i_1, i_2} = \left(\sum_{j_1, j_2 \in \{0, 1\}} \frac{1}{\Lambda_{i_1 + j_1, i_2 + j_2}}\right)^{-1}.
\end{equation}
Then, we have for \( j_1, j_2 \in \{0,1\} \) that
\begin{align}
	\label{eq:jc}
	\leadeq{(\Pi_{\cM_{i_1, i_2}, \Lambda} y)_{i_1 + j_1, i_2 + j_2} }\\
	= {} & Z_{i_1 + j_1, i_2 + j_2}
	+ \frac{(-1)^{j_1 + j_2}}{\Lambda_{i_1 + j_1, i_2 + j_2}} \max \left\{ -\Gamma_{i_1, i_2}\sum_{k_1 \in \{0,1\}, k_2 \in \{0,1\}} (-1)^{k_1 + k_2} Z_{i_1 + k_1, i_2 + k_2}, 0 \right\},
\end{align}
and for \( (l_1, l_2) \notin (i_1 + \{0, 1\}) \times (i_2 + \{0, 1\})  \) that
\begin{equation}
	\label{eq:jd}
	(\Pi_{\cM_{i_1, i_2}, \Lambda} y)_{l_1, l_2} =  Z_{l_1, l_2}.
\end{equation}
Together with the closed form solution for projecting onto the box \( \cC(Y) \),
\begin{equation}
	\label{eq:jf}
	(\Pi_{\cC(Y)}(y)_{i_1, i_2}) =
	\left\{
	\begin{aligned}
		& Z_{i_1, i_2}, & \quad & Z_{i_1, i_2} \in \left[\log \frac{2 Y_{i,j}}{3 \samples}, \log \frac{2 Y_{i_1,i_2}}{\samples}\right] \\
		& \log \frac{2 Y_{i_1, i_2}}{\samples}, & \quad & Z_{i_1, i_2} > \log \frac{2 Y_{i_1, i_2}}{\samples},\\
		& \log \frac{2 Y_{i_1, i_2}}{3 \samples}, & \quad & Z_{i_1, i_2} < \log \frac{2 Y_{i_1, i_2}}{3 \samples},
	\end{aligned}
	\right.
\end{equation}
we end up with the iterative projection algorithm given in Algorithm \ref{alg:fast-proj}, which in turn leads to the proximal Newton method, Algorithm \ref{alg:prox-newton}.
Note that we did not implement a line search but instead chose to directly update our iterates with \( \rho^{(k)} \), which seems to not pose any problems in practice.

Similarly, if instead of the box-constrained estimator \eqref{eq:est-2}, we are interested in calculating the regular MLE over \( \mtp \), \eqref{eq:mle-1}, we can omit the projection onto \( \cC(Y) \) in Algorithm~\ref{alg:fast-proj}.

\begin{algorithm}[H]
\normalsize
	\caption{{Fast projection onto \( \cM \)}}
	\label{alg:fast-proj}
	\begin{algorithmic}
		\Function{Project}{$ y, \Lambda, \cC(Y) $}
		\State \( \theta \gets y, \quad \eta \gets 0 \in \R^{(\dimone - 1) \times (\dimtwo - 1)}, \quad \eta' \gets 0 \in \R^{\dimone \times \dimtwo} \)
		\Comment Initialize \( \theta \) and residuals \( \eta, \eta' \)
		\For{\( i_1 = 1, \dots, \dimone - 1, \, i_2 = 1, \dots, \dimtwo-1 \)}
			\State \( \Gamma_{i_1, i_2} \gets \left(\sum_{j_1, j_2 \in \{0, 1\}} (\Lambda_{i_1 + j_1, i_2 + j_2})^{-1} \right)^{-1} \)
			\Comment Initialize harmonic mean of weights
		\EndFor
		\While{not converged}
		\State \( \theta' \gets \Pi_{\cC(Y)}(\theta + \eta')\)
		\Comment Project onto \( \cC(Y) \) \ldots
		\State \( \eta' \gets \theta + \eta' - \theta', \quad \theta \gets \theta' \)
		\Comment \ldots and store corresponding residual
		\For{\( i_1 = 1, \dots, \dimone - 1,\, i_2 = 1, \dots, \dimtwo - 1 \)}
			\Comment Project onto \( \cM \) by projecting onto all \( \cM_{i_1, i_2} \) in turn
			\State \( \tilde \eta \gets \max \left\{ \eta_{i_1, i_2} -\Gamma_{i_1, i_2}\sum_{k_1 \in \{0,1\}, k_2 \in \{0,1\}} (-1)^{k_1 + k_2} \theta_{i_1 + k_1, i_2 + k_2}, 0 \right\} \)
			\For{\( j_1 \in \{0, 1\}, \, j_2 \in \{0, 1\} \)}
			\State \( \theta_{i_1 + j_1, i_2 + j_2} \gets \theta_{i_1 + j_1, i_2 + j_2} + (-1)^{j_1 + j_2} \Lambda_{i_1 + j_1, i_2 + j_2}^{-1} (\tilde \eta - \eta_{i_1, i_2}) \)
			\EndFor
			\State \( \eta_{i_1, i_2} \gets \tilde \eta \)
		\EndFor
		\EndWhile
		\State \Return \( \theta \)
		\EndFunction
	\end{algorithmic}
\end{algorithm}

\begin{algorithm}[ht]
\normalsize
	\caption{Restricted ML solution via proximal Newton method}
	\label{alg:prox-newton}
	\begin{algorithmic}
		\Function{RestrictedMaximumLikelihood}{$Y$}
		\State \( \theta \gets Y/\samples \)
		\While{not converged}
		\For{\( i_1 = 1, \dots, \dimone, i_2 = 1, \dots, \dimtwo \)}
		\State \( \Lambda_{i_1, i_2} \gets \exp(\theta_{i, j}) \)
		\Comment Update weight matrix
		\EndFor
		\State \( y \gets \theta + Y/\samples \oslash \Lambda - \1 \)
		\State \( \theta \gets \Call{Project}{y, \Lambda, \cC(Y)} \)
		\Comment perform Newton step
		\EndWhile
		\State \Return \( \theta \).
		\EndFunction
	\end{algorithmic}
\end{algorithm}

In practice, convergence in Algorithm \ref{alg:fast-proj} can be checked by computing a measure of feasibility such as \( 0 \lor \max \{ -(\Done \theta \Dtwo)_{i,j} : i \in [\dimone - 1], j \in [\dimtwo - 1] \} \) or stopping when the distance between successive iterates becomes small.
Similarly, we stop the proximal Newton method, Algorithm \ref{alg:prox-newton}, when two successive iterates become very close to each other or the gain in the objective function is very small.

Following these considerations, computing the density estimator \eqref{eq:gd} is straightforward by computing a histogram of the samples in \( [0,1]^2 \) and applying Algorithm \ref{alg:prox-newton} to obtain \( \hat p \), which yields the piecewise constant approximation in \eqref{eq:gd}.

\section{Numerical experiments}
\label{sec:numerics}

This section is devoted to simulations which corroborate our theoretical findings. 
Further details on the underlying implementation can be found in Appendix \ref{sec:num-extended}.

\subsection{Experiments for the grid estimator}

In this section, we set \( n = \dimone = \dimtwo \) for simplicity.

We consider the following family of ground truth probability mass functions:
Let \( n \in \mathbb{N} \), set
\begin{equation}
	\label{eq:gl}
	\tilde \theta_{i, j} = 1 + \log(L) \frac{(i-1)(j-1)}{(n-1)^2}, \quad i, j \in [n],
\end{equation}
for \( L > 1 \) that can be varied and
\begin{equation}
	\label{eq:ir}
	p^\ast_{i, j} = \frac{\exp(\tilde \theta_{i, j})}{\sum_{i, j} \exp(\tilde \theta_{i,j})}, \quad i, j \in [n].
\end{equation}
By construction, \( \tilde \theta \) is supermodular, so \( p^\ast \) is \( \mtp \), and \( \log(L(p^\ast)) = \log(L) \).
We sample \( N \) \iid observations \( \{Z_k\}_{k=1}^N \) from \( p \) and form the matrix \( Y \) as in \eqref{eq:jy}.
We consider three estimators for \( p^\ast \): the empirical frequency matrix \( Y/N \), the MLE given by \eqref{eq:mle-1}, and the box-constrained estimator in \eqref{eq:est-2}.
The latter two estimators are computed by variants of Algorithm \ref{alg:prox-newton}.
Note that in cases where some of the entries in the frequency matrix are zero, we cannot take logarithms and thus report \( Y/N \) as the output of the box-constrained estimator, while the unconstrained MLE can be calculated as in Section~\ref{sec:algo} above.
For the unconstrained MLE, we do observe numerical instabilities when the number of observations is very low, eventually leading to underflows in the calculation.
This can be remedied by imposing mild lower bounds on the resulting density, see Appendix~\ref{sec:num-instability}.

\begin{figure}[ht]
	\centering
	\begin{subfigure}{0.45\textwidth}
		\includegraphics[width=\textwidth]{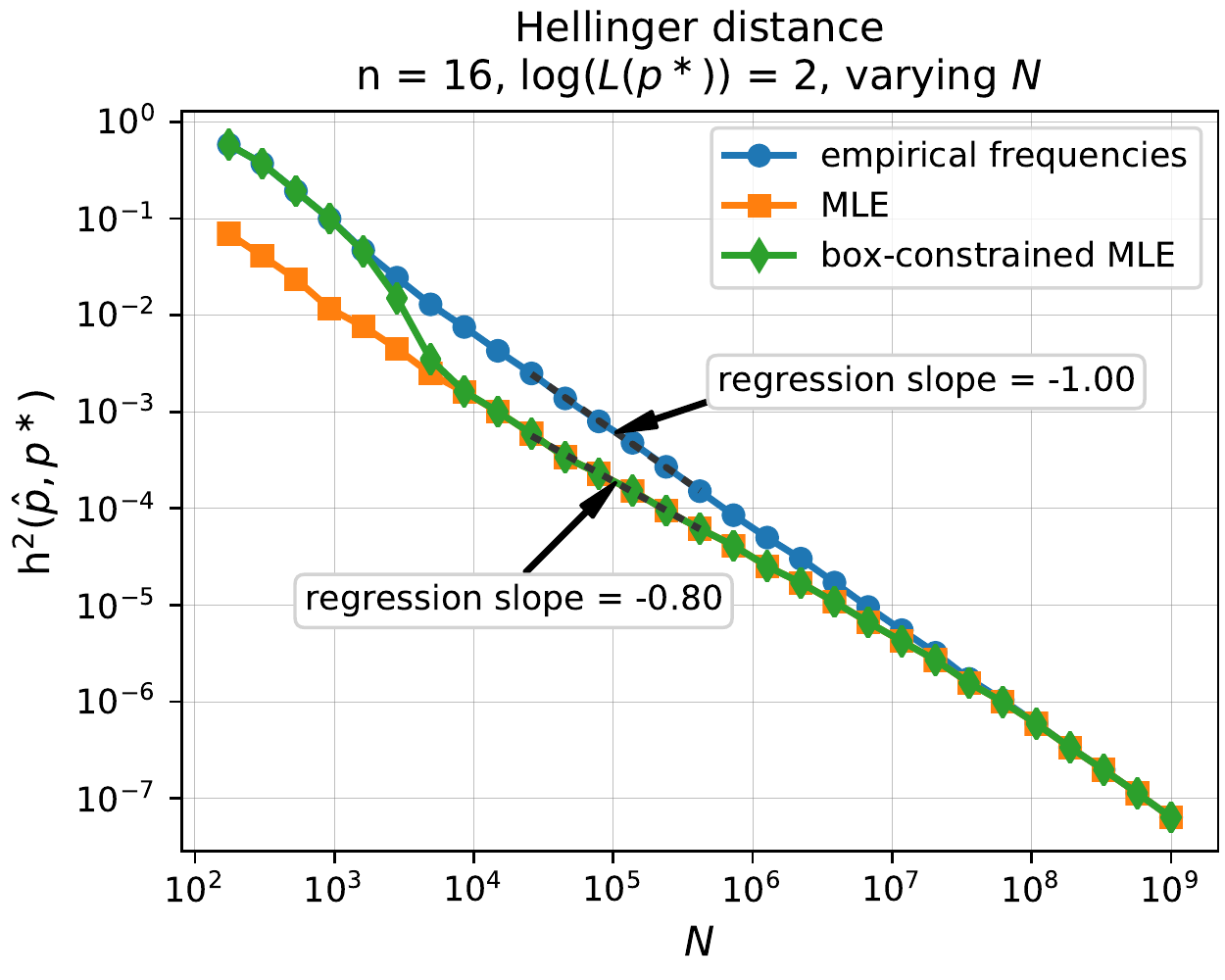}
		\caption{Varying \( N \), \( \log(L(p^\ast)) = 2 \)}
		\label{fig:pmf_large_V_var_N}
	\end{subfigure}
	\hspace{1em}
	\begin{subfigure}{0.45\textwidth}
		\includegraphics[width=\textwidth]{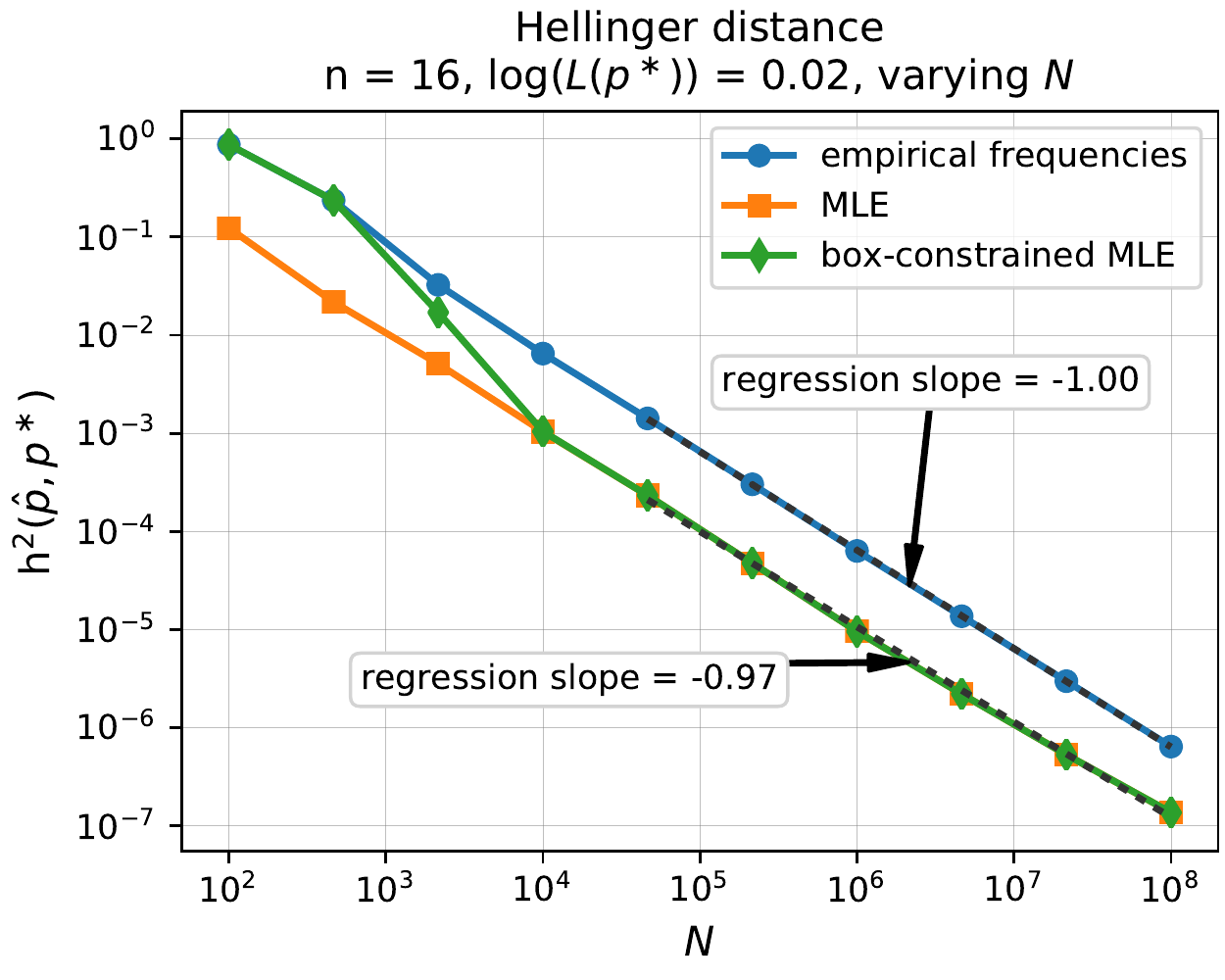}
		\caption{Varying \( N \), \( \log(L(p^\ast)) = 0.02 \)}
		\label{fig:pmf_small_V_var_N}
	\end{subfigure}
	\begin{subfigure}{0.45\textwidth}
		\includegraphics[width=\textwidth]{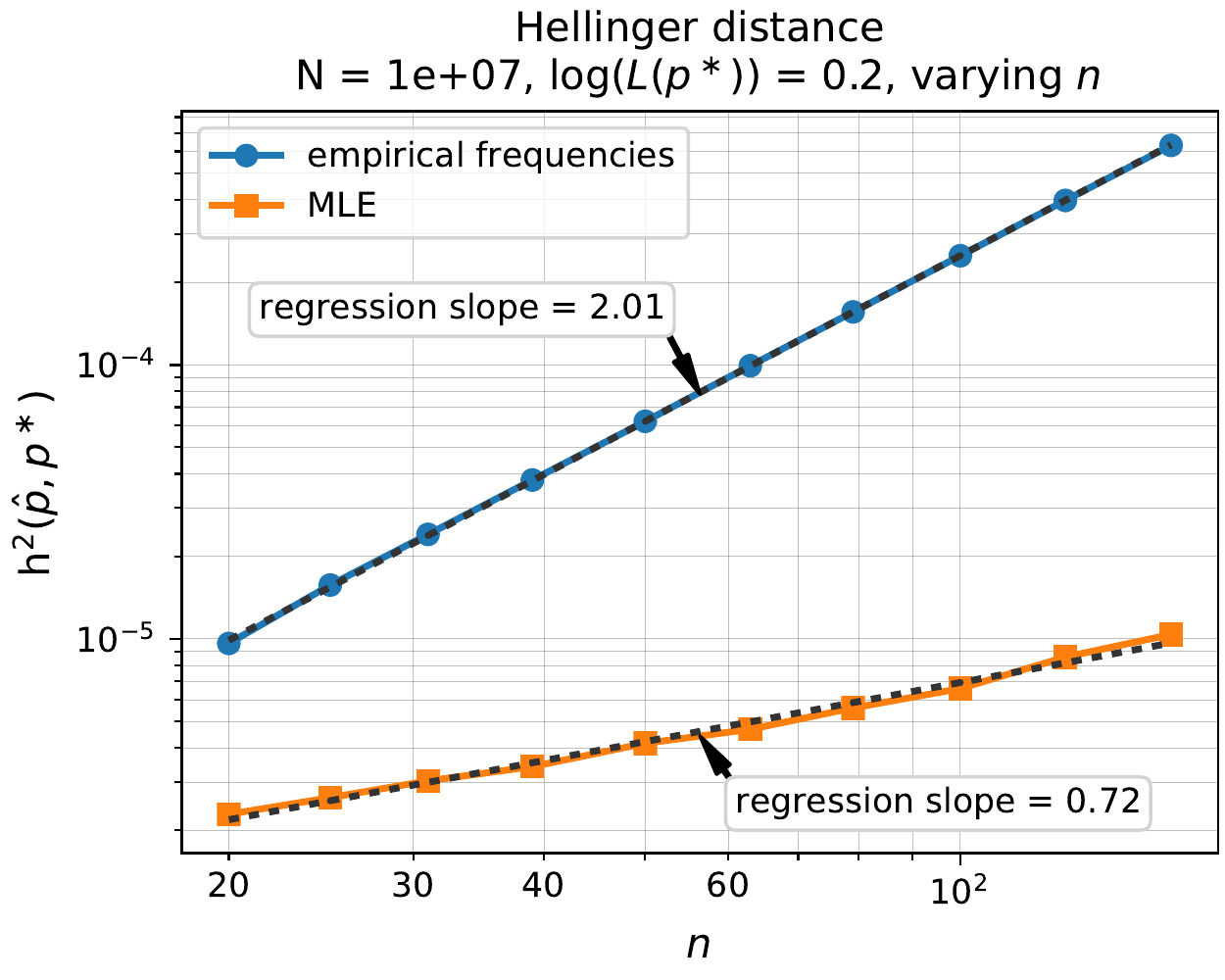}
		\caption{Varying \( n \), \( \log(L(p^\ast)) = 0.2 \)}
		\label{fig:pmf_medium_V_var_n}
	\end{subfigure}
	\hspace{1em}
	\begin{subfigure}{0.45\textwidth}
		\includegraphics[width=\textwidth]{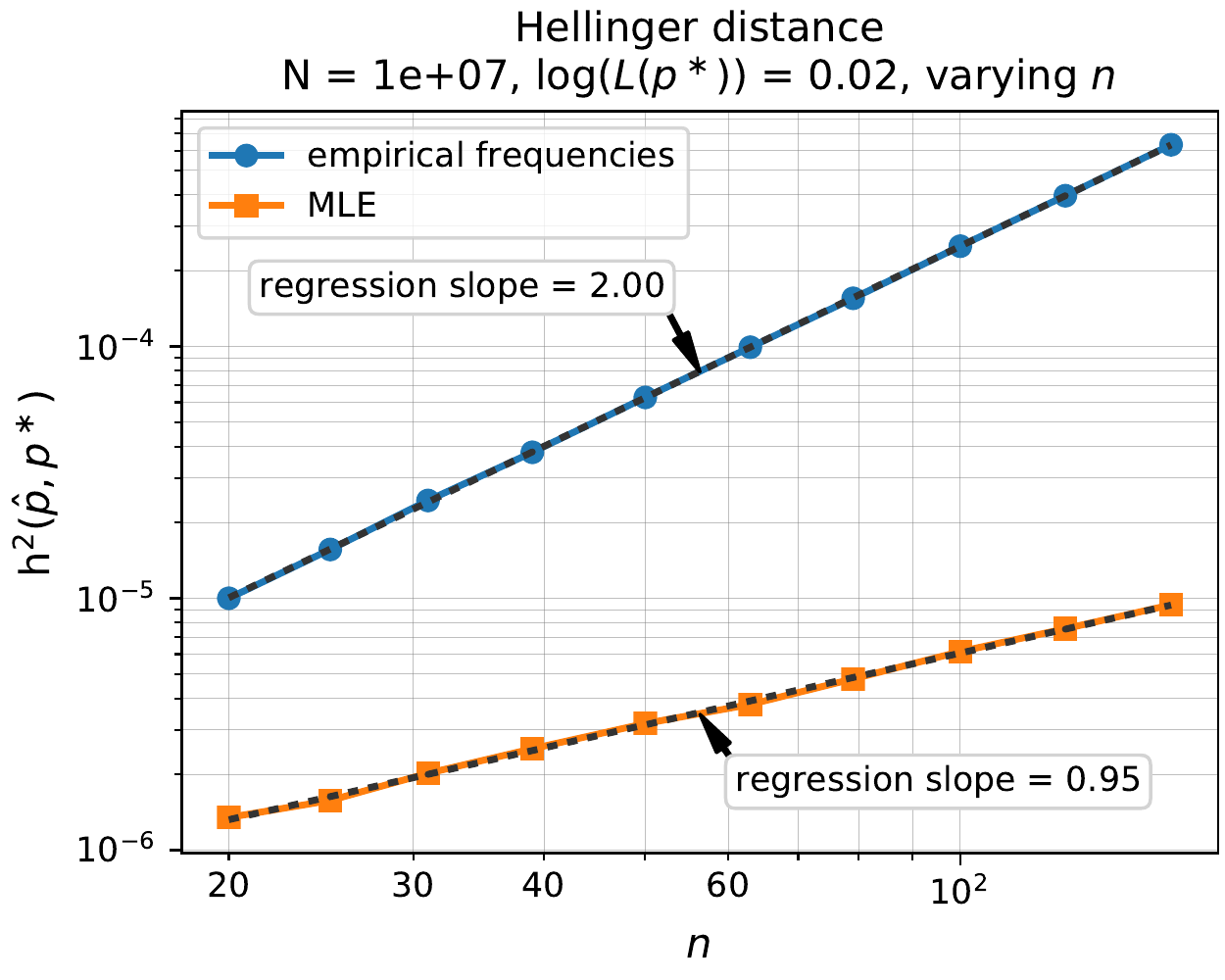}
		\caption{Varying \( n \), \( \log(L(p^\ast)) = 0.02 \)}
		\label{fig:pmf_small_V_var_n}
	\end{subfigure}
	\caption{Estimation of a density on a grid}
	\label{fig:pmf}
\end{figure}

In Figure \ref{fig:pmf}, we plot the squared Hellinger distance \( \hel^2(p^\ast, \hat p) \) for the three estimators in four setups, averaged over 20 independent replicates. More specifically, we report the results of linearly regressing the logarithms of the distances on the logarithms of the varying parameters over one or more manually selected ranges, corresponding to an estimate of the polynomial dependence on the parameter in question.

In Figures \ref{fig:pmf_large_V_var_N} and \ref{fig:pmf_small_V_var_N}, we vary the sample size \( N \) and keep \( n = 16 \) fixed for \( \log(L(p^\ast)) \in \{ 2, 0.02\} \), respectively, while in Figures \ref{fig:pmf_medium_V_var_n} and \ref{fig:pmf_small_V_var_n}, we vary the grid size \( n \) and keep \( N = \) 10,000,000 fixed for \( \log(L(p^\ast)) \in \{0.2, 0.02\} \), respectively.
We observe that all three estimators achieve an \( N^{-1} \) asymptotic rate in Figure \ref{fig:pmf_large_V_var_N}.
Moreover, the box-constrained estimator \eqref{eq:est-2} and the regular MLE \eqref{eq:mle-1} show very similar performance:
For small \( N \), the probability for zero entries in \( Y/N \) is high and thus \( Y/N \) is used instead of the box-constrained estimator as explained in the above paragraph, which explains that its performance coincides with that of the empirical frequency matrix in this regime, while the MLE performs better because the dominant factor in this regime is \( n/N \).
For an intermediate regime of \( N \), the estimation performance of the MLE and the box-constraint estimator is consistently better than that of \( Y/N \), but it is attenuated once the \( (\log(L(p^\ast)) \, n/N)^{2/3} \) rate becomes active, which can be seen in Figure \ref{fig:pmf_large_V_var_N}.
On the other hand, in Figure \ref{fig:pmf_small_V_var_N}, \( N \) is not large enough relative to \( \log(L(p^\ast)) \) to capture this regime.
Finally, for very large values of \( N \), the performance of all three estimators coincides, which for the box-constraint estimator matches the proof of the upper bound (up to logarithmic factors), see Remark~\ref{rem:ml-vs-emp}.

A similar behavior can be seen in Figures \ref{fig:pmf_medium_V_var_n} and \ref{fig:pmf_small_V_var_n}, where the performance of the frequency matrix scales with \( n^2 \), while the regular MLE scales approximately like \( n^{2/3} \) (regression coefficient of 0.72) for a larger value of \( \log(L(p^\ast)) \) and like \( n \) (regression coefficient of 0.95) when \( \log(L(p^\ast)) \) is small.
Note that the performance of the box-constrained estimator is not plotted here since it mostly coincides with that of the regular MLE.

\begin{figure}[H]
	\centering
	\begin{subfigure}{0.45\textwidth}
		\includegraphics[width=\textwidth]{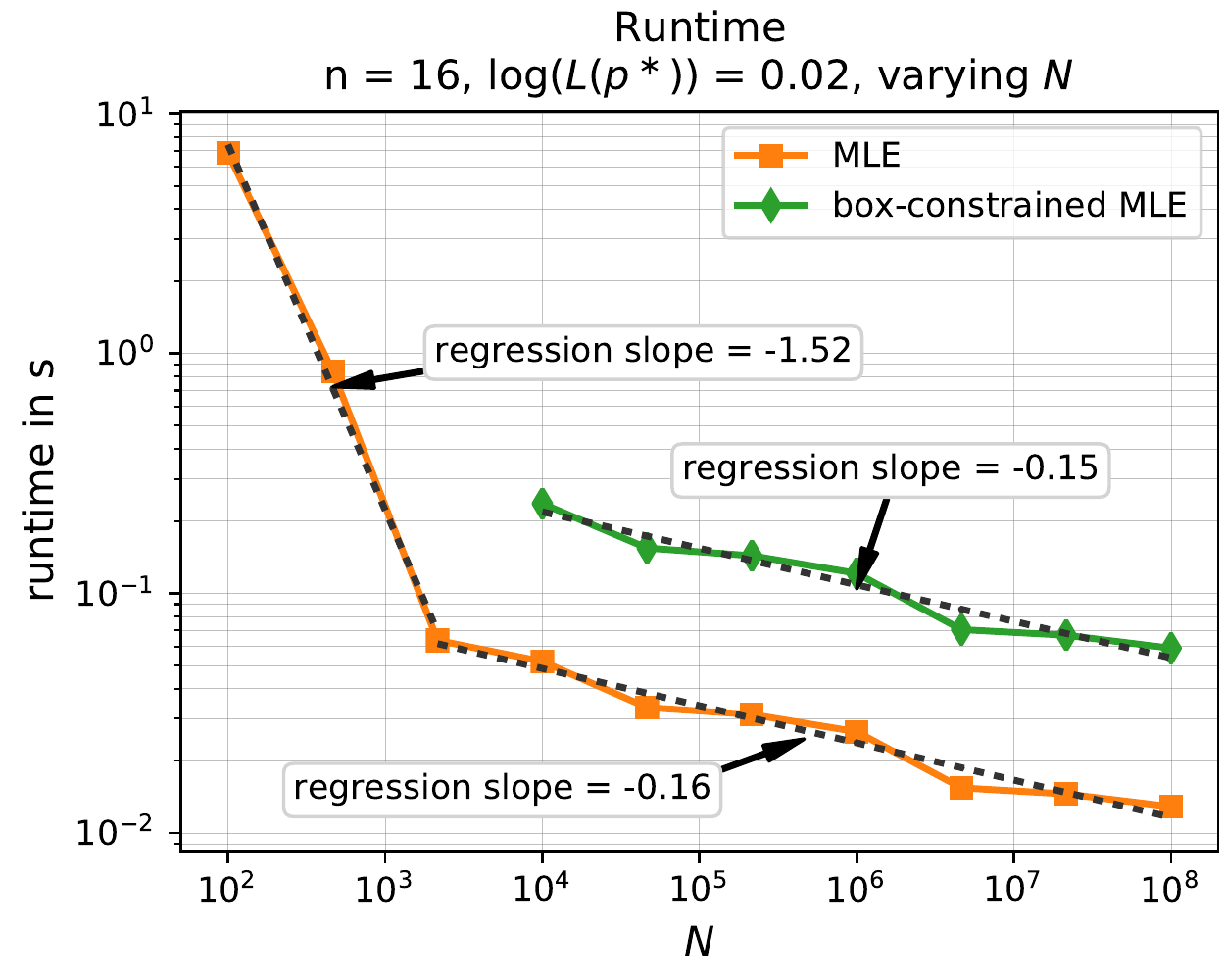}
		\caption{Varying \( N \)}
		\label{fig:pmf_var_N_runtimes}
	\end{subfigure}
	\hspace{1em}
	\begin{subfigure}{0.45\textwidth}
		\includegraphics[width=\textwidth]{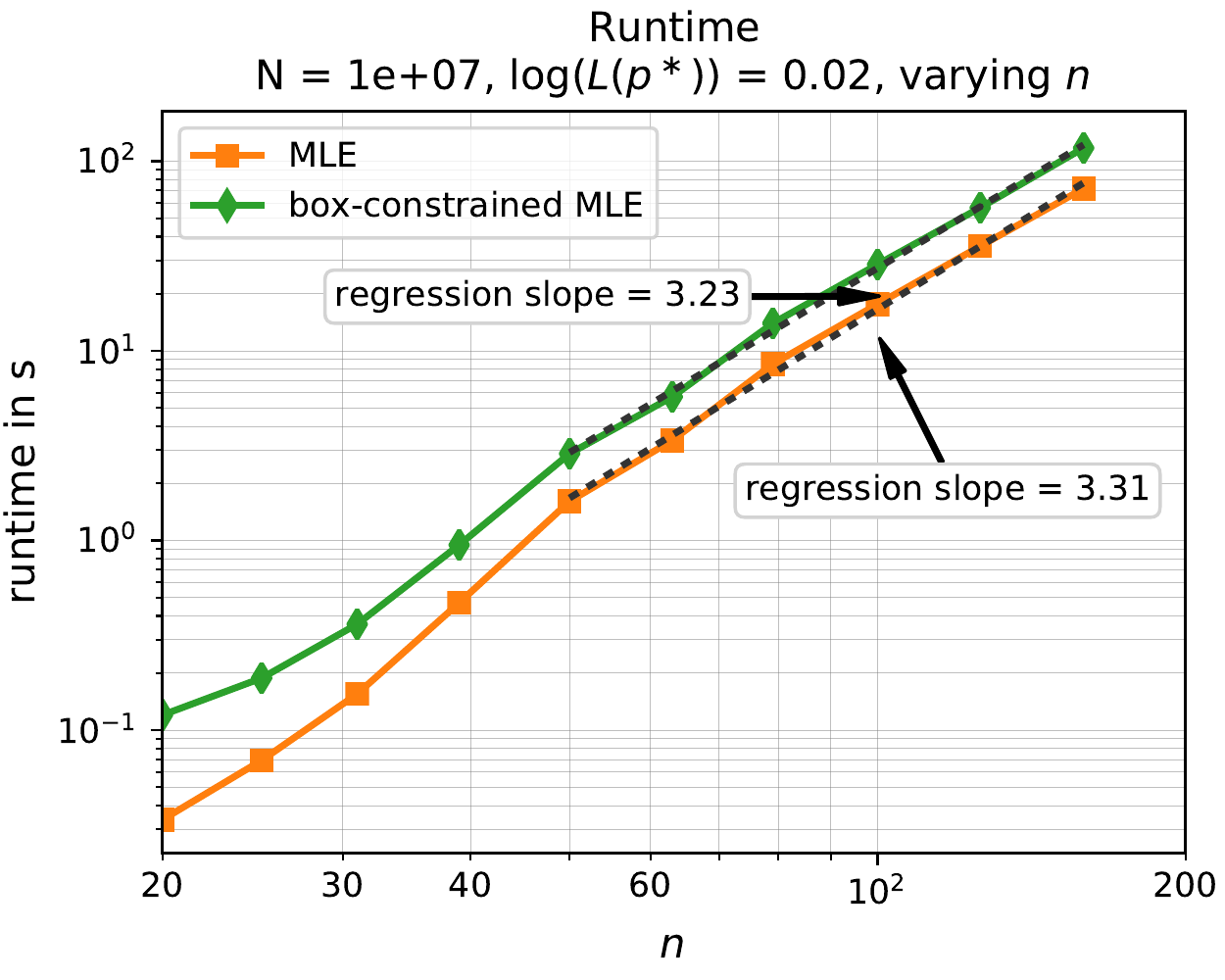}
		\caption{Varying \( n \)}
		\label{fig:pmf_var_n_runtimes}
	\end{subfigure}
	\caption{Runtimes for density on grid}
	\label{fig:pmf-runtimes}
\end{figure}

To investigate the practical performance of the proposed algorithms, in Figures \ref{fig:pmf-runtimes}, we report the runtime averaged over 20 replicates on an AMD 3400G desktop processor.
Here, as well as in the previous examples, we stopped Algorithm \ref{alg:fast-proj} when a relative change in \( \ell^2 \)-norm of less than \( 10^{-6} \) was detected.
Similarly, Algorithm~\ref{alg:prox-newton} was stopped at a relative accuracy of \( 10^{-5} \).
In Figure \ref{fig:pmf_var_N_runtimes}, we observe that the conditioning of the problem improves with larger sample sizes \( N \) and deteriorates for small values of \( N \), leading to a decay in runtime of approximately \( N^{-1.52} \) up to \( N \approx 1,000 \), followed by a milder dependence of \( N^{-0.15} \) for larger values of \( N \).
Note that the runtime for the boxed estimator is only plotted for \( N \ge 10,000 \) because of the presence of zeros in the empirical frequency matrix for smaller values of \( N \).

In Figure \ref{fig:pmf_var_n_runtimes}, we see that, as expected from a Dykstra-type algorithm, the conditioning of the problem worsens with increasing \( n \), necessitating more iterations and thus leading to an increase of runtime until convergence that is larger than the cost of one iteration, which is of order \( n^2 \).
However, it is still reasonably mild, scaling roughly like \( n^{3.3} \) for larger values of \( n \).
Overall, this highlights the practicability of the proposed algorithm for problems of medium size: Instances with \( n \approx 50 \) can be solved within four seconds, while problems of size \( n = 160 \) take under two minutes.
Moreover, adding the additional box constraint \eqref{eq:constraint} only slightly increases the runtime when \( n \) is large.

\subsection{Experiments for continuous density estimation}

We consider the multivariate Gaussian distribution \( P^\ast = N(\mu^\ast, \Sigma^\ast) \) with parameters
\begin{equation}
	\label{eq:jz}
	\mu^\ast =
	\begin{pmatrix}
		0.5\\
		0.5
	\end{pmatrix},
	\quad
	\Sigma^\ast = 
	\begin{pmatrix}
		0.2 & 0.1\\
		0.1 & 0.2
	\end{pmatrix},
\end{equation}
conditioned on the event that \( Z \in [0,1]^2 \) where \( Z \sim P^\ast \).
In other words, we consider the density
\begin{equation}
	\label{eq:ka}
	\rho^\ast(x) = \frac{1}{\int_{[0,1]^2} \tilde \rho(y) \, \ud y} \tilde \rho(x), \quad x \in [0,1]^2,
\end{equation}
where
\begin{equation}
	\label{eq:kb}
	\tilde \rho(x) = \frac{1}{2 \pi \sqrt{0.03}} \exp \left( - \frac{10}{3} \left( (x_1 - 0.5)^2 + (x_2 - 0.5)^2 - (x_1 - 0.5) (x_2 - 0.5) \right) \right).
\end{equation}
\begin{figure}[H]
	\centering
	\begin{subfigure}{0.42\textwidth}
		\includegraphics[width=\textwidth]{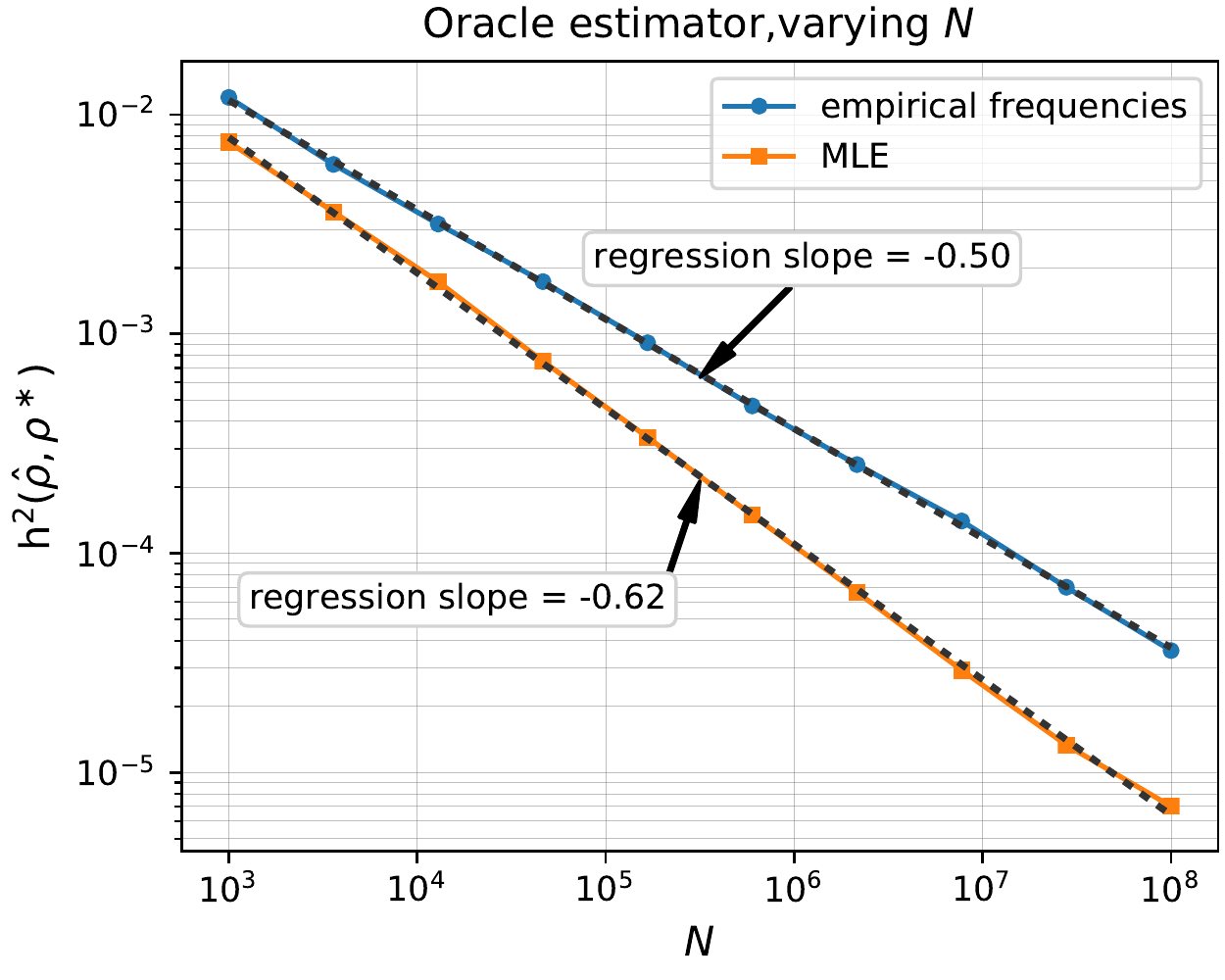}
		\caption{Comparison between empirical frequencies and \( \mtp \) estimator, oracle choice of \( n \)}
		\label{fig:cts_oracle}
	\end{subfigure}
	\hspace{1em}
	\begin{subfigure}{0.42\textwidth}
		\includegraphics[width=\textwidth]{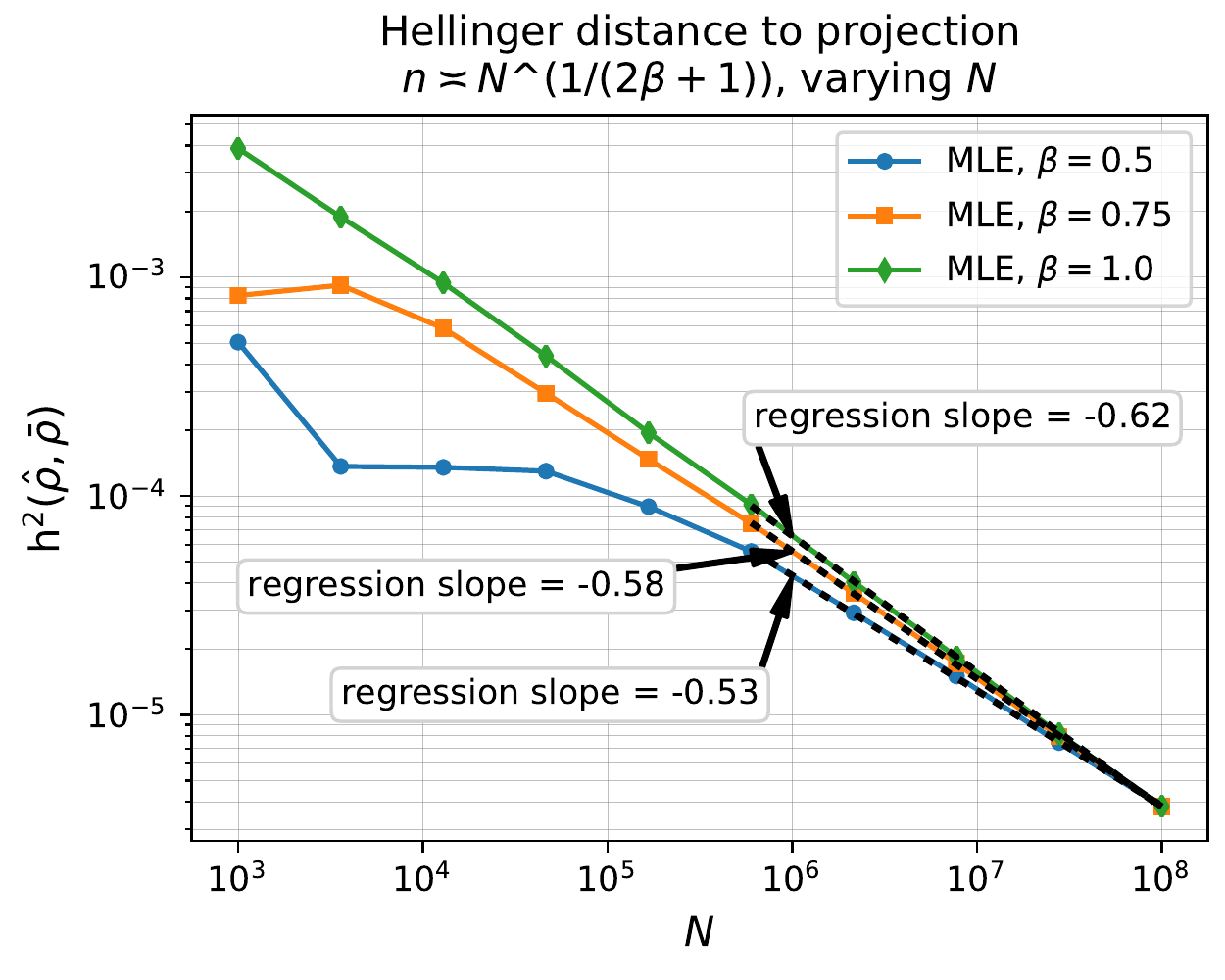}
		\caption{Variance part of the error when varying the scaling of \( n \)}
		\label{fig:cts_var_beta}
	\end{subfigure}
	\caption{Performance of continuous density estimation}
	\label{fig:cts_performance}
\end{figure}

\begin{figure}[h]
	\centering
	\begin{subfigure}{0.42\textwidth}
		\includegraphics[width=\textwidth]{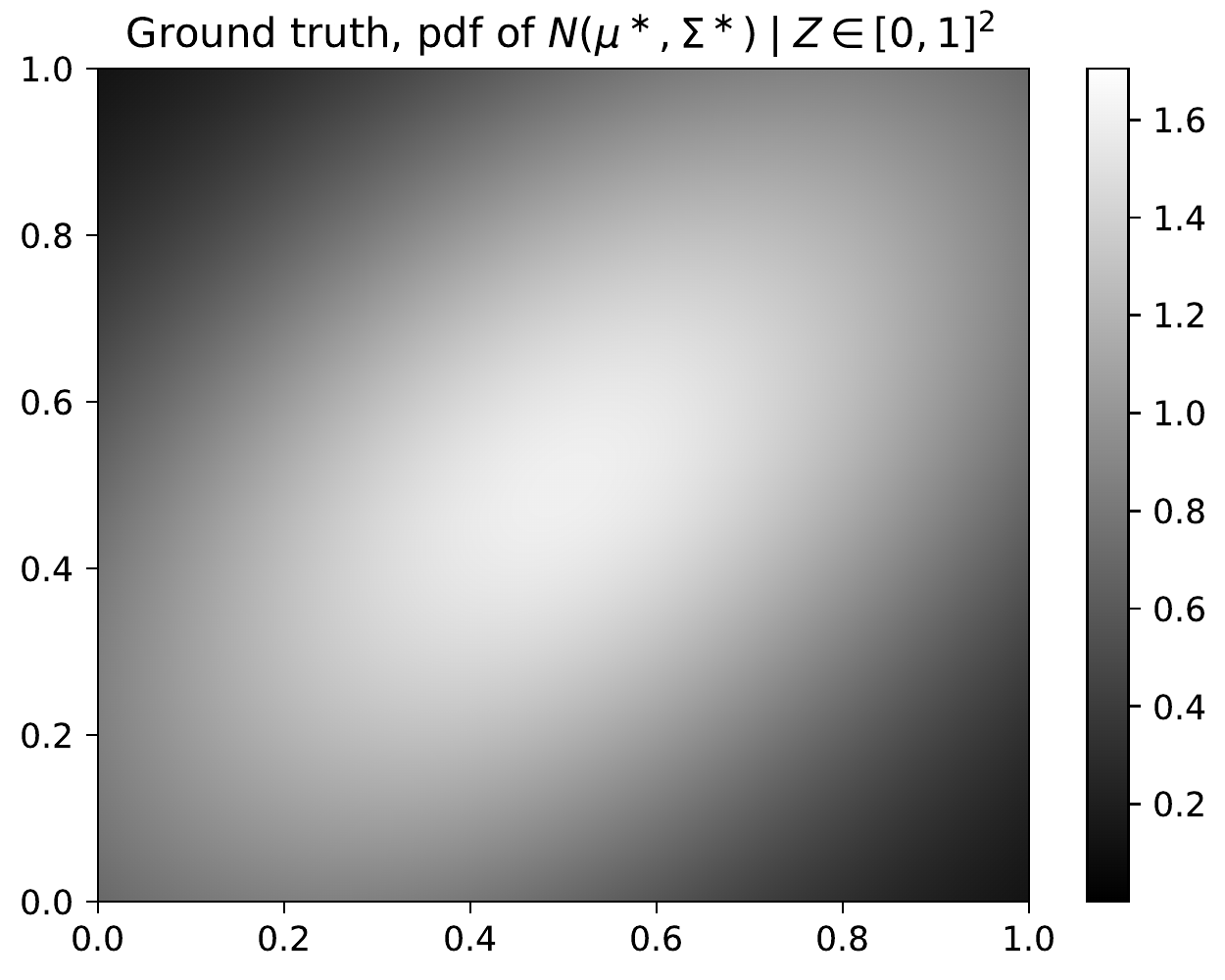}
		\caption{Ground truth \( \rho^\ast \)}
		\label{fig:cts_pdf}
	\end{subfigure}\\
	\begin{subfigure}{0.42\textwidth}
		\includegraphics[width=\textwidth]{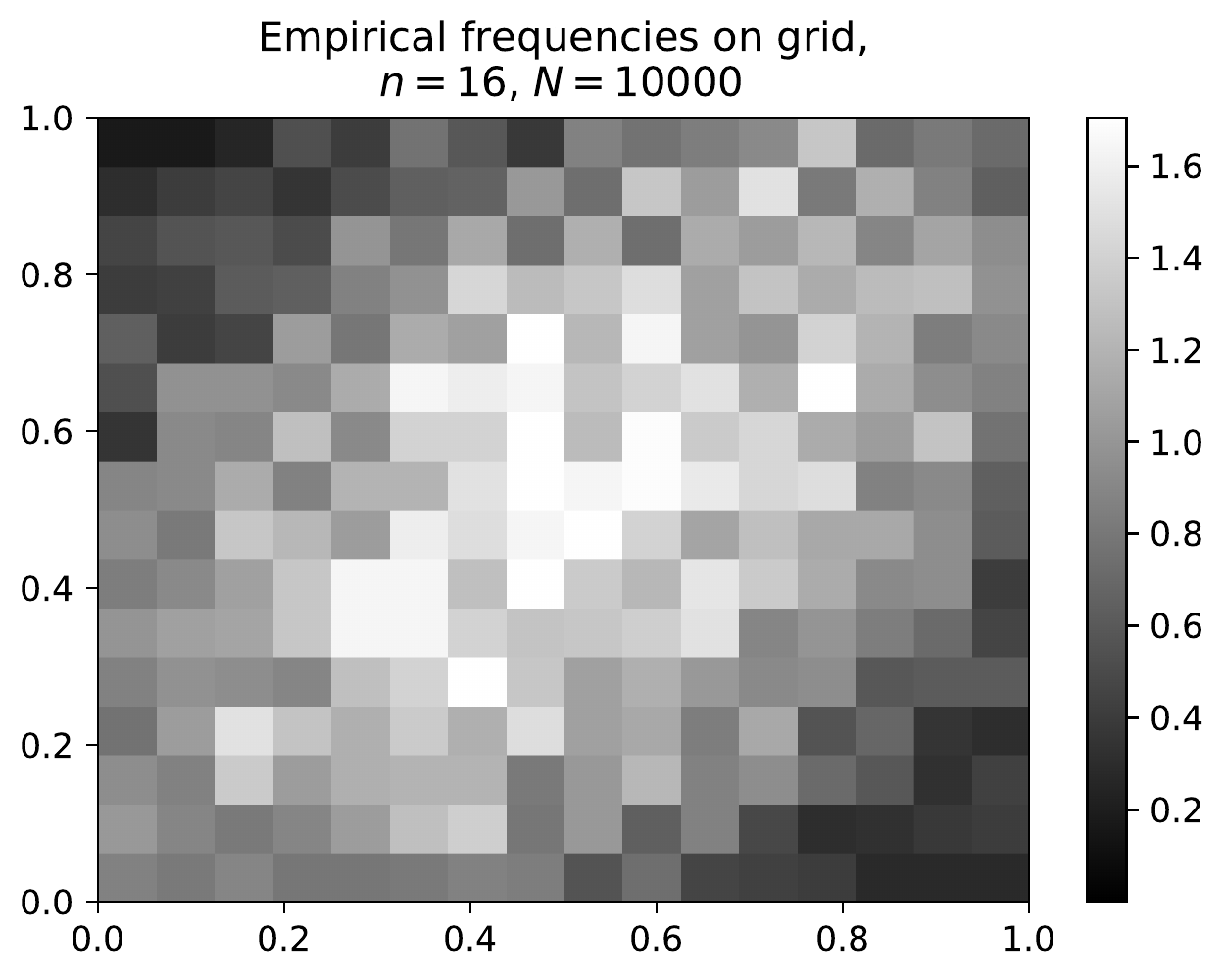}
		\caption{Approximation via frequencies \( Y/N \)}
		\label{fig:cts_freq}
	\end{subfigure}
	\hspace{1em}
	\begin{subfigure}{0.42\textwidth}
		\includegraphics[width=\textwidth]{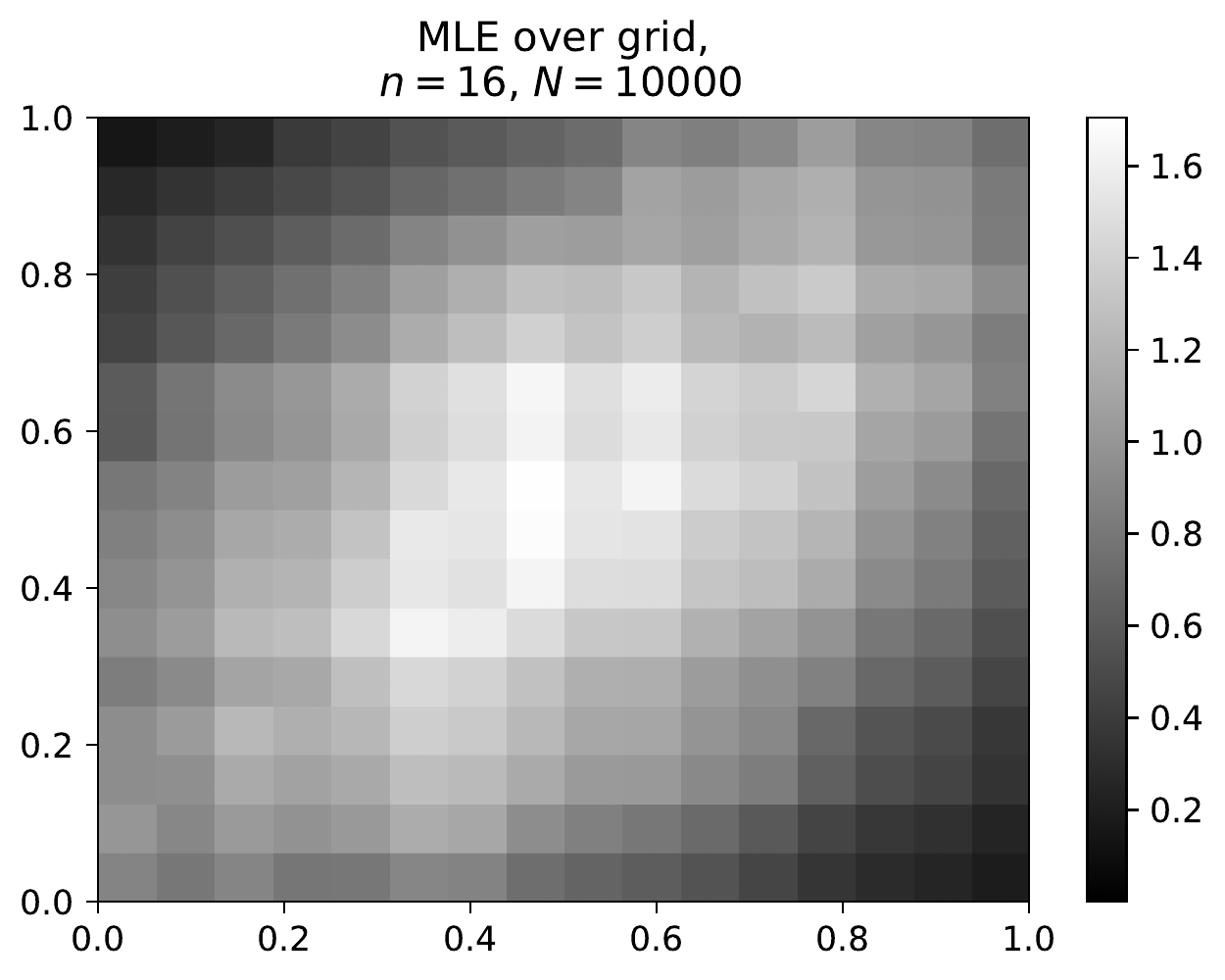}
		\caption{Approximation via MLE \( \hat \rho \)}
		\label{fig:cts_mle}
	\end{subfigure}
	\caption{Visual comparison of continuous density estimation}
	\label{fig:cts}
\end{figure}

Note that on \( [0,1]^2 \), \( \rho^\ast \in C^\infty \) and that it is \( \mtp \), which can be easily checked by computing the mixed derivative \( \partial_1 \partial_2 \) of the log-density.
Here, we use it to evaluate the performance of the gridding strategy from Section~\ref{sec:est-cts} for both an oracle choice of \( n \), that is, exploiting the knowledge of the ground truth to pick the best possible value of \( n \) from a given list, and a fixed scaling of \( n \) in the cases \( \beta \in \{0.5, 0.75, 1.0\} \).

First, in Figure~\ref{fig:cts_oracle}, with a varying number of \iid observations \( \{Z_k\}_{k=1}^N \) from \( P^\ast \), we plot the squared Hellinger distance \( \hel^2(\hat \rho, \rho^\ast) \) for the estimator in \eqref{eq:gd}, where \( n \) is picked from 10 logarithmically spaced values between \( n = 4 \) and \( n = 200 \) according to which yields the smallest Hellinger distance, and for a similarly defined estimator where \( \hat p \) is replaced by the empirical frequency matrix \( Y/N \).
We observe that the empirical frequency matrix achieves a rate of about \( N^{-1/2} \), corresponding to the rate for general H\"older functions in 2D with \( \beta = 1 \), while the \( \mtp \) MLE comes close to the predicted \( N^{-2/3} \) rate that corresponds to the \( \beta \in [1, 2) \) range (regression coefficient of \( 0.62 \)).

Second, to investigate the effect of different \( \beta \), in Figure \ref{fig:cts_var_beta}, we use a fixed scaling \( n = C N^{1/(2 \beta + 1)} \).
Note that we cannot expect to observe the rates in Theorem~\ref{thm:upper-cts} when considering the distance \( \hel^2(\hat \rho, \rho^\ast) \) in this setup since \( \rho^\ast \) is \( C^\infty \)-smooth.
Denoting by \( \bar \rho \) the piecewise constant approximation to \( \rho^\ast \) (see \eqref{eq:ge} below), this is due to the fact that the bias term \( \hel^2(\bar \rho, \rho^\ast )\) could dominate the overall error \( \hel^2(\hat \rho, \rho^\ast) \).
Hence, we only plot the Hellinger distance corresponding to the variance part \( \hel^2(\hat \rho, \bar \rho) \).
For computational reasons, \( C \) is chosen for each \( \beta \) so that for \( N = 10^8 \), we have \( n = 200 \).
Due to the similarities between the regular MLE and its box-constrained version observed in the previous section, all calculations were performed using the regular MLE, \eqref{eq:mle-1}, resulting in slightly faster computations.

Performing linear regression on the doubly logarithmic plot for large values of \( N \), we observe rates of \( 0.53 \), \( 0.58 \), and \( 0.62 \) for \( \beta = 0.5, 0.75, 1.0 \), respectively.
These are close to \( 0.5 \), \( 0.6 \), and \( 2/3 \), respectively, as predicted by Theorem~\ref{thm:upper-cts}.
Additionally, we present heat maps of the density \( \rho^\ast \) (Figure~\ref{fig:cts_pdf}), as well as an approximation via the frequncy matrix \( Y/N \) (Figure~\ref{fig:cts_freq}) and the MLE (Figure~\ref{fig:cts_mle}) for \( N = \) 10,000 and \( n = 16 \).
The visual smoothing effect of the MLE is quite obvious in this case.

\section{Proofs} 
\label{sec:proofs}

The proofs of our results are provided in this section. 
We first prove the upper bounds Theorems~\ref{thm:upper} and~\ref{thm:upper-cts} in the discrete and smooth cases respectively, and then the lower bounds Theorems~\ref{thm:lower} and~\ref{thm:lower-cts}. 
In the proofs, we make use of the well-known relation~\cite[Lemma~7.23]{Mas07} that for PMFs $p$ and $q$ on $\cX = [\dimone] \times [\dimtwo]$ or $[0, 1]^2$, 
\begin{align}
2 \hel^2(p, q) \le \KL(p, q) \le 2 \Big( 2 + \log \big( \max_{x \in \cX} \frac{p(x)}{q(x)} \big) \Big) \hel^2(p, q) . \label{eq:hel-kl}
\end{align}

\subsection{Proof of Theorem \ref{thm:upper}}
\label{sec:pf-ub}


\subsubsection{Setup of the proof: quadratic approximation}

Let us define \( \eps \defn \frac{Y}{\samples} - p^\ast \). 
Denote by $\cA_1$ the event of probability $1 - 2 \delta$ that the bounds in Lemma~\ref{lem:obs-bd} hold. 
On this event, $\groundtruth$ lies in the cube $\cC(Y)$ defined in \eqref{eq:constraint}, so we have 
$$ 
\frac 1{\samples} \langle Y, \tilde \theta \rangle - \sum_{i, j} e^{\tilde \theta_{i, j}} \ge \frac 1{\samples} \langle Y, \groundtruth \rangle - 1 , 
$$
which is equivalent to
\begin{align} \label{eq:basic}
\langle p^\ast, \groundtruth - \tilde \theta \rangle + \sum_{i, j} e^{\tilde \theta_{i, j}} - 1 \le \langle \eps, \tilde \theta - \groundtruth \rangle .
\end{align}
In addition, the definition of $\cC(Y)$ yields that $|\tilde \theta_{i, j} - \groundtruth_{i, j}| \le \log 2 - \log \frac 23 < 1.1$ for all $i, j$.

By a quadratic Taylor approximation of $e^y$, it holds  for $x \le 0$ and  $|y - x| \le 1.1$ that
\begin{align}
e^{x} + e^{x} (y - x) +  e^{x} (y - x)^2 / 4 \le e^y \le e^{x} + e^{x} (y - x) +  2 \, e^{x} (y - x)^2 . \label{eq:taylor-exp}
\end{align}
Applying this approximation to the exponential terms of the left-hand side of \eqref{eq:basic}, we obtain
\begin{align*}
	\leadeq{\langle p^\ast, \groundtruth - \tilde \theta \rangle + \sum_{i, j} e^{\groundtruth_{i, j}} + \sum_{i, j} e^{\groundtruth_{i, j}} ( \tilde \theta_{i, j} - \groundtruth_{i, j} ) + \frac 14 \sum_{i, j} e^{\groundtruth_{i, j}} ( \tilde \theta_{i, j} - \groundtruth_{i, j} )^2 - 1 }\\
&\le \langle p^\ast, \groundtruth - \tilde \theta \rangle + \sum_{i, j} e^{\tilde \theta_{i, j}} - 1 \\
&\le \langle p^\ast, \groundtruth - \tilde \theta \rangle + \sum_{i, j} e^{\groundtruth_{i, j}} + \sum_{i, j} e^{\groundtruth_{i, j}} ( \tilde \theta_{i, j} - \groundtruth_{i, j} ) + 2 \sum_{i, j} e^{\groundtruth_{i, j}} ( \tilde \theta_{i, j} - \groundtruth_{i, j} )^2 - 1 ,
\end{align*}
or equivalently, 
\begin{align} 
\frac 14 \sum_{i, j} p^\ast_{i, j} ( \tilde \theta_{i, j} - \groundtruth_{i, j} )^2 
\le \langle p^\ast, \groundtruth - \tilde \theta \rangle + \sum_{i, j} e^{\tilde \theta_{i, j}} - 1 
\le 2 \sum_{i, j} p^\ast_{i, j} ( \tilde \theta_{i, j} - \groundtruth_{i, j} )^2 . \label{eq:qua-approx}
\end{align}
The rest of the proof hinges on this quadratic approximation.
Particularly, it follows from~\eqref{eq:basic} and~\eqref{eq:qua-approx} that
\begin{align}
\frac 14 \sum_{i, j} p^\ast_{i, j} ( \tilde \theta_{i, j} - \groundtruth_{i, j} )^2
\le \langle \eps, \tilde \theta - \groundtruth \rangle . \label{eq:fg}
\end{align}
The main task in the sequel is to bound the right-hand side of~\eqref{eq:fg}. 
The strategy builds upon a spectral decomposition technique from the paper~\cite{HutMaoRigRob19} on Monge matrix estimation. 
%
%

\subsubsection{Spectral decomposition of the difference operator}

Recall the difference operator $D$ defined in~\eqref{eq:diff-op} and $\Dtwo$ defined analogously for dimension $\dimtwo$. 
Throughout the proof, whenever we introduce notation in dimension $\dimone$,  the analogous one in dimension $\dimtwo$ is denoted by the same symbol with a tilde. 
We will decompose the noise $\eps$ in \eqref{eq:fg} according to a spectral decomposition of $D$, so let us recall some basic facts about the matrix $D$. 

Denote the singular value decomposition of \( \Done \) by
\begin{equation}
\label{eq:dc}
\Done =  U \Sigma W^\top, \quad U \in \R^{(\dimone - 1) \times (\dimone - 1)}, \quad \Sigma \in \R^{(\dimone - 1) \times \dimone}, \quad W \in \R^{\dimone \times \dimone}
\end{equation}
where we order the non-zero singular values of \( D \) in \( \Sigma \) in ascending magnitude, so that the last column of \( W \) spans the null-space of \( D \).
In addition, we write $W = \begin{bmatrix}
w_1 & \cdots & w_{\dimone}
\end{bmatrix}.$
Let us define a set of double indices
\begin{equation}
\label{eq:eu}
J \defn \big\{ (l, r) \in [\dimone] \times [\dimtwo] : lr \le k \big\} \cup \big( [\dimone] \times \{\dimtwo\} \big) \cup \big( \{\dimone\} \times [\dimtwo]  \big) ,
\end{equation}
and set \( J^c = ([\dimone] \times [\dimtwo]) \setminus J \).

We introduce a projection operator \( \Pi : \R^{\dimone \times \dimtwo} \to \R^{\dimone \times \dimtwo} \), defined as the projection onto the linear span of \( \{ w_i \tilde w_j^\top : (i, j) \in J^c \}  \) that is orthogonal with respect to the inner product
\begin{equation}
	\label{eq:a}
	\langle A , B \rangle_{1/p^\ast} \defn \sum_{i, j} \frac{1}{p^\ast_{i, j}} A_{i, j} B_{i, j}.
\end{equation}
In particular, there exists an orthonormal basis \( \{ \V^{(l, r)} \in \R^{\dimone \times \dimtwo} : (l, r) \in [\dimone] \times [\dimtwo] \} \) of $\R^{\dimone \times \dimtwo}$ with respect to the inner product \( \langle . , . \rangle_{1/p^\ast} \) such that
\begin{equation}
	\label{eq:b}
	\Pi(A) = \sum_{(l, r) \in J^c} \V^{(l, r)} \langle \V^{(l, r)} , A \rangle_{1/p^\ast}
	\quad \text{ and } \quad
	(I - \Pi)(A) = \sum_{(l, r) \in J} \V^{(l, r)} \langle \V^{(l, r)} , A \rangle_{1/p^\ast}.
\end{equation}

To characterize these projection operators further, we introduce the following notation. 
Let $\odot$ and $\oslash$ denote entrywise multiplication and division respectively between matrices. With a slight abuse of notation, we use $\sqrt{p^\ast}$ and $1/p^\ast$ to denote the entrywise square root and the entrywise inverse of $p^\ast$ respectively.
Let \( \Lambda  \) and \( \Lambda^{-1} \) be the scaling operators from $\R^{\dimone \times \dimtwo}$ to itself, defined by 
\begin{alignat}{2}
	\label{eq:c}
	\Lambda(A) = {} & A \otimes p^\ast, \quad &A \in {} & \R^{\dimone \times \dimtwo}, \\
	\Lambda^{-1}(A) = {} & A \oslash p^\ast, \quad &A \in {} & \R^{\dimone \times \dimtwo},
\end{alignat}
respectively. 
Let $\mathcal{L}$ be the linear operator from $\R^{|J^c|}$ (indexed by $(l, r) \in J^c$) to $\R^{\dimone \times \dimtwo}$, defined by
\begin{alignat}{2}
	\label{eq:c2}
	\mathcal{L} (B) = {} & \sum_{(l, r) \in J^c} B_{l, r} w_l \tilde w_r^\top , \quad & B \in {} & \R^{|J^c|},
\end{alignat}
and denote by \( \mathcal{L}^\top \) the transpose of $\mathcal{L}$ with respect to the standard inner products in the corresponding spaces. 
In other words, we have
\begin{equation}
	\label{eq:d}
	(\mathcal{L}^\top (A))_{l, r} = \langle w_l \tilde w_r^\top , A \rangle, \quad (l, r) \in J^c, \quad A \in \R^{\dimone \times \dimtwo}.
\end{equation}


The linear operators $\Pi$ and $\Lambda$ can be viewed as $\dimone \dimtwo \times \dimone \dimtwo$ matrices, while the linear operator $\mathcal{L}$ can be seen as an $\dimone \dimtwo \times |J^c|$ matrix. 
Moreover, 
we have the following lemma whose proof is deferred to Section~\ref{sec:gram-eig}.

\begin{lemma} \label{lem:gram-eig}
The smallest eigenvalue of the operator \( \mathcal{L}^\top \Lambda^{-1} \mathcal{L} \) satisfies that 
	\begin{equation}
		\label{eq:y}
		\lambda_{\min}(\mathcal{L}^\top \Lambda^{-1} \mathcal{L}) \ge \frac{1}{\pmax}.
	\end{equation}
Moreover, $\Pi$ can be written as 
\begin{equation}
	\label{eq:et}
	\Pi = \mathcal{L} (\mathcal{L}^\top \Lambda^{-1} \mathcal{L})^{-1} \mathcal{L}^\top \Lambda^{-1}.
\end{equation}
\end{lemma}



\medskip

To control \( \langle \epsilon , \premle - \groundtruth \rangle \) on the right-hand side of \eqref{eq:fg}, we decompose it as 
\begin{align}
\langle \varepsilon , \premle - \groundtruth \rangle
=  \langle (I - \Pi) (\eps  ) ,  \premle - \groundtruth   \rangle
+  \langle \Pi  ( \eps  ) ,  \premle - \groundtruth   \rangle. \label{eq:es}
\end{align}
Before proceeding to bound these two terms separately, we state two lemmas whose proofs are deferred to Sections~\ref{sec:null} and~\ref{sec:var-bd}, respectively. 

\begin{lemma} \label{lem:null}
	The image of the projection \( \Pi \) is included in the image of the map $A \mapsto \Done^\top A \Dtwo$.
%
\end{lemma}

\begin{lemma} \label{lem:var-bd}
For any $(i, j) \in [\dimone] \times [\dimtwo]$, we have that
$$
\sum_{(l, r) \in J^c} \Sigma_{l, l}^{-2} \tilde \Sigma_{r, r}^{-2} U_{i, l}^2 \tilde U_{j, r}^2
\lesssim \frac{\dimone \dimtwo}{k} \log (\dimtwo)  .
$$
\end{lemma}

\subsubsection{Bounding the first term in~\eqref{eq:es}}

By H\"older's inequality, 
\begin{align}
	\langle (I - \Pi)(\eps) ,  \premle - \groundtruth \rangle
	= \big\langle (I - \Pi)(\eps) \oslash \sqrt{p^\ast},  (\premle - \groundtruth) \odot  \sqrt{p^\ast} \, \big\rangle 
	\le \| (I - \Pi)(\eps) \|_{1/p^\ast} \|  \premle - \groundtruth \|_{p^\ast} .  \label{eq:ev}
\end{align}

Now we focus on the quantity $\| (I - \Pi)(\eps) \|_{1/p^\ast} $. By the definition of $\Pi$ and the orthogonality condition that $\langle \V^{(l, r)}, \V^{(l', r')} \rangle = 0$ for any $(l, r) \ne (l', r')$, we obtain that 
\begin{align}
	\| (I - \Pi)(\eps) \|_{1/p^\ast} ^2
&= \Big\| \sum_{(l, r) \in J} \V^{(l, r)} \langle \V^{(l, r)} , \epsilon \rangle_{1/p^\ast} \Big\|_{1/p^\ast}^2 \\
&= \sum_{(l, r) \in J} \langle \V^{(l, r)} , \epsilon \rangle_{1/p^\ast}^2
\le |J| \max_{(l, r) \in J} ( \langle \V^{(l, r)} , \epsilon \rangle_{1/p^\ast} )^2 . \label{eq:sq-bd}
\end{align}
Note that we can write
\begin{align}
\label{eq:ee}
\langle \V^{(l, r)} , \epsilon \rangle_{1/p^\ast} = \big\langle \frac{1}{\samples} \V^{(l, r)} \oslash p^\ast , \samples \eps \big\rangle .
\end{align}
Recall that $Y$ has the multinomial distribution $\Multi(N, p^*)$, and $\samples \eps = Y - \samples p^\ast$ is the deviation of $Y$ from its mean. 
Therefore, Lemma \ref{lem:multi-max} yields that on an event $\cA_2$ of probability $1 - \delta$,
\begin{align}
	\leadeq{\max_{(l, r) \in J} | \langle \V^{(l, r)}, \epsilon \rangle_{1/p^\ast} |}\\
	\lesssim {} & \Big( \max_{(l, r) \in J} \Big\| \frac{1}{\samples} \V^{(l, r)} \oslash p^\ast \Big\|_{\samples p^\ast} \Big) \sqrt{\log (|J| / \delta)} + \Big( \max_{(l, r) \in J} \Big\| \frac{1}{\samples} \V^{(l, r)} \oslash p^\ast \Big\|_{\infty} \Big) \log (|J| / \delta) . \label{eq:bern-bd}
\end{align}

To bound the two norms above, we note that by orthogonality of the \( \V^{(l, r)} \) with respect to \( \langle . , . \rangle_{1 / p^\ast} \),
\begin{align}
\label{eq:ef}
\Big\| \frac{1}{\samples} \V^{(l, r)} \oslash p^\ast \Big\|_{\samples p^\ast}^2 
= \Big\| \frac{1}{\sqrt{\samples}} \V^{(l, r)} \Big\|_{1/p^\ast}^2 
= \frac{1}{N}.
\end{align}
In addition, it holds that
\begin{align}
\Big\| \frac{1}{\samples} \V^{(l, r)} \Big\|_\infty
\le {} & \Big\| \frac{1}{\samples} \V^{(l, r)} \Big\|_F
= \Big\| \frac{1}{\samples \sqrt{\samples}} \V^{(l, r)} \oslash \sqrt{p^\ast} \Big\|_{\samples p^\ast}\\
\le {} & \frac{1}{\sqrt{\samples \pmin}} \Big\| \frac{1}{\samples} \V^{(l, r)} \Big\|_{\samples p^\ast}
= \frac{1}{N \sqrt{\pmin}}
\le \frac{1}{\sqrt{N \log(\dimone/\delta)}},
\label{eq:fm}
\end{align}
where we used \eqref{eq:ef} and that \( N \ge 12 \log(\dimone \dimtwo/\delta)/\pmin \) by assumption.
Further, we can control the cardinality of \( J \) by
\begin{equation}
\label{eq:ex}
\dimone \le |J| \le \dimone \dimtwo, 
\quad \text{and} \quad 
|J| = \sum_{(l, r) \in J}  1
\le \dimone + \dimtwo - 1 + \sum_{r=1}^{\dimtwo} \lfloor k/r \rfloor
\le 2 \dimone + k \log(\dimtwo).
\end{equation}

Combining~\eqref{eq:ev},~\eqref{eq:sq-bd},~\eqref{eq:bern-bd},~\eqref{eq:ef},~\eqref{eq:fm} and~\eqref{eq:ex}, we see that on the event \( \cA_2 \),
\begin{align}
	\langle (I - \Pi)(\eps) ,  \premle - \groundtruth \rangle
&\lesssim \|  \premle - \groundtruth \|_{p^\ast} \sqrt{\dimone + k \log(\dimtwo) } \Big( \sqrt{\frac{\log (\dimone / \delta)}{\samples}} + \frac{ \log(\dimone / \delta) }{ \sqrt{\samples \log(\dimone/\delta)}} \Big) \\
&\lesssim \|  \premle - \groundtruth \|_{p^\ast}  \sqrt{\frac{ \dimone \log(\dimone / \delta) + k \log(\dimone / \delta) \log(\dimtwo) }{\samples}} . \label{eq:proj-term}
\end{align}

\subsubsection{Bounding the second term in \eqref{eq:es}}

By Lemma~\ref{lem:null}, the image of $\Pi$ is included in the image of the adjoint map $A \mapsto \Done^\top A \Dtwo$.
Since \( A \mapsto \Done^\top (\Done^\top)^\dag A \Dtwo^\dag \Dtwo \) is the orthogonal projection onto this image, we thus have
\begin{align}
\label{eq:ey}
\langle \Pi ( \eps ) ,  \premle - \groundtruth   \rangle 
= {} &  \langle \Done^\top (\Done^\top)^\dag \Pi  ( \eps   )  \Dtwo^\dag \Dtwo,  \premle - \groundtruth   \rangle  \\
= {} &  \langle (\Done^\dag)^\top \Pi  ( \eps  )  \Dtwo^\dag, \Done ( \premle - \groundtruth) \Dtwo^\top  \rangle \\
\le {} & \big \| (\Done^\dag)^\top \Pi ( \eps   )  \Dtwo^\dag  \big\|_\infty  \,
\big\| \Done ( \premle - \groundtruth)   \Dtwo^\top \big\|_1 \label{eq:holder}
\end{align}
by H\"older's inequality.

We first consider the term $\big \| (\Done^\dag)^\top \Pi ( \eps   )  \Dtwo^\dag  \big\|_\infty$.
By the formula \eqref{eq:et} for $\Pi$ and the singular value decomposition of \( \Done \), it holds that
\begin{align}
\label{eq:eh}
(\Done^\dag)^\top \Pi ( \eps   )  \Dtwo^\dag
= {} & \sum_{(l, r) \in J^c} \Sigma^{-1}_{l, l} \Sigma^{-1}_{r, r} U_{\cdot,l} \tilde U_{\cdot,r}^\top \langle w_l \tilde w_r^\top , \Pi (\epsilon) \rangle\\
= {} & \sum_{(l, r) \in J^c} \Sigma^{-1}_{l, l} \Sigma^{-1}_{r, r} U_{\cdot,l} \tilde U_{\cdot,r}^\top \langle w_l \tilde w_r^\top , \mathcal{L} (\mathcal{L}^\top \Lambda^{-1} \mathcal{L})^{-1} \mathcal{L}^\top \Lambda^{-1} (\epsilon) \rangle\\
= {} & \sum_{(l, r) \in J^c} \Sigma^{-1}_{l, l} \Sigma^{-1}_{r, r} U_{\cdot,l} \tilde U_{\cdot,r}^\top \langle \Lambda^{-1} \mathcal{L} (\mathcal{L}^\top \Lambda^{-1} \mathcal{L})^{-1} \mathcal{L}^\top (w_l \tilde w_r^\top) , \epsilon \rangle .
\end{align}
By \eqref{eq:d} and the orthogonality of the vectors \( \{ w_l \}_{l \in [\dimone]} \) and \( \{ \tilde w_r \}_{r \in [\dimtwo]} \), we have that \( \mathcal{L}^\top (w_l \tilde w_r^\top) = e^{(l, r)} \) if \( (l, r) \in J^c \) and zero otherwise, where \( e^{(l, r)} \) denotes the coordinate vector in \( \R^{|J^c|} \) with a one in the \( (l, r) \)th component and zero in all others.
Hence, if we define $a^{(i,j)} \in \R^{|J^c|}$ for $(i, j) \in [\dimone - 1] \times [\dimtwo - 1]$ by
\begin{equation}
	\label{eq:x}
	(a^{(i,j)})_{l, r} \defn \Sigma_{l, l}^{-1} \Sigma_{r, r}^{-1} U_{i, l} \tilde U_{j, r} ,
\end{equation}
then for \( (i, j) \in [\dimone-1] \times [\dimtwo-1] \), we obtain
\begin{align}
	\label{eq:u}
\left( (\Done^\dag)^\top \Pi ( \eps   )  \Dtwo^\dag  \right)_{i, j}
= {} & \sum_{(l, r) \in J^c} (a^{(i, j)})_{l, r} \langle \Lambda^{-1} \mathcal{L} (\mathcal{L}^\top \Lambda^{-1} \mathcal{L})^{-1} e^{(l, r)} , \epsilon \rangle\\
= {} & \langle \underbrace{\frac{1}{N} \Lambda^{-1} \mathcal{L} (\mathcal{L}^\top \Lambda^{-1} \mathcal{L})^{-1} a^{(i, j)}}_{=: B^{(i, j)} }, N \epsilon \rangle . 
\end{align}
As before, Lemma~\ref{lem:multi-max} yields that on an event $\cA_3$ of probability $1 - \delta$,
\begin{align}
\Big\| (\Done^\dag)^\top \Pi ( \eps   )  \Dtwo^\dag \Big\|_\infty 
\lesssim \Big( \max_{i, j} \big\| B^{(i, j)} \big\|_{\samples p^\ast} \Big) \sqrt{\log (\dimone / \delta)} + \Big( \max_{i, j} \big\| B^{(i, j)} \big\|_{\infty} \Big) \log (\dimone / \delta) .
\end{align}

We proceed to bound \( \| B^{(i, j)} \|_{\samples p^\ast} \) and \( \| B^{(i, j)} \|_\infty \).
First,
\begin{align}
	\label{eq:w}
	\Big\| B^{(i, j)} \Big\|_{N p^\ast}^2
	= {} & \frac{1}{N} \Big\| \Lambda^{-1} \mathcal{L} (\mathcal{L}^\top \Lambda^{-1} \mathcal{L})^{-1} a^{(i, j)} \Big\|_{p^\ast}^2\\
	= {} & \frac{1}{N} (a^{(i, j)})^\top (\mathcal{L}^\top \Lambda^{-1} \mathcal{L})^{-1} \mathcal{L}^\top \Lambda^{-1} \Lambda \Lambda^{-1} \mathcal{L} (\mathcal{L}^\top \Lambda^{-1} \mathcal{L})^{-1} a^{(i, j)}\\
	= {} & \frac{1}{N} \Big\| (\mathcal{L}^\top \Lambda^{-1} \mathcal{L})^{-1/2} a^{(i, j)} \Big\|_2^2\\
	\le {} & \frac{\pmax}{N} \, \| a^{(i, j)} \|_2^2
\end{align}
by Lemma \ref{lem:gram-eig}.
Then, by definition,
\begin{align}
\big\| a^{(i, j)} \big\|_{2}^2
= \sum_{(l, r) \in J^c} \Sigma_{l, l}^{-2} \tilde \Sigma_{r, r}^{-2} U_{i, l}^2 \tilde U_{j, r}^2 
\lesssim  \frac{\dimone \dimtwo}{k} \log (\dimtwo)  ,
\end{align}
where the inequality is due to Lemma~\ref{lem:var-bd}. 
As for \( \| B^{(i, j)}  \|_\infty \), we proceed as in \eqref{eq:fm} to obtain
\begin{align}
\label{eq:ek}
\|  B^{(i, j)} \|_\infty
\le \frac{1}{\sqrt{N \pmin}} \| B^{(i, j)} \|_{\samples p^\ast}
\lesssim \sqrt{\frac{\pmax \dimone \dimtwo  \log(\dimtwo)}{N k \log(\dimone / \delta) }},
\end{align}
where we used again that by assumption, \( N \ge 12 \log(\dimone \dimtwo/\delta)/\pmin \).
Combining the above bounds yields that on the event $\cA_3$, 
\begin{align}
\big\| (\Done^\dag)^\top \Pi ( \eps   )  \Dtwo^\dag  \big\|_\infty 
\lesssim \sqrt{\frac{ \pmax \dimone \dimtwo \log( \dimone / \delta) \log (\dimtwo) }{\samples k} }. \label{eq:li-bd}
\end{align}

Next, we turn to the quantity $\big\| \Done ( \premle - \groundtruth)   \Dtwo^\top \big\|_1$. Note that for any $\theta$ such that $\Done \theta \Dtwo^\top \ge 0$, it holds
\begin{align}
\| \Done \theta \Dtwo^\top \|_1 
&= \sum_{i=1}^{\dimone-1} \sum_{j=1}^{\dimtwo-1} (\Done \theta \Dtwo^\top)_{i, j} \\
&= \sum_{i=1}^{\dimone-1} \sum_{j=1}^{\dimtwo-1} ( \theta_{i, j} + \theta_{i+1, j+1} - \theta_{i+1, j} - \theta_{i, j+1} )
= \theta_{1,1} + \theta_{\dimone, \dimtwo} - \theta_{\dimone, 1} - \theta_{1, \dimtwo} . \label{eq:l1-sum}
\end{align}
Therefore, we obtain
\begin{align}
\| \Done \groundtruth \Dtwo^\top \|_1 
= \groundtruth_{1,1} + \groundtruth_{\dimone, \dimtwo} - \groundtruth_{\dimone, 1} - \groundtruth_{1, \dimtwo}
= \log \frac{ p^\ast_{1,1} p^\ast_{\dimone, \dimtwo} }{ p^\ast_{\dimone, 1} p^\ast_{1, \dimtwo} } = \log \big( \pratio(p^\ast) \big). \label{eq:seminorm}
\end{align}
Furthermore, recall that on the event $\cA_1$, both $\groundtruth$ and $\premle$ lie in the set $\cC(Y)$ defined in~\eqref{eq:constraint}. Hence
\begin{align}
\| \Done \premle \Dtwo^\top \|_1 
&= \premle_{1,1} + \premle_{\dimone, \dimtwo} - \premle_{\dimone, 1} - \premle_{1, \dimtwo} \\
&\le \log \frac{2 Y_{1,1}}{\samples} + \log \frac{2 Y_{\dimone, \dimtwo}}{\samples} - \log \frac{2 Y_{\dimone, 1}}{3 \samples} - \log \frac{2 Y_{1,\dimtwo}}{3 \samples} \\
&= \log \frac{2 Y_{1,1}}{3 \samples} + \log \frac{2 Y_{\dimone, \dimtwo}}{3 \samples} - \log \frac{2 Y_{\dimone, 1}}{\samples} - \log \frac{2 Y_{1,\dimtwo}}{\samples} + 4 \log (3) \\
& \le \groundtruth_{1,1} + \groundtruth_{\dimone, \dimtwo} - \groundtruth_{\dimone, 1} - \groundtruth_{1, \dimtwo} + 4 \log (3) \\
& = \log \big( \pratio(p^\ast) \big) + 4 \log (3) .
\end{align}
We conclude that on the event $\cA_1$,
\begin{align}
\big\| \Done ( \premle - \groundtruth)   \Dtwo^\top \big\|_1 
\le \| \Done \premle \Dtwo^\top \|_1  + \| \Done \groundtruth \Dtwo^\top \|_1 \le 2 \log \big( \pratio(p^\ast) \big) + 4 \log (3) . \label{eq:l1-bd}
\end{align}

It then follows from~\eqref{eq:holder},~\eqref{eq:li-bd}, and~\eqref{eq:l1-bd} that on the event $\cA_1 \cap \cA_3$,
\begin{align}
\langle \Pi  ( \eps   ) ,  \premle - \groundtruth   \rangle 
\lesssim \sqrt{\frac{ \pmax \dimone \dimtwo \log( \dimone / \delta) \log (\dimtwo) }{\samples k} } \Big( \log \big( \pratio(p^\ast) \big) + 1 \Big) .
\end{align}

\subsubsection{Finishing the proof of Theorem~\ref{thm:upper}}

Combining the bounds on the two terms of~\eqref{eq:es} and applying~\eqref{eq:fg}, we obtain that on the event $ \cA_1 \cap \cA_2 \cap \cA_3$ of probability at least $1 - 4 \delta$, 
\begin{align}
	\|  \premle - \groundtruth \|_{p^\ast}^2
	\lesssim {} & \|  \premle - \groundtruth \|_{p^\ast} \sqrt{\frac{ \dimone \log(\dimone / \delta) + k \log(\dimone / \delta) \log(\dimtwo) }{\samples}}  \\
	{} & + \Big( \log \big( \pratio(p^\ast) \big) + 1 \Big) \sqrt{\frac{ \pmax \dimone \dimtwo \log( \dimone / \delta) \log (\dimtwo) }{\samples k} } 
. \label{eq:two-one}
\end{align}
Finally, by the definitions of $\premle$ and $\hat \theta$ in~\eqref{eq:est-1} and~\eqref{eq:est-2}, it holds that
\begin{align}
\KL (p^\ast, \hat p) & = \sum_{i, j} p^\ast_{i, j} \log \frac{p^\ast_{i, j}}{\hat p_{i, j}} \\
& = \sum_{i, j} p^\ast_{i, j} \big( \groundtruth_{i, j} - \hat \theta_{i, j} \big) \\
& = \sum_{i, j} p^\ast_{i, j} \big( \groundtruth_{i, j} -  \tilde \theta_{i, j} \big) + \log \sum_{i, j} e^{\tilde \theta_{i, j}} \\
& \le \langle p^\ast, \groundtruth - \tilde \theta \rangle + \sum_{i, j} e^{\tilde \theta_{i, j}} - 1 \\
& \le 2 \| \premle - \groundtruth \|_{p^\ast}^2 , \label{eq:kl-bd}
\end{align}
where the first inequality holds because $\log x \le x - 1$ and the second holds thanks to \eqref{eq:qua-approx}. 
Therefore, we conclude from~\eqref{eq:two-one} and~\eqref{eq:kl-bd} that 
\begin{align}
\label{eq:fl}
\KL (p^\ast, \hat p)
&\lesssim \frac{ \dimone \log(\dimone / \delta) + k \log(\dimone / \delta) \log(\dimtwo) }{\samples}  \\
& \quad \  + \Big( \log \big( \pratio(p^\ast) \big) + 1 \Big) \sqrt{\frac{ \pmax \dimone \dimtwo \log( \dimone / \delta) \log (\dimtwo) }{\samples k} } .
\end{align}

Balancing out the terms that depend on \( k \) yields the optimal choice
\begin{align}
\label{eq:fo}
k=  \Big( \log \big( \pratio(p^\ast) \big) + 1 \Big)^{2/3}  \Big( \frac{ \pmax \dimone \dimtwo \samples }{ \log(\dimtwo) \log(\dimone / \delta) } \Big)^{1/3} ,
\end{align}
which leads to
\begin{align}
\KL (p^\ast, \hat p)
&\lesssim \frac{ \dimone \log(\dimone / \delta) }{\samples} + (\pmax \dimone \dimtwo)^{1/3} \Big( \log \big( \pratio(p^\ast) \big) + 1 \Big)^{2/3}  
 \Big( \frac{ \log(\dimone / \delta) \log(\dimtwo) }{ \samples } \Big)^{2/3}.
\end{align}
Since the KL divergence dominates the Hellinger distance by the first inequality in \eqref{eq:hel-kl}, this completes the proof. 

\begin{remark}
	\label{rem:ml-vs-emp}
It is not hard to see that, if we choose the set $J$ in \eqref{eq:eu} instead to be the entire grid $[\dimone] \times [\dimtwo]$, then the same argument yields the rate
	\begin{equation}
		\label{eq:ac}
		\hel^2(p^\ast, \hat p) \lesssim \frac{\dimone \dimtwo \log(\dimone / \delta)}{N},
	\end{equation}
	which matches the rate of the empirical frequency matrix in Lemma~\ref{lem:emp} up to a logarithmic factor.
	In fact, the numerical experiments in Section~\ref{sec:numerics}, in particular Figure~\ref{fig:pmf_large_V_var_N}, suggest that the performance of \( \hat p \) exactly matches that of the empirical frequency matrix in this regime.
\end{remark}

\subsubsection{Proof of Lemma~\ref{lem:gram-eig}}
\label{sec:gram-eig}

Let \( B \in \R^{|J^c|} \) with \( \| B \|_2 = 1 \).
Since \( \mathcal{L} B \) is a sum of matrices that are orthonormal with respect to the standard inner product, weighted by the entries of $B$, it holds that \( \| \mathcal{L} B \|_2 = 1 \).
Hence,
\begin{align}
	\label{eq:aa}
	B^\top \mathcal{L}^\top \Lambda^{-1} \mathcal{L} B
	\ge \min_{G : \| G \|_2 = 1} G^\top \Lambda^{-1} G
	= \lambda_{\min}( \Lambda^{-1} )
	= \frac{1}{\pmax},
\end{align}
which yields the first claim. 

For the second claim, recall that $\Pi A$ is defined to be the orthogonal projection of $A$ onto the image of $\mathcal{L}$ with respect to the inner product $\langle \cdot, \cdot \rangle_{1/p^*}$. 
Thus $\Pi A = \mathcal{L} B$ where $B \in \R^{|J^c|}$ minimizes $$
\|\mathcal{L} B - A\|_{1/p^*}^2 = \langle \mathcal{L} B - A, \Lambda^{-1} (\mathcal{L} B - A) \rangle .
$$ 
The first-order optimality condition then gives the desired formula for $\Pi$.

\subsubsection{Proof of Lemma~\ref{lem:null}}
\label{sec:null}

The image of the map \( A \mapsto \Done^\top A \Dtwo \) is the orthogonal complement of the kernel of \( A \mapsto \Done A \Dtwo^\top \), which can be characterized as follows.

The matrices $\Sigma$, $U$ and $W$ in the singular value decomposition $\Done = U \Sigma W^\top$ are known~\cite{Str07} to be
\begin{equation}
\label{eq:cv}
\Sigma_{i, i} = 2 \left| \sin \left( \frac{\pi i}{2 \dimone} \right) \right|, \quad i \in [\dimone-1] ,
\end{equation}
\begin{equation}
\label{eq:cu}
U_{i, j} = \sqrt{\frac{2}{\dimone}} \sin \left( \frac{\pi i j}{\dimone} \right), \quad i, j \in [\dimone-1] ,
\end{equation}
\begin{align}
\label{eq:el}
\text{and } \quad  W_{i, j} =
\left\{
\begin{aligned}
\sqrt{ \frac{2}{\dimone}} \cos \left( \frac{\pi j (i-1/2)}{\dimone} \right), \quad {} & j \in [\dimone-1] , \, i \in [\dimone]  \\
\frac{1}{\sqrt{\dimone}}, \quad {} & j = \dimone .
\end{aligned}
\right. \qquad 
\end{align}

Fix a matrix $A$ for which $D A \Dtwo^\top = 0$. Then we have $\Sigma W^\top A \tilde W \tilde \Sigma^\top = 0$, so the matrix $W^\top A \tilde W$ has all entries equal to zero except on its last row and last column. 
Consequently, it holds that 
$$
A = W W^\top A \tilde W \tilde W^\top = \sum_{i=1}^{\dimone} w_i w_i^\top A \tilde w_{\dimtwo} \tilde w_{\dimtwo}^\top + \sum_{j=1}^{\dimtwo-1}  w_{\dimone} w_{\dimone}^\top A \tilde w_j \tilde w_j^\top .
$$
Hence, the orthogonal complement of the kernel of \( A \mapsto \Done A \Dtwo^\top \) is spanned by the matrices \( \{ w_l \tilde w_r^\top : (l, r) \in [\dimone - 1] \times [\dimtwo - 1] \} \).
By the definition of \( \Pi \) as the projection onto the span of \( \{ w_i \tilde w_j^\top : (i, j) \in J^c \} \), its image is contained in the kernel of \( A \mapsto \Done A \Dtwo^\top \).

%
%
%


\subsubsection{Proof of Lemma~\ref{lem:var-bd}}
\label{sec:var-bd}

This result can be easily obtained from the proof of Lemma~10 of \cite{HutMaoRigRob19}, but we provide a complete proof for the reader's convenience. 

We start with the first bound in the lemma.  Without loss of generality, assume  that \( \dimone \) is odd, so \( \dimone - 1 \) is even. 
Note that because of the symmetry
\begin{equation}
\label{eq:da}
\sin \left( \frac{\pi i j}{\dimone} \right)
= \sin \left( \frac{\pi j (\dimone - i)}{\dimone} \right), \quad i = 1, \dots, \dimone-1,
\end{equation}
it is enough to consider \( i = 1, \dots, \frac{\dimone-1}{2} \).
We make use of the following inequalities to control the \( \sin \) terms involved:
\begin{alignat}{2}
\label{eq:di}
| \sin(x) | \leq {} & 1, \quad && \text{for all } x \in \R;\\
\sin(x) \leq {} & x, \quad && \text{for } x \in [0, \infty);\\
\sin(x) \geq {} & \frac{2}{\pi} x \geq \frac{1}{2} x, \quad && \text{for } x \in [0, \frac{\pi}{2}].
\end{alignat}

Plugging in the entries of $\Sigma$ and $U$ as stated in \eqref{eq:cv} and \eqref{eq:cu}, respectively, yields
\begin{align}
\label{eq:dv}
\sum_{(l, r) \in J^c} \Sigma_{l, l}^{-2} \tilde \Sigma_{r, r}^{-2} U_{i, l}^2 \tilde U_{j, r}^2
= {} & \sum_{(l, r) \in J^c} \frac{4 \sin \left( \frac{\pi i l}{\dimone} \right)^2 \sin \left( \frac{\pi j r}{\dimtwo} \right)^2}{16 \dimone \dimtwo \sin \left( \frac{\pi l}{2 \dimone} \right)^2 \sin \left( \frac{\pi r}{2 \dimtwo} \right)^2}\\
\lesssim {} & \frac{1}{\dimone \dimtwo} \sum_{(l, r) \in J^c} \frac{\dimone^2 \dimtwo^2}{l^2 r^2}\\
\lesssim {} & \dimone \dimtwo \sum_{r=1}^{\dimtwo} \left( \frac{1}{r^2} \sum_{l = \lceil k/r \rceil}^{\dimone} \frac{1}{l^2} \right)\\
\lesssim {} & \dimone \dimtwo \sum_{r=1}^{\dimtwo} \left( \frac{1}{r^2} \sum_{l = \lceil k/r \rceil + 1}^{\dimone} \frac{1}{l^2}\right) + \dimone \dimtwo \sum_{r=1}^{\dimtwo} \frac{1}{r^2} \frac{1}{\lceil k/r \rceil^2}\\
\lesssim {} & \dimone \dimtwo \sum_{r=1}^{\dimtwo} \frac{1}{r^2} \frac{r}{k} + \dimone \dimtwo \sum_{r=1}^{k} \frac{1}{k^2} + \dimone \dimtwo \sum_{r=k+1}^{\dimtwo} \frac{1}{r^2} \\
\lesssim {} & \frac{\dimone \dimtwo}{k} \log (\dimtwo) + \frac{\dimone \dimtwo}{k} + \frac{\dimone \dimtwo}{k} \lesssim \frac{\dimone \dimtwo}{k} \log (\dimtwo) , \label{eq:k-split}
\end{align}
where we have twice used the bound $\sum_{r=k+1}^\infty \frac{1}{r^2} \le \frac{1}{k}$ for any $k \ge 1$.


\subsection{Proof of Theorem \ref{thm:upper-cts}}
\label{sec:upper-cts-proof}

	With the notation introduced 
	in Section~\ref{sec:est-cts}, 
	we define the piecewise constant density
	\begin{equation}
		\label{eq:ge}
		\bar \density(x) \defn \dimtradeoff^2 \pgt_{i, j} 
		= \dimtradeoff^2 \int_{S_{i, j}} \densitygt(x) \, \ud x , \quad x \in S_{i, j}, \, i, j  \in [\dimtradeoff].
	\end{equation}
	By the triangle inequality for the Hellinger distance, we can estimate
	\begin{equation}
		\label{eq:gt}
		\hel^2(\hat \density, \densitygt) \le 2 \hel^2(\hat \density, \bar \density) + 2 \hel^2(\bar \density, \densitygt),
	\end{equation}
	and we proceed to bound the two quantities on the right-hand side of \eqref{eq:gt}.

	For the first term on the right-hand side of \eqref{eq:gt}, we have
	\begin{align}
		\label{eq:gf}
		\hel^2(\hat \density, \bar \density)
		= {} & \int_{[0,1]^2} \Big( \sqrt{\hat \density(x)} - \sqrt{\bar \density (x)} \Big)^2 \, \ud x\\
		= {} & \sum_{i, j = 1}^n \frac{1}{\dimtradeoff^2} \Big( \sqrt{\dimtradeoff^2 \hat p_{i, j}} - \sqrt{\dimtradeoff^2 \pgt_{i, j}} \, \Big)^2\\
		= {} & \hel^2(\hat p, \pgt) . 
	\end{align}
	By assumption \eqref{eq:fz} and definition \eqref{eq:fy}, we have that $ \pmin \ge \dmin / \dimtradeoff^2 $. Hence if $\samples \ge \frac{ 12 \dimtradeoff^2 \log (n^2/\delta) }{ \dmin }$, 
	then the results for the estimator \( \hat p \) in Theorem \ref{thm:upper} lead to 
	\begin{equation}
		\label{eq:gu}
		\hel^2 ( \hat \density, \bar \density ) \lesssim 
		\frac{ \dimtradeoff \log(\dimtradeoff / \delta) }{\samples } 
		+  (\pmax \, \dimtradeoff^2)^{1/3} \Big( \log \big( \pratio(p^\ast) \big) + 1 \Big)^{2/3}  
		\Big( \frac{ \log(\dimtradeoff / \delta) \log(\dimtradeoff) }{ \samples } \Big)^{2/3} 
	\end{equation}
	with probability at least $1- 4 \delta$. 
Moreover, recall definition \eqref{eq:ratio} and note that 
	\begin{equation}
		\label{eq:ga}
		\pgt_i
		\le 
		\frac{\dmax}{\dimtradeoff^2} 
		\quad \text{ and } \quad 
		\pratio( \pgt ) \le \frac{ \dmax^2 }{ \dmin^2 } . 
	\end{equation}
	It then follows that 
\begin{equation}
\label{eq:t1}
	\hel^2 ( \hat \density, \bar \density ) \lesssim \frac{ \dimtradeoff \log( \dimtradeoff / \delta ) }{\samples}  + 	( \dmax )^{1/3}  \Big( \log \Big( \frac{ \dmax }{ \dmin } \Big) + 1 \Big)^{2/3} 
	\Big( \frac{ \log( \dimtradeoff / \delta  ) \log( \dimtradeoff ) }{ \samples } \Big)^{2/3} . 
\end{equation}
	

	To bound the second term on the right-hand side of \eqref{eq:gt}, note that by the mean value theorem for integrals and the continuity of \( \densitygt \), for all \( i , j \in [\dimtradeoff] \), there exist \( \zeta_{i, j} \) such that
	\begin{equation}
		\label{eq:gh}
		\bar \density(x)  = \dimtradeoff^2 \int_{S_{i, j}} \densitygt(y) \, \ud y = \densitygt(\zeta_{i, j}).
	\end{equation}
	Note that for any $a, b > 0$, it holds that 
		\begin{equation}
			\label{eq:gi}
			\big| \sqrt{a} - \sqrt{b} \, \big|
			\le \frac{ | a - b | }{ \sqrt{a} \lor \sqrt{b} } .  
		\end{equation}
	Moreover, assumptions \eqref{eq:go} and \eqref{eq:fv} imply that 
		\begin{equation}
		\label{eq:gs}
		| \densitygt(x) - \densitygt(y) | \le \constholder \| x - y \|_2^{\tilde \expholder} , \quad x, y \in [0,1]^2 ,
		\end{equation}
		where we recall $\tilde \expholder = \expholder \land 1$. 
	Combining the above facts, we obtain 
	\begin{align}
		\label{eq:gg}
		\hel^2(\bar \density, \densitygt)
		= {} & \sum_{i, j = 1}^n \int_{S_{i, j}} \Big( \sqrt{\densitygt(x)} - \sqrt{\densitygt(\zeta_{i, j})} \, \Big)^2 \, \ud x\\
		\le {} & \frac{1}{\dmin} \sum_{i, j = 1}^n \int_{S_{i, j}} \Big( \densitygt(x) - \densitygt(\zeta_{i, j}) \Big)^2 \, \ud x\\
		\le {} & \frac{ \constholder^2 }{\dmin} \sum_{i, j = 1}^n \int_{S_{i, j}} \diam(S_{i, j})^{2 \tilde \expholder}
		= \frac{ 2 \constholder^2 }{\dmin} \dimtradeoff^{-2 \tilde \expholder} . \label{eq:t2}
	\end{align}

	Plugging inequalities \eqref{eq:t1} and \eqref{eq:t2} into \eqref{eq:gt}, we conclude that for $\samples \ge \frac{ 12 \dimtradeoff^2 \log (\dimtradeoff^2/\delta) }{ \dmin }$, 
	\begin{equation}
		\label{eq:gj}
		\hel^2(\hat \density, \densitygt) \lesssim 
		\frac{ \dimtradeoff \log( \dimtradeoff / \delta ) }{\samples}  +  ( \dmax )^{1/3} \Big( \log \Big( \frac{ \dmax }{ \dmin } \Big) + 1 \Big)^{2/3} 
		\Big( \frac{ \log( \dimtradeoff / \delta  ) \log( \dimtradeoff ) }{ \samples } \Big)^{2/3} 
		+ \frac{ \constholder^2 }{\dmin} \dimtradeoff^{-2 \tilde \expholder} 
	\end{equation}
	with probability $1 - 4 \delta$. 
	Setting
	\begin{equation}
		\label{eq:gk}
		\dimtradeoff = \Big\lfloor \Big( \frac{\constholder^2 \samples}{\dmin}  \Big)^{1/(2 \tilde \expholder + 1)} \land \Big( \frac{\dmin \samples }{ 24 \log ( \frac{ \dmin \samples}{ 12 \delta } ) } \Big)^{1/2} \Big\rfloor 
	\end{equation}
	then leads to
	the bound
	\begin{equation}
		\label{eq:ub1}
		\hel^2(\hat \density, \densitygt) 
		\lesssim \frac{ \log(\constholder \samples / \delta) }{ \dmin^{1/(2 \tilde \expholder + 1)} }
		\frac{ \constholder^{2/(2 \tilde \expholder + 1)} }{ \samples^{2 \tilde \expholder/(2 \tilde \expholder + 1)} }  + ( \dmax )^{1/3} \Big( \log \Big( \frac{ \dmax }{ \dmin } \Big) + 1 \Big)^{2/3} 
		\Big( \frac{ \log( \constholder \samples / \delta  ) \log( \constholder \samples ) }{ \samples } \Big)^{2/3} ,
	\end{equation}
	for $\expholder > 0.5 $ and $\samples \gtrsim \big[ \frac{ \constholder^4 }{ \dmin^{2 \tilde \beta+3} } \log^{2 \tilde \beta +1} ( \frac{ \constholder }{\dmin \delta } ) \big]^{ \frac{1}{ 2 \tilde \expholder - 1} } , $  
	and 
		\begin{align}
			\hel^2(\hat \density, \densitygt) &\lesssim \Big[ \frac{\dmin  \log( \samples / \delta ) }{ \samples } \Big]^{1/2}
			+ ( \dmax )^{1/3} \Big( \log \Big( \frac{ \dmax }{ \dmin } \Big) + 1 \Big)^{2/3} 
				\Big( \frac{ \log( \samples / \delta  ) \log( \samples ) }{ \samples } \Big)^{2/3} \\
			& \quad + \frac{ \constholder^2 }{\dmin} \Big[ 
				\frac{ \log( \samples / \delta ) }{ \dmin \samples } \Big]^{\tilde \expholder} ,
						\label{eq:ub2}
		\end{align}
		for $0 < \beta \le 0.5$, where the hidden constants depend on $\beta$. 
		Choosing $\delta = 1/(4 \samples^4)$ completes the proof. 

\subsection{Proof of Theorem \ref{thm:lower}}


We prove the theorem by treating the two terms $\dimone/\samples$ and $1/\samples^{2/3}$ separately. 

\subsubsection{The first term \texorpdfstring{$\dimone/\samples$}{n1/N0}}

Without loss of generality, we assume that $8$ divides $\dimone$. Let $\ham$ denote the Hamming distance between two binary vectors. By the Gilbert-Varshamov bound (see, for example, \cite[Lemma~4.10]{Mas07}), there exists a set $\{ w^{(k)} \}_{k=1}^M$ of points in $\{0, 1\}^{\dimone}$ such that
\begin{itemize} 
\item
$\ham(0, w^{(k)} ) = \dimone/4$, 

\item
$\ham( w^{(k)} , w^{(\ell)} ) \ge \dimone/8$ for all distinct $k, \ell \in [M]$, and 

\item
$\log(M) \ge \dimone/30$.
\end{itemize} 

For $\delta \in [0, 1]$ and each $k \in [M]$, let us define a density $p^{(k)}$ on $[\dimone] \times [\dimtwo]$ by
$$
p^{(k)}_{i, j} = 
\frac{ 4 ( 1 + \delta w^{(k)}_i ) }{(4 + \delta) \dimone \dimtwo}  \quad \text{ for } (i, j) \in [\dimone] \times [\dimtwo].
$$
Note that each $p^{(k)}$ is indeed a density because $\sum_{i} w^{(k)}_i  = \dimone/4$ and thus
$$
\sum_{i, j} p^{(k)}_{i, j} = \sum_{i, j} \frac{ 4 ( 1 + \delta w^{(k)}_i ) }{(4 + \delta) \dimone \dimtwo} 
= \frac 4{4+\delta} + \frac{ 4 \delta  }{(4 + \delta) \dimone} \sum_{i} w^{(k)}_i  
= 1.
$$
Also, each $p^{(k)}$ is totally positive because it has constant rows and thus $p^{(k)}_{i, j} p^{(k)}_{i+1, j+1} = p^{(k)}_{i, j+1} p^{(k)}_{i+1, j}$.

Furthermore, since $\delta \in [0, 1]$, we see that $\frac{4}{5 \dimone \dimtwo} \le p^{(k)}_{i, j} \le \frac{8}{5 \dimone \dimtwo}$. 
The relation~\eqref{eq:hel-kl} yields 
\begin{align}
\KL(p^{(k)}, p^{(\ell)}) \le 5 \hel^2( p^{(k)}, p^{(\ell)} ) . \label{eq:dist-eq}
\end{align}
Note that $p^{(k)}_{i, j}$ can only take two possible values $\frac{4}{(4 + \delta) \dimone \dimtwo}$ or $\frac{4 (1 + \delta)}{(4 + \delta) \dimone \dimtwo}$, so $\Big( \sqrt{p^{(k)}_{i, j}} - \sqrt{p^{(\ell)}_{i, j}} \, \Big)^2$ is either $0$ or $\frac{4 ( \sqrt{ 1 + \delta } - 1 )^2 }{ (4 + \delta) \dimone \dimtwo }$. 
Therefore, we have 
\begin{align}
\hel^2( p^{(k)}, p^{(\ell)} ) 
= \sum_{i \in [\dimone], \, j \in [\dimtwo]} \Big( \sqrt{p^{(k)}_{i, j}} - \sqrt{p^{(\ell)}_{i, j}} \, \Big)^2 = \ham( w^{(k)}, w^{(\ell)} ) \frac{4 ( \sqrt{ 1 + \delta } - 1 )^2 }{ (4 + \delta) \dimone } . \label{eq:hel-exp}
\end{align}
Together with the condition $\dimone/8 \le \ham( w^{(k)} , w^{(\ell)} ) \le \dimone$ for $k \ne \ell$, relations~\eqref{eq:dist-eq} and~\eqref{eq:hel-exp} yield 
\begin{align}
\hel^2( p^{(k)}, p^{(\ell)} )  \ge \frac{ ( \sqrt{ 1 + \delta } - 1 )^2 }{2 (4 + \delta) } \quad \text{and} \quad \KL ( p^{(k)} , p^{(\ell)} ) \le \frac{20 ( \sqrt{ 1 + \delta } - 1 )^2 }{ (4 + \delta) } . \label{eq:kl-ul}
\end{align}
Additionally, since the KL divergence tensorizes, if we let $p^{\otimes \sample}$ denote the distribution of $\sample$ independent observations sampled according to the density $p$ on $[\dimone] \times [\dimtwo]$, then
\begin{align}
\KL \big( ( p^{(k)} )^{ \otimes \sample } , ( p^{(\ell)} )^{ \otimes \sample } \big) = \sample \, \KL ( p^{(k)} , p^{(\ell)} ) \le \sample \frac{20 ( \sqrt{ 1 + \delta } - 1 )^2 }{ (4 + \delta) } . \label{eq:n-kl}
\end{align}

For a sufficiently small constant $c_1>0$, choose $\delta \in [0, 1]$ so that $\sample \frac{20 ( \sqrt{ 1 + \delta } - 1 )^2 }{ (4 + \delta) } = c_1  \dimone \le 0.1 \log(M)$. We can apply \cite[Theorem~2.5]{Tsy09} together with~\eqref{eq:kl-ul} and~\eqref{eq:n-kl} to obtain that
$$
\inf_{\hat p} \sup_{p^\ast \, \mtp } \p_{ ( p^\ast )^{ \otimes \sample } } \Big\{ \hel^2( \hat p, p^\ast ) \ge c_2 \frac{\dimone}{\sample} \Big\} 
\ge 0.1 .
$$

\subsubsection{The second term \texorpdfstring{$1/\samples^{2/3}$}{1/N0\textasciicircum(2/3)}}

We turn to the second term in the lower bound. Consider positive integers $k_1 \le \dimone$ and $k_2 \le \dimtwo$ such that $4$ divides $k_1 k_2$, and $k_i$ divides $n_i$ for $i = 1,2$, without loss of generality. If $k_i$ does not divide $n_i$, with minor revision, the proof works on a sub-grid $[\dimone'] \times [\dimtwo']$ where $k_i$ divides $n_i'$ and $n_i/2 \le n_i' \le n_i$. Thus we make these mild assumptions to ease the notation.


The strategy of proving the lower bound is based on constructing an appropriate packing of supermodular log-densities, which correspond to totally positive densities. 
By the Gilbert-Varshamov bound~\cite[Lemma~4.10]{Mas07} again, we obtain a set $\{\tau^{(\ell)} \}_{\ell=1}^M$ of matrices in $\{0, 1\}^{k_1 \times k_2}$ such that
\begin{itemize} 
\item
$\ham(0, \tau^{(\ell)} ) = k_1 k_2 /4$, 

\item
$\ham( \tau^{(\ell)} , \tau^{(r)} ) \ge k_1 k_2/8$ for all distinct $\ell, r \in [M]$, and 

\item
$\log(M) \ge k_1 k_2/30$.
\end{itemize} 

For each $\tau^{(\ell)}$, we need to carefully define a log-density $\theta^{(\ell)}\in \R^{\dimone \times \dimtwo}$ that is supermodular and amenable to distance calculation. To that end, we have the following construction which simplifies computation later. 
For $i\in[\dimone]$ and $j\in[\dimtwo]$, define $u_i \defn \lceil i k_1/\dimone \rceil \in [k_1]$ and $v_j \defn \lceil j k_2/\dimtwo \rceil \in [k_2]$.
Moreover, for any $\delta \in [0, 1/6]$ and $(i, j) \in [\dimone] \times [\dimtwo]$, we define $\tilde \delta_{u_i, v_j} \in [\delta, 3 \delta] \subset [0, 1/2]$ so that
\begin{align}
\exp \Big( \frac{u_iv_j}{k_1k_2} +  \frac{  \tilde \delta_{u_i, v_j} }{k_1k_2} \Big)
- \exp \Big( \frac{u_iv_j}{k_1k_2} \Big)
= \exp \Big( 1 +  \frac{ \delta }{k_1k_2} \Big) - \exp(1) . \label{eq:deltaij}
\end{align}
To see why $\tilde \delta_{u_i, v_j}$ is properly defined in the range $[\delta, 3 \delta]$, first note that 
the quantity $\tilde \delta_{u_i, v_j}$ is larger for smaller $u_i v_j \in [k_1 k_2]$. Hence it suffices to check that there exists $\delta' \in [\delta, 3 \delta]$ such that
\begin{align}
\exp \Big( \frac{  \delta' }{k_1k_2} \Big) - \exp(0)
= \exp \Big( 1 +  \frac{ \delta }{k_1k_2} \Big) - \exp(1) .
\end{align}
This follows from 
that for $x \in [0,1/6]$,
\begin{equation}
	\label{eq:hp}
	\exp(x) - \exp(0) \le \exp(1+x) - \exp(1) \le \exp(3x) - \exp(0) .
\end{equation}

With $\tau^{(\ell)}_{u_i,v_j}$ chosen earlier and $\tilde \delta_{u_i, v_j}$ defined in~\eqref{eq:deltaij}, we consider the quantity
\begin{align}
\tilde \theta^{(\ell)}_{i, j} \defn \frac{u_iv_j}{k_1k_2} +  \tau^{(\ell)}_{u_i,v_j} \frac{ \tilde \delta_{u_i, v_j} }{k_1k_2} , \label{eq:theta-def-1} 
\end{align}
and further define the log-density
\begin{align}
\theta^{(\ell)}_{i, j} &\defn  \tilde \theta^{(\ell)}_{i, j}  - \log \sum_{s \in [\dimone], \, t \in [\dimtwo]} \exp \big( \tilde \theta^{(\ell)}_{i, j}  \big) \\
&= \frac{u_iv_j}{k_1k_2} +  \tau^{(\ell)}_{u_i,v_j} \frac{ \tilde \delta_{u_i, v_j} }{k_1k_2}  - \log \underbrace{\sum_{s \in [\dimone], \, t \in [\dimtwo]} \exp \Big( \frac{u_sv_t}{k_1k_2} + \tau^{(\ell)}_{u_s,v_t} \frac{ \tilde \delta_{u_s, v_t} }{k_1k_2}  \Big)}_{=: \, \mathfrak{N}} . \label{eq:theta-def-2} 
\end{align}
Finally, the density $p^{(\ell)}$ is defined by $p^{(\ell)}_{i, j} \defn \exp( \theta^{(\ell)}_{i, j} )$.

Note that the normalization factor $\mathfrak{N}$ in~\eqref{eq:theta-def-2} guarantees that $\sum_{i, j} \exp(\theta^{(\ell)}_{i, j}) = 1$, so $p^{(\ell)}$ is indeed a density. Moreover, crucial to our computation later, the normalization factor in fact does not depend on $\ell \in [M]$ thanks to the definition of $\tilde \delta_{u_i, v_j}$ in~\eqref{eq:deltaij}; namely,
\begin{align}
\mathfrak{N} &= \sum_{s \in [\dimone], \, t \in [\dimtwo]} \exp \Big( \frac{u_sv_t}{k_1k_2} + \tau^{(\ell)}_{u_s,v_t} \frac{ \tilde \delta_{u_s, v_t} }{k_1k_2}  \Big) \\
& = \sum_{s \in [\dimone], \, t \in [\dimtwo]} \exp \Big( \frac{u_sv_t}{k_1k_2} \Big) + \sum_{ (s, t): \, \tau^{(\ell)}_{u_s,v_t} = 1 } \Big[ \exp \Big( \frac{u_sv_t}{k_1k_2} + \frac{ \tilde \delta_{u_s, v_t} }{k_1k_2}  \Big) - \exp \Big( \frac{u_sv_t}{k_1k_2} \Big) \Big] \\
& = \sum_{s \in [\dimone], \, t \in [\dimtwo]} \exp \Big( \frac{u_sv_t}{k_1k_2} \Big) + \Big| \Big\{ (s, t): \, \tau^{(\ell)}_{u_s,v_t} = 1 \Big\} \Big| \cdot \Big[ \exp \Big( 1 +  \frac{ \delta }{k_1k_2} \Big) - e  \Big] \\
& = \sum_{s \in [\dimone], \, t \in [\dimtwo]} \exp \Big( \frac{u_sv_t}{k_1k_2} \Big) + \frac{\dimone \dimtwo}{4} \Big[ \exp \Big( 1 +  \frac{ \delta }{k_1k_2} \Big) - e  \Big] 
\end{align}
where the last equality follows from that $\ham(0, \tau^{(\ell)} ) = k_1 k_2 /4$ and the definitions of $u_s$ and $v_t$. 

Next, we check that $\theta^{(\ell)}$ is supermodular, so that $p^{(\ell)}$ is totally positive. Since $\tilde \theta^{(\ell)} \in \R^{\dimone \times \dimtwo}$ defined in~\eqref{eq:theta-def-1} is equal to $\theta^{(\ell)}$ plus a common constant on each entry, it suffices to check that $\tilde \theta^{(\ell)}_{i, j} + \tilde \theta^{(\ell)}_{i+1,j+1}-\tilde \theta^{(\ell)}_{i,j+1}-\tilde \theta^{(\ell)}_{i+1,j} \geq 0$. 
There are two cases:
\begin{enumerate}
\item 
If $u_i = u_{i+1}$ or $v_j = v_{j+1}$, then we have, respectively, either $\tilde \theta^{(\ell)}_{i, j} = \tilde \theta^{(\ell)}_{i+1,j}, \,  \tilde \theta^{(\ell)}_{i,j+1}=\tilde \theta^{(\ell)}_{i+1,j+1}$ or $\tilde \theta^{(\ell)}_{i, j} = \tilde \theta^{(\ell)}_{i,j+1}, \,  \tilde \theta^{(\ell)}_{i+1,j}=\tilde \theta^{(\ell)}_{i+1,j+1}$. In both cases, the difference above is $0$.

\item 
Otherwise, we have $u_{i+1} = u_i + 1$ and $v_{j+1} = v_j + 1$. Then it holds
\begin{align}
&\quad \  \tilde \theta^{(\ell)}_{i, j} + \tilde \theta^{(\ell)}_{i+1,j+1}-\tilde \theta^{(\ell)}_{i,j+1}-\tilde \theta^{(\ell)}_{i+1,j}  \\
&= \frac{u_iv_j + (u_i+1)(v_j+1) - u_i(v_j+1)-(u_i+1)v_j}{k_1k_2} \\
&\quad \  + \frac{\tau^{(\ell)}_{u_i,v_j} \tilde \delta_{u_i, v_j} +\tau^{(\ell)}_{u_{i+1},v_{j+1}} \tilde \delta_{u_{i+1}, v_{j+1}} -\tau^{(\ell)}_{u_i,v_{j+1}} \tilde \delta_{u_i, v_{j+1}} -\tau^{(\ell)}_{u_{i+1},v_j} \tilde \delta_{u_{i+1}, v_j} }{k_1k_2} \\
& \geq  \frac1{k_1k_2} - \frac {6\delta} {k_1k_2}  \geq 0 ,
\end{align}
\end{enumerate}
since $\tilde \delta_{u_i, v_j} \le 3 \delta \le 1/2$. 


Having verified that each $\theta^{(\ell)}$ is a supermodular log-density, we proceed to study $\hel^2 (\theta^{(\ell)}, \theta^{(r)})$ for distinct $\ell, r \in [M]$.  
Since the normalization term $\mathfrak{N}$ does not depend on the index $\ell$, definition~\eqref{eq:theta-def-2} yields
\begin{align}
| \theta^{(\ell)}_{i, j} -  \theta^{(r)}_{i, j} | = 
\big| \tau^{(\ell)}_{u_i,v_j} - \tau^{(r)}_{u_i,v_j} \big| \frac{ \tilde \delta_{u_i, v_j} }{k_1k_2} \le \frac 12 . \label{eq:term-bd}
\end{align}
By the definition of the Hellinger distance, it holds that
\begin{align}
\hel^2( p^{(\ell)} , p^{(r)} ) &= \sum_{i \in [\dimone], \, j \in [\dimtwo]} \Big( \sqrt{p^{(\ell)}_{i, j}} - \sqrt{p^{(r)}_{i, j}} \, \Big)^2 \\ 
&= \sum_{i \in [\dimone], \, j \in [\dimtwo]} \Big( \exp( \theta^{(\ell)}_{i, j}/2 ) - \exp( \theta^{(r)}_{i, j}/2 ) \Big)^2 \\
&= \sum_{i \in [\dimone], \, j \in [\dimtwo]} p^{(\ell)}_{i, j} \Big( 1 - \exp \big( (\theta^{(r)}_{i, j} - \theta^{(\ell)}_{i, j} )/2 \big)  \Big)^2 .
\end{align}
Using the approximation $x^2/2 \le (1 - e^x)^2 \le 2 x^2$ for $|x| \le 1/4$, we obtain
\begin{align}
\frac 18 \sum_{i \in [\dimone], \, j \in [\dimtwo]} p^{(\ell)}_{i, j} \big( \theta^{(r)}_{i, j} - \theta^{(\ell)}_{i, j} \big)^2 \le
\hel^2( p^{(\ell)} , p^{(r)} ) \le
\frac 12 \sum_{i \in [\dimone], \, j \in [\dimtwo]} p^{(\ell)}_{i, j} \big( \theta^{(r)}_{i, j} - \theta^{(\ell)}_{i, j} \big)^2 .
\end{align}
Furthermore, it is easily seen from~\eqref{eq:theta-def-2} that
$
| \theta^{(\ell)}_{i, j} - \theta^{(\ell)}_{i,' j'} | \le 1.5
$,
so that $1/5 \le p^{(\ell)}_{i, j} / p^{(\ell)}_{i,' j'} \le 5$ for any $(i, j), (i', j') \in [\dimone] \times [\dimtwo]$. As a result, we obtain $\frac{1}{5 \dimone \dimtwo} \le p^{(\ell)}_{i, j} \le \frac{5}{\dimone \dimtwo}$ and thus
\begin{align}
\frac {1}{40 \dimone \dimtwo} \big\| \theta^{(\ell)} - \theta^{(r)} \big\|_2^2 
\le \hel^2( p^{(\ell)} , p^{(r)} ) \le
\frac{5}{2\dimone \dimtwo} \big\| \theta^{(\ell)} - \theta^{(r)} \big\|_2^2 . \label{eq:hel-l2-bd}
\end{align}
In addition, it follows from~\eqref{eq:hel-kl} that
\begin{align}
\KL( p^{(\ell)} , p^{(r)} ) \le 11 \, \hel^2( p^{(\ell)} , p^{(r)} ) .  \label{eq:kl-hel-bd}
\end{align}

It remains to study $\| \theta^{(\ell)} - \theta^{(r)} \|_2^2$. To this end, we obtain from~\eqref{eq:theta-def-2} that
\begin{align}
\sum_{i \in [\dimone], j \in [\dimtwo]}  (\theta^{(\ell)}_{i, j} - \theta^{(r)}_{i,j})^2  
&= \sum_{i \in [\dimone], j \in [\dimtwo]}  \big( \tau^{(\ell)}_{u_i,v_j} - \tau^{(r)}_{u_i,v_j} \big)^2 \Big( \frac{ \tilde \delta_{u_i, v_j} }{k_1k_2} \Big)^2 \\
&= \sum_{u \in[k_1], v \in[k_2]} \frac{\dimone \dimtwo}{k_1 k_2} \big( \tau^{(\ell)}_{u,v} - \tau^{(r)}_{u,v} \big)^2 \frac{ \tilde \delta_{u, v}^2 }{k_1^2 k_2^2} .
\end{align}
Since $\tilde \delta_{u, v} \in [\delta, 3 \delta]$, we have the bounds
\begin{align}
\frac{ \delta^2 \dimone \dimtwo}{k_1^3 k_2^3}  \sum_{u \in[k_1], v \in[k_2]}  \big( \tau^{(\ell)}_{u,v} - \tau^{(r)}_{u,v} \big)^2 
\le \| \theta^{(\ell)} - \theta^{(r)} \|_2^2 
\le \frac{ 9 \delta^2 \dimone \dimtwo}{k_1^3 k_2^3}  \sum_{u \in[k_1], v \in[k_2]}  \big( \tau^{(\ell)}_{u,v} - \tau^{(r)}_{u,v} \big)^2 .
\end{align}
By the construction of the packing $\{\tau^{(\ell)}\}_{\ell \in [M]}$,
$$
\frac{ k_1 k_2 }{ 8 } \le \sum_{u \in[k_1], v \in[k_2]}  \big( \tau^{(\ell)}_{u,v} - \tau^{(r)}_{u,v} \big)^2 = \ham(\tau^{(\ell)}, \tau^{(r)} ) \le k_1 k_2 .
$$
Therefore, combining the above bounds yields that
\begin{align}
\frac{ \delta^2 \dimone \dimtwo}{8 k_1^2 k_2^2}  
\le \| \theta^{(\ell)} - \theta^{(r)} \|_2^2 
\le \frac{ 9 \delta^2 \dimone \dimtwo}{k_1^2 k_2^2}  .
\end{align}
This together with~\eqref{eq:hel-l2-bd} implies that
\begin{align}
\frac{ \delta^2 }{320 k_1^2 k_2^2}  
\le \hel^2( p^{(\ell)} , p^{(r)} ) 
\le \frac{ 45 \delta^2}{2 k_1^2 k_2^2}  .  \label{eq:hel-ul-bd}
\end{align}

To complete the proof, we may choose $\delta = c_1 \big( \frac{k_1^3 k_2^3}{\sample} \big)^{1/2} \in [0, 1/6]$ for a sufficiently small constant $c_1>0$, provided that $k_1^3 k_2^3 \lesssim \sample$. Then the bounds~\eqref{eq:hel-ul-bd} and~\eqref{eq:kl-hel-bd} combined imply that 
$$
\KL \big( ( p^{(k)} )^{ \otimes \sample } , ( p^{(\ell)} )^{ \otimes \sample } \big) = \sample \, \KL( p^{(\ell)}, p^{(r)} ) \le c_2 k_1 k_2 \le 0.1 \log(M) .
$$
Therefore, we can apply \cite[Theorem~2.5]{Tsy09} together with the lower bound in~\eqref{eq:hel-ul-bd} to see that
$$
\inf_{\hat p} \sup_{p^\ast \, \mtp } \p_{ ( p^\ast )^{ \otimes \sample } } \Big\{ \hel^2( \hat p , p^* )  \ge c_2 \frac{k_1 k_2}{\sample} \Big\} 
\ge 0.1 ,
$$
where we continue to use the notation $\hat \theta_{i,j} = \log \hat p_{i, j}$ and $\theta^\ast_{i,j} = \log p^\ast_{i, j}$. 

Note that $k_1 k_2$ needs to be chosen so that $\delta = c_1 \big( \frac{k_1^3 k_2^3}{\sample} \big)^{1/2} \le 1/6$. Hence if $\sample \lesssim \dimone^3 \dimtwo^3$, then we choose $k_1 k_2 \asymp \sample^{1/3}$ to obtain the lower bound of order $\sample^{-2/3}$. If $\sample \gtrsim \dimone^3 \dimtwo^3$, then we choose $k_1 = \dimone$ and $k_2 = \dimtwo$ to obtain the lower bound of order $\frac{\dimone \dimtwo}{\sample}$.

\subsection{Proof of Theorem \ref{thm:lower-cts}}

We first set up the proof for smooth densities, and then prove the theorem for each regime of $\expholder$. 

\subsubsection{Differential characterization and setup}

We begin by stating a short lemma that yields a condition for total positivity in terms of the derivatives of a density.

\begin{lemma}
	\label{lem:diff-char}
	A function \( f \in C^2([0,1]^2) \) fulfills
	\begin{equation}
		\label{eq:gz}
		f(w, z) + f(x, y) \ge f(x,z) + f(w, y), \quad \text{for all } 0 \le x \le w \le 1, \, 0 \le y \le z \le 1
	\end{equation}
	if and only if
	\begin{equation}
		\label{eq:ha}
		\partial_1 \partial_2 f(x, y) \ge 0, \quad \text{for all } x, y \in [0,1].
	\end{equation}
	Moreover, if \( \density > 0 \) is a probability density in \( C^2([0,1]^2) \), then \( \density \) is totally positive if and only if \( \log \density \) fulfills \eqref{eq:ha}, if and only if
	\begin{equation}
		\label{eq:hb}
		-\frac{1}{\density(x,y)^2} \partial_1 \density(x, y) \partial_2 \density(x, y) + \frac{1}{\density(x,y)} \partial_1 \partial_2 \density(x, y) \ge 0, \quad \text{for all } x, y \in [0,1].
	\end{equation}
\end{lemma}

\begin{proof}
The first claim follows easily from the fundamental theorem of calculus and the continuity of $f$. 
	To obtain the second claim, note that because \( \density \) is bounded away from zero, we can take logarithms, and the $\mtp$ condition~\eqref{eq:tp} is equivalent to \eqref{eq:gz} with \( f = \log \density \).
	Computing the derivative of \( \log \density \) by applying the chain rule finally yields the condition \eqref{eq:hb}.
\end{proof}

\declarecommand{\constholdnorm}{C_1}
\declarecommand{\constsep}{c_2}
\declarecommand{\constdeltaenlarge}{C_3}
\declarecommand{\constbounded}{C_4}
\declarecommand{\constexp}{C_5}
\declarecommand{\constltwo}{C_6}
\declarecommand{\constceil}{C_7}


	To prove Theorem~\ref{thm:lower-cts}, 
	we distinguish three cases, \( \expholder \le 1 \), \( \expholder \ge 2 \), and \( 1 < \expholder < 2 \).
	In the first case where $\expholder \le 1$, 
	it suffices to consider densities that only depend on one variable, which fulfill the $\mtp$ constraint automatically, leading to the rate \( \sample^{-2 \expholder/(2 \expholder + 1)} \) for the estimation of a one-dimensional H\"older function. 	
	On the other hand, in the case \( \expholder \ge 2 \), we appeal to two-dimensional constructions in density estimation. The $\mtp$ condition is a second-order constraint and hence can be satisfied by a carefully chosen set of H\"older functions for \( \expholder \ge 2 \), leading to the rate \( \sample^{-2 \expholder/(2 \expholder + 2)} \). 
	Finally, in the remaining regime \( 1 \le \expholder \le 2 \), we in fact use the construction for \( \expholder = 2 \), yielding the rate \( \sample^{-2/3} \).


	\subsubsection{Case \texorpdfstring{\( \expholder \le 1 \)}{beta <= 1}.}

	For \( \expholder \le 1 \), 
	we apply an argument based on Fano's inequality, \cite[Theorem 2.5]{Tsy09}.
	The following construction is standard in proving lower bounds for nonparametric estimation; see \cite[Section 2.6]{Tsy09}, for example.

	Fix a nonzero function \( g \in C^{\infty}(\R) \) supported in \( [0,1] \) such that \( \int_\R g(x) \, \ud x = 0 \) and \( g \) is $1/2$-Lipschitz. 
	Let \( k \in \mathbb{N}	\) and assume without loss of generality that it is divisible by 4.
	Denote by \( x_j \), \( j = 1, \dots, k, \) the left endpoints of an equidistant subdivision of the interval \( [0, 1] \), that is,
	\begin{equation}
		\label{eq:gx}
		x_j = \frac{j - 1}{k}, \quad j = 1, \dots, k,
	\end{equation}
	leading to the subdivision
	\begin{equation}
		\label{eq:he}
		S_{j} = (x_j, 0) + [0,\frac{1}{k}] \times [0,1], \quad j = 1, \dots, k ,
	\end{equation}
	of the square \( [0,1]^2 \).
	With this, define the functions \( g_j \in C^{\infty}([0, 1]^2) \) as
	\begin{equation}
		\label{eq:gw}
		g_j(x, y) = \frac{1}{k^\expholder} g(k(x - x_j)), \quad (x, y) \in [0,1]^2, \, j = 1, \dots, k. 
	\end{equation}
	Note that each \( g_j \) is supported in \( S_j \). 
Moreover, we have \( g_j \in \classholder(\expholder, 1/2) \) for $\expholder \le 1$: for $(x, y), (x', y') \in S_j$, by the $1/2$-Lipschitzness of $g$, 
\begin{align}
| g_j(x, y) - g_j(x', y') | &= \frac{1}{k^\expholder} | g(k (x - x_j)) - g(k (x' - x_j)) | \\
&\le \frac{1}{2 k^\expholder} | k(x - x') |
\le \frac{1}{2 k^\expholder} | k(x - x') |^{\beta} = \frac 12 |x - x'|^\beta . 
\end{align}
For \( (x,y) \in S_j \) and \( (x', y') \in S_{j'} \) in the case \( j \neq j' \), we obtain the same estimate by applying the above to segments of the line connecting the two points.
This also shows that $\|g_j\|_\infty \le 1/2$ as we can take $(x, y)$ at the boundary of $S_j$ so that $g_j(x, y) = 0$.

	By the Gilbert-Varshamov bound \cite[Lemma 4.10]{Mas07}, there exists a set \( \{\tau^{(\ell)}\}_{\ell=1}^M \) with \( \tau^{(\ell)} \in \{0,1\}^k \) such that
	\begin{itemize} 
		\item \(\ham(0, \tau^{(\ell)} ) = k/4\), 
		\item \(\ham( \tau^{(\ell)} , \tau^{(r)} ) \ge k/8\) for all distinct \(\ell, r \in [M]\), and 
		\item \(\log(M) \ge k/30\).
	\end{itemize} 
	For each \( \ell \in [M] \), set
	\begin{equation}
		\label{eq:gy}
		\density^{(\ell)}(x, y) = 1 + \sum_{j=1}^{k} \tau^{(\ell)}_j g_j(x, y).
	\end{equation}

	As \( \| g_j \|_\infty \le 1/2 \) for each \( j \), all \( \density^{(\ell)} \) are bounded within $[1/2, 3/2]$.
	Since \( g \) is mean-zero, \( \density^{(\ell)} \) is a density. We also have that \( \density^{(\ell)} \in \classholder(\expholder, 1/2) \) by definition because $ g_j \in \classholder(\expholder, 1/2)$. 
	Moreover, checking condition \eqref{eq:hb} in Lemma \ref{lem:diff-char} yields that all densities are $\mtp$, since they only depend on \( x \).

	To check the conditions of \cite[Theorem 2.5]{Tsy09}, we first apply \eqref{eq:hel-kl} to obtain that
	\begin{equation}
		\label{eq:hz}
		\KL(\density^{(\ell)}, \density^{(r)}) \lesssim h^2(\density^{(\ell)}, \density^{(r)}), \quad \ell, k \in [M],
	\end{equation}
	since all densities are bounded from above and below.
	Next, the boundedness of the densities and the mean value theorem together imply that 
	\begin{equation}
		\label{eq:ia}
		h^2(\density^{(\ell)}, \density^{(r)}) \asymp \int_{[0,1]^2} (\density^{(\ell)}(z) - \density^{(r)}(z))^2 \, \ud z = \| \density^{(\ell)} - \density^{(r)} \|_{L^2([0,1]^2)},
	\end{equation}
	which can be estimated as
	\begin{align}
		\label{eq:hd}
		\| \density^{(\ell)} - \density^{(r)} \|_{L^2([0,1]^2)}
		= {} & \sum_{j = 1}^{k} \int_{S_j} \left(\density^{(\ell)}(z) - \density^{(r)}(z)\right)^2 \, \ud z\\
		= {} & \sum_{j : \tau^{(\ell)}_j \neq \tau^{(r)}_j} \int_{S_j} (g_j(z))^2 \, \ud z\\
		= {} & \frac{1}{k^{2 \expholder + 1}} \sum_{j : \tau^{(\ell)}_j \neq \tau^{(r)}_j} \int_0^1 (g(x))^2 \, \ud x\\
		\asymp {} & \ham(\tau^{(\ell)}, \tau^{(r)}) k^{-2 \expholder-1} 
		\asymp k^{- 2 \expholder} , 
	\end{align}
	where the hidden constants only depend on the choice of $g$, which can be made absolute.  

	That means on the one hand that
	\begin{equation}
		\label{eq:ib}
		\KL((\density^{(\ell)})^{\otimes \sample}, (\density^{(r)})^{\otimes \sample}) \lesssim \sample k^{-2 \expholder},
	\end{equation}
	and on the other hand that
	\begin{equation}
		\label{eq:ic}
		\hel^2(\density^{(\ell)}, \density^{(r)}) \gtrsim k^{-2 \expholder}.
	\end{equation}
	Thus, if we pick
	\begin{equation}
		\label{eq:hf}
		k = C \lceil \sample^{1/(2 \expholder + 1)} \rceil 
	\end{equation}
	for a sufficiently large constant $C> 0$, we can ensure that 
	\begin{equation}
		\label{eq:id}
		\KL((\density^{(\ell)})^{\otimes \sample}, (\density^{(r)})^{\otimes \sample}) \le 0.1 \log(M) ,
	\end{equation}
	and we conclude by \cite[Theorem 2.5]{Tsy09} that
	\begin{equation}
		\label{eq:hg}
		\inf_{\tilde \density} \sup_{\densitygt \in \classholder(\expholder, \constholder)} \p_{(P^\ast)^{\otimes \sample}} \left( \hel^2(\tilde \density, \densitygt) \ge \constsep \sample^{\frac{-2 \expholder}{2 \expholder + 1}} \right) \ge \frac{1}{3},
	\end{equation}
	for a constant \( \constsep > 0 \).

	\subsubsection{Case \texorpdfstring{\( \expholder \ge 2 \)}{beta >= 2}.}

	For \( \expholder \ge 2 \), 
	we need to construct hypotheses that depend on both variables \( x \) and \( y \). 
	Let \( k \in \mathbb{N} \) to be determined later, and without loss of generality, assume that \( k \) is divisible by \( 4 \).
	Fix a non-zero, non-negative function \( f \in C^{\infty}(\R) \) with support in \( [0,1] \), and set 
	\begin{equation}
		\label{eq:hi}
		g(x,y) = f(x) f(y).
	\end{equation}
	Moreover, for \( i, j \in [k] \), define 
	\( w_{i, j} \) as the corners of 
	an equidistant partition of \( [0,1]^2 \) denoted by \( S_{i, j} \), that is, 
	\begin{align}
		\label{eq:hj}
		w_{i, j} = \Big( \frac{i-1}{k}, \frac{j - 1}{k} \Big)^\top 
		\quad \text{ and } \quad
		S_{i, j} =  w_{i, j} + \Big[ 0, \frac{1}{k} \Big]^2, \quad i, j \in [k] .
	\end{align}
	In addition, we let 
	\begin{equation}
		\label{eq:hh}
		g_{i, j} (z) = \frac{1}{k^\expholder} g(k ( z - w_{i, j} )), \quad z \in [0,1]^2,
	\end{equation}
	which is supported in \( S_{i, j} \). 
	Note that for any $r \in \N$ and $\alpha \in \{1, 2\}^r$, we have 
	\begin{equation}
		\label{eq:hy}
		\partial^{\alpha} g_{i, j} (z) = \frac{k^{r}}{k^\expholder} \partial^{\alpha} g(k(z - w_{i, j})) . 
	\end{equation}
	Since $g \in C^{\infty} (\R^2)$, it is easily verified that by definition, \( g_j \in  \classholder(\expholder, C_1) \) for some constant \( C_1 > 0 \) that only depends on \( g \). 

	By the Gilbert-Varshamov bound \cite[Lemma 4.10]{Mas07}, there exists a set \( \{\tau^{(\ell)}\}_{\ell=1}^M \) such that \( \tau^{(\ell)} \in \{0,1\}^{k \times k} \) and
	\begin{itemize} 
		\item \(\ham(0, \tau^{(\ell)} ) = k^2/4\), 
		\item \(\ham( \tau^{(\ell)} , \tau^{(r)} ) \ge k^2/8\) for all distinct \(\ell, r \in [M]\), and 
		\item \(\log(M) \ge k^2/30\).
	\end{itemize} 
	Next, we associate a density to each \( \tau^{(\ell)} \) such that we can control the pairwise distances between these densities.
	Similar to the proof of Theorem \ref{thm:lower}, we claim that there exists a useful choice of scaling constants that ensures that the normalization factor of the log-densities stays the same among all \( \ell \).

	For a fixed \( \delta \in [0,1] \), we claim that there exists a constant \( \constdeltaenlarge \) such that for every \( (i, j) \in [k]^2 \), there exists \( \tilde \delta_{i, j} \in [\delta, \constdeltaenlarge \delta] \) with
	\begin{equation}
		\label{eq:hn}
		\underbrace{\int_{S_{i, j}} \left[ \exp \left( z_1 z_2 + \tilde \delta_{i, j} g_{i, j}(z) \right) - \exp \left( z_1 z_2 \right) \right] \, \ud z}_{=:h_{i, j}(\tilde \delta_{i, j})}
		= \underbrace{\int_{S_{1, 1}} (\exp \left( 1 + \delta g_{1, 1}(z) \right) - e) \, \ud z}_{=: H}.
	\end{equation}
	To see why this is true, denote the left-hand side of \eqref{eq:hn} by \( h_{i, j}(\tilde \delta_{i, j}) \) and the right-hand side by \( H \) as above.
	Observe that \( h_{i, j}(\tilde \delta_{i, j}) \) is a continuous function of \( \tilde \delta_{i, j} \) as a consequence of the bounded convergence theorem.
	Hence, the intermediate value theorem allows us to conclude \eqref{eq:hn} if we can show that \( h_{i, j}(\delta) \le H \) and \( h_{i, j}(\constdeltaenlarge \delta) \ge H \).
	The first inequality \( h_{i, j}(\delta) \le H \) follows from the fact that
	\begin{equation}
		\label{eq:ie}
		\exp \left( z_1 z_2 + \tilde \delta_{i, j} g_{i, j}(z) \right) - \exp(z_1 z_2) \le \exp \left( 1 + \tilde \delta_{i, j} g_{i, j}(z) \right) - \exp(1)
	\end{equation}
	and changing the limits of the integral.
	The second inequality \( h_{i, j}(\constdeltaenlarge \delta) \ge H \) follows from the following estimates.
	For \( h_{i, j}(\constdeltaenlarge \delta) \), we have by the fundamental theorem of calculus and the fact that \( \exp(t) \ge 1 \) for \( t \ge 0 \), that
	\begin{align}
		\label{eq:if}
		h_{i, j}(\constdeltaenlarge \delta)
		= {} & \int_{S_{i, j}} \int_{z_1 z_2}^{z_1 z_2 + \constdeltaenlarge \delta g_{i, j}(z)} \exp(t) \, \ud t \, \ud z\\
		\ge {} & \int_{S_{i, j}} \int_{0}^{\constdeltaenlarge \delta g_{i, j}(z)} \exp(t) \, \ud t \, \ud z\\
		= {} & \int_{S_{1, 1}} \int_{0}^{\constdeltaenlarge \delta g_{1, 1}(z)} \exp(t) \, \ud t \, \ud z\\
		\ge {} & \constdeltaenlarge \delta \int_{S_{1, 1}} g_{1, 1}(z) \, \ud z.
	\end{align}
	On the other hand, for \( H \), we similarly have
	\begin{align}
		\label{eq:ig}
		H
		= {} & \int_{S_{1, 1}} \int_0^{\delta g_{1, 1}(z)} \exp(1 + t) \, \ud t \, \ud z\\
		= {} & e \int_{S_{1, 1}} \int_0^{\delta g_{1, 1}(z)} \exp(t) \, \ud t \, \ud z\\
		\le {} & \delta e \sup_{z \in S_{1, 1}} \exp(\delta g_{1, 1}(z)) \int_{S_{1, 1}} g_{1, 1}(z) \, \ud z\\
		\le {} & e \sup_{z \in S_{1, 1}} \exp(g_{1, 1}(z)) \int_{S_{1, 1}} g_{1, 1}(z) \, \ud z.
	\end{align}
	By definition, it holds that $\| g_{1, 1} \|_\infty \le \|g\|_\infty$. 
	Therefore, with the above estimates combined, we see that if \( \constdeltaenlarge \ge e \cdot \exp( \|g\|_\infty) \), then \( h_{i, j}(\constdeltaenlarge \delta) \ge H \), and thus \eqref{eq:hn} is proved.

	For each $\tau^{(\ell)}$, let us define 
	\begin{equation}
		\label{eq:hk}
		\tilde \eta^{(\ell)}(z) \defn z_1 z_2 + \sum_{i, j \in [k]} \tau^{(\ell)}_{i, j} \tilde \delta_{i, j} g_{i, j}(z),
	\end{equation}
	which after normalization leads to the log-densities
	\begin{equation}
		\label{eq:hl}
		\eta^{(\ell)}(z) \defn \tilde \eta^{(\ell)}(z) - \log \underbrace{\int_{[0,1]^2} \exp(\tilde \eta^{(\ell)}(w)) \, \ud w}_{=: \mathfrak{N}},
	\end{equation}
	and the densities \( \density^{(\ell)}(z) \defn \exp(\eta^{(\ell)}(z)) . \)

	As in the proof of Theorem \ref{thm:lower}, \( \mathfrak{N} \) does not depend on \( \ell \), since by the fact that $g_{i, j}$ is supported on $S_{i, j}$, then \eqref{eq:hn} and that \( \ham(0, \tau^{(\ell)}) = k^2/4 \), we have 
	\begin{align}
		\label{eq:hm}
		\mathfrak{N}
		= {} & \int_{[0,1]^2} \exp \bigg( z_1 z_2 + \sum_{i, j \in [k]} \tau^{(\ell)}_{i, j} \tilde \delta_{i, j} g_{i, j}(z) \bigg) \, \ud z \\
		= {} & \sum_{i,j \in [k]} \int_{S_{i, j}} \exp \bigg( z_1 z_2 + \tau^{(\ell)}_{i, j} \tilde \delta_{i, j} g_{i, j}(z) \bigg) \, \ud z\\
		= {} & \sum_{i,j \in [k]} \int_{S_{i, j}} \exp \left( z_1 z_2 \right) \, \ud z + \sum_{(i, j) : \tau^{(\ell)}_{i, j} = 1} \left[ \int_{S_{i, j}} \exp \left( z_1 z_2 + \tilde \delta_{i, j} g_{i, j}(z) \right) \, \ud z - \int_{S_{i, j}} \exp \left( z_1 z_2 \right) \, \ud z \right]\\
		= {} & \int_{[0,1]^2} \exp \left( z_1 z_2 \right) \, \ud z + \ham(0, \tau^{(\ell)}) \int_{S_{1, 1}} (\exp \left( 1 + \delta g_{1, 1}(z) \right) - e) \, \ud z\\
		= {} & \int_{[0,1]^2} \exp \left( z_1 z_2 \right) \, \ud z + \frac{k^2}{4} \int_{S_{1, 1}} (\exp \left( 1 + \delta g_{1, 1}(z) \right) - e) \, \ud z\\
		= {} & \int_{[0,1]^2} \exp \left( z_1 z_2 \right) \, \ud z + \frac{1}{4} \int_{[0,1]^2} \left( \exp \left( 1 + \frac{\delta}{k^{\expholder}} g(z) \right) - e \right) \, \ud z.
		\label{eq:ii}
	\end{align}
	Additionally, from the last line \eqref{eq:ii} of the above calculation, we can also conclude that \( \mathfrak{N} \) can be bounded from above and below by positive constants independent from \( k \).

	Moreover, for all \( \ell \) and \( z \in [0,1]^2 \), it holds that
	\begin{align}
		\label{eq:ih}
		\partial_1 \partial_2 \eta^{(\ell)}(z)
		= {} & \partial_1 \partial_2 \tilde \eta^{(\ell)}(z)\\
		= {} & 1 + \sum_{i,j \in [k]} \tau^{(\ell)}_{i, j} \tilde \delta_{i, j} \partial_1 \partial_2 g_{i, j}(z)\\
		= {} & 1 + \sum_{i,j \in [k]} \tau^{(\ell)}_{i, j} \tilde \delta_{i, j} \frac{k^2}{k^\expholder} \partial_1 \partial_2 g(k(z - w_{i, j}))\\
		\ge {} & 1 - \constdeltaenlarge \delta \frac{k^2}{k^\expholder} \sup_{z} | \partial_1 \partial_2 g(z)|,
	\end{align}
	which, in view of \( \expholder \ge 2 \), can be made positive if \( \delta \) is chosen to be a sufficiently small constant. 
	Lemma \ref{lem:diff-char} then implies that all \( \density^{(\ell)} \) are $\mtp$ densities. 
	
	In addition, for any $r \in \N$ and $\alpha \in \{1, 2\}^r$, we have by definition 
	\begin{align}
	\partial^\alpha \rho^{(\ell)} (z) &= \rho^{(\ell)} (z) \cdot \partial^\alpha  \eta^{(\ell)} (z) \\
	&= \exp \bigg( z_1 z_2 + \sum_{i, j \in [k]} \tau^{(\ell)}_{i, j} \tilde \delta_{i, j} g_{i, j}(z) - \log \mathfrak{N} \bigg) \cdot \bigg( \partial^\alpha (z_1 z_2) + \sum_{i, j \in [k]} \tau^{(\ell)}_{i, j} \tilde \delta_{i, j} \partial^\alpha g_{i, j}(z) \bigg) . 
	\end{align}
	Since $\mathfrak{N}$ is bounded from above and below by positive constants, the first factor (that is, $\rho^{(\ell)} (z)$) is bounded from above and below. 
	Also, we have that \( g_{i, j} \in  \classholder(\expholder, C_1) \) and $\tilde \delta_{i, j} \le C_3 \delta$, so it is easily seen that if $\delta$ is chosen to be a sufficiently small constant, then $\rho^{(\ell)} \in \classholder(\expholder, 1)$ by definition. 

	Finally, we bound \( \KL(\density^{(\ell)}, \density^{(r)}) \) and \( \hel(\density^{(\ell)}, \density^{(r)}) \). 
	We have seen above that $\rho^{(\ell)}$ can be bounded from above and below, that is, $C_4^{-1} \le \rho^{(\ell)} \le C_4$ for a constant $C_4 > 0$. 
	Moreover, we can choose $C_4$ so that $- C_4 \le \eta^{(\ell)} \le C_4$. 
%
	For the Hellinger distance, we write
	\begin{align}
		\label{eq:il}
		\hel^2(\density^{(\ell)}, \density^{(r)})
		= {} & \int_{[0,1]^2} \left( \sqrt{\density^{(\ell)}(z)} - \sqrt{\density^{(r)}(z)} \right)^2 \, \ud z\\
		= {} & \int_{[0,1]^2} \left( \exp(\eta^{(\ell)}(z)/2) - \exp(\eta^{(r)}(z)/2) \right)^2 \, \ud z\\
		= {} & \int_{[0,1]^2} \density^{(\ell)}(z) \left( 1 - \exp(\eta^{(r)}(z)/2 - \eta^{(\ell)}(z)/2) \right)^2 \, \ud z . 
	\end{align}
	By the Taylor expansion, we can obtain a quadratic control of the exponential term of the form
	\begin{equation}
		\label{eq:in}
		x^2 \le (1 - \exp(x))^2 \le \constexp x^2, \quad x \in [-\constbounded, \constbounded],
	\end{equation}
	where \( \constexp > 0 \). This allows us to bound
	\begin{equation}
		\label{eq:io}
		\frac{1}{4 \constbounded} \int_{[0,1]^2} (\tilde \eta^{(\ell)}(z) - \tilde \eta^{(r)}(z))^2 \, \ud z
		\le \hel^2(\density^{(\ell)}, \density^{(r)})
		\le \frac{ \constexp \constbounded }{4} \int_{[0,1]^2} (\tilde \eta^{(\ell)}(z) - \tilde \eta^{(r)}(z))^2 \, \ud z.
	\end{equation}
	Taking into account \eqref{eq:hel-kl}, it remains to bound the \( L^2 \) distance between \( \tilde \eta^{(\ell)} \) and $\tilde \eta^{(r)}$. Using again that the support of $g_{i, j}$ is in $S_{i, j}$, we have
	\begin{align}
		\label{eq:hq}
		\int_{[0,1]^2} (\tilde \eta^{(\ell)}(z) - \tilde \eta^{(r)}(z))^2 \, \ud z
		= {} & \int_{[0,1]^2} \bigg( \sum_{i, j \in [k]} ( \tau^{(\ell)}_{i, j} - \tau^{(r)}_{i, j} )  \tilde \delta_{i, j} g_{i, j}(z) \bigg)^2 \, \ud z \\
		= {} & \sum_{i, j \in [k]} (\tau^{(\ell)}_{i, j} - \tau^{(r)}_{i, j})^2 \tilde \delta_{i, j}^2 \int_{S_{i, j}} g_{i, j}(z)^2 \, \ud z\\
		\le {} & \sum_{i, j \in [k]} (\tau^{(\ell)}_{i, j} - \tau^{(r)}_{i, j})^2 \constdeltaenlarge^2 \delta^2 \int_{S_{1, 1}} g_{1, 1}(z)^2 \, \ud z \\
		= {} & \ham(\tau^{(\ell)}, \tau^{(r)}) \constdeltaenlarge^2 \delta^2 \frac{1}{k^{2\expholder+2}} \int_{[0,1]^2} g(z)^2 \, \ud z
		\le
		\frac{\constdeltaenlarge^2 \delta^2}{k^{2 \expholder}} \int_{[0,1]^2} g(z)^2 \, \ud z,
	\end{align}
	where we changed the limits of integration by substitution and used that \( \ham(\tau^{(\ell)}, \tau^{(r)}) \le k^2 \).
	Similarly, we can derive a lower bound of the same order, using that $\tilde \delta_{i, j} \ge \delta$ and \( \ham(\tau^{(\ell)}, \tau^{(r)}) \ge k^2/4 \). In conclusion, there exists a constant \( \constltwo > 0 \) such that
	\begin{equation}
		\label{eq:ip}
		\frac{1}{\constltwo} \frac{\delta^2}{k^{2 \expholder}}
		\le \hel^2(\density^{(\ell)}, \density^{(r)})
		\le \constltwo \frac{\delta^2}{k^{2 \expholder}}.
	\end{equation}

	To finish, we note that for the \( \KL \) condition of \cite[Theorem 2.5]{Tsy09} to hold, we need
	\begin{equation}
		\label{eq:iq}
		\KL((\density^{(\ell)})^{\otimes \sample}, (\density^{(r)})^{\otimes \sample})
		= \sample \KL(\density^{(\ell)}, \density^{(r)})
		\le C_7 \sample \hel^2(\density^{(\ell)}, \density^{(r)})
		\le \sample \constltwo C_7 \frac{\delta^2}{k^{2 \expholder}}
		\le 0.1 \log(M),
	\end{equation}
	which, in view of the bound \( \log(M) \ge k^2/30 \), can be fulfilled by choosing \( \delta \) to be a sufficiently small constant and 
	$
		k = \big\lceil \sample^{\frac{1}{2 \expholder + 2}} \big\rceil . 
	$
	This then leads to a separation of the hypotheses of
	\begin{equation}
		\label{eq:is}
		\hel(\density^{(\ell)}, \density^{(r)})
		\ge c_8 \sample^{\frac{- 2 \expholder}{2 \expholder + 2}},
	\end{equation}
	so \cite[Theorem 2.5]{Tsy09} yields
	\begin{equation}
		\label{eq:it}
		\inf_{\tilde \density} \sup_{\densitygt \in \classholder(\expholder, \constholder)} \p_{(P^\ast)^{\otimes \sample}} \left( \hel^2(\tilde \density, \densitygt) \ge c_8 \sample^{\frac{-2 \expholder}{2 \expholder + 2}} \right) \ge \frac{1}{3}.
	\end{equation}	

	\subsubsection{Case \texorpdfstring{\( 1 < \expholder < 2 \)}{1 < beta < 2}.}

	For \( 1 < \expholder < 2 \), note that \( \classholder(2, \constholder) \subseteq \classholder(\expholder, \constholder) \), so the above construction in the case \( \expholder = 2 \) still remains valid, which we can use to conclude
	\begin{equation}
		\label{eq:iu}
		\inf_{\tilde \density} \sup_{\densitygt \in \classholder(\expholder, \constholder)} \p_{(P^\ast)^{\otimes \sample}} \left( \hel^2(\tilde \density, \densitygt) \ge \constsep \sample^{-\frac{2}{3}} \right) \ge \frac{1}{3},
	\end{equation}
	which finishes the proof.

\section{Conclusion and discussion}
\label{sec:conclusion}

In this work, we studied minimax estimation of discrete and continuous two-dimensional totally positive distributions. Particularly, for estimation of $\beta$-H\"older smooth distributions, we established the minimax rates of estimation in the squared Hellinger distance up to polylogarithmic factors, for any $\beta \ge 0.5$. In addition, we proposed and implemented efficient algorithms to compute our estimators. The numerical experiments supported our theoretical findings. 

Several questions are left open for future research. First, for $\beta \in (0, 0.5)$, the upper bound for our estimator does not match the minimax lower bound. Moreover, our bounds do not capture the optimal dependency on the pointwise infimum or supremum of the ground-truth density. These are possibly artifacts of our estimation procedure or proofs. 
Second, we studied a variant of the MLE with an extra box constraint. While this box-constrained MLE has almost the same computational cost and empirical performance as the original MLE, it is theoretically more desirable to establish the same guarantees for the original MLE. 
Third, it is of significant interest to study estimation of totally positive distributions in general dimensions. However, our current proof techniques do not generalize to higher dimensions straightforwardly, and we leave this to future research.

\appendix

\section{Nonexistence of MLE under \texorpdfstring{$\mtp$}{MTP2} constraint alone}
\label{sec:ill-posedness}

In this section, we show that without further regularity assumptions on the underlying densities, the MLE under the $\mtp$ constraint does not exist.

\begin{lemma}
	\label{lem:ill-posedness}
	Let \( \densitygt \) be an \( \mtp \) density on \( [0,1]^2 \) with respect to the Lebesgue measure. Let \( X_1, \dots, X_N \) be \( N \) \iid observations from the corresponding probability distribution.
	Then, the optimization problem
	\begin{equation}
		\max \sum_{i=1}^{N} \log \density(X_i) \quad \text{\emph{subject to: }} \density \text{ is an } \mtp \text{ density w.r.t. the Lebesgue measure}
	\end{equation}
	is almost surely unbounded. Consequently, the MLE under the $\mtp $ constraint does not exist.
\end{lemma}

\begin{proof}
	Denote by \( \p_{\densitygt} \) the probability distribution corresponding to \( \densitygt \) and by \( \p_{\densitygt}^{\otimes N} \) the probability distribution of \( N \) \iid observations from \( \p_{\densitygt} \).
	Let
	\begin{equation}
		\label{eq:kc}
		\mathcal{A} = \{ (X_i)_1 \neq (X_j)_1 \text{ for all } i \neq j \}.
	\end{equation}
	Then,
	\begin{align}
		\p_{\densitygt}^{\otimes N}(\mathcal{A})
		= {} & 1 - \p_{\densitygt}^{\otimes N}\Big(\bigcup_{i, j \in [N], \, i \neq j} \left\{ (X_i)_1 = (X_j)_1 \right\} \Big)\\
		\ge {} & 1 - \sum_{i, j \in [N], \, i \neq j} \int_{[0,1]^4} \1\{x_1 = x_2\} \densitygt(x_1, y_1) \densitygt(x_2, y_2) d x_1 \, d x_2 \, d y_1 \, d y_2 \label{eq:illdef-integral}\\
		= {} & 1 \label{eq:illdef-0},
	\end{align}
	where the second line \eqref{eq:illdef-integral} follows from the sub-additivity of the probability measure and the definition of \( \mathcal{A} \), and the third line \eqref{eq:illdef-0} follows from the fact that the integrand in \eqref{eq:illdef-integral} is only non-zero on a lower dimensional subset of \( [0,1]^4 \) and hence is zero almost everywhere with respect to the Lebesgue measure.

	Similarly, if \( \mathcal{B} = \{ X_i \notin \{0, 1\} \text{ for all } i \} \), then \( (P^\ast)^{\otimes N}(\mathcal{B}) = 1 \).

	For the rest of the proof, assume that the event \( \mathcal{A} \cap \mathcal{B} \) occurred.
	Because of the definitions of \( \mathcal{A} \) and \( \mathcal{B} \), and the fact that \( N \) is finite, the minimum distance between the first coordinates is positive, as is the minimum distance to any of the interval boundaries, that is,
	\begin{equation}
		\epsilon_0 = \Big( \min_{i, j \in [N],\, i \neq j} | (X_i)_1 - (X_j)_1 | \Big) \wedge \Big( \min_{i \in [N]} (X_i)_1 \Big) \wedge \Big( \min_{i \in [N]} \big( 1 - (X_i)_1 \big) \Big)  > 0.
	\end{equation}
	Let \( f \in C^\infty(\R) \) be a non-negative bump function supported in \( [-1, 1] \) such that \( \int_{\R} f(x) d x = 1 \) and \( f(0) = f_0 > 0 \).
	For \( 0 < \epsilon < \epsilon_0/2 \), set
	\begin{equation}
		\rho_\epsilon(x, y) = \frac{1}{N} \sum_{i = 1}^{N} \frac{1}{\epsilon}f \left( \frac{x - (X_i)_1}{\epsilon} \right).
	\end{equation}
	Then,
	\begin{align}
		\label{eq:kd}
		\int_{[0,1]^2} \rho_\epsilon(x, y) \, d x \, d y
		= {} & \int_{[0,1]} \frac{1}{N} \sum_{i = 1}^{N} \frac{1}{\epsilon} f \left( \frac{x - (X_i)_1}{\epsilon} \right) d x\\
		= {} & \frac{1}{N} \sum_{i = 1}^{N} \int_{\R} f(\xi) \, d \xi = 1,
	\end{align}
	so \( \rho_\epsilon \) is again a probability distribution on \( [0,1]^2 \).
	Moreover, because \( \rho_\epsilon \) does not depend on \( y \), by Lemma~\ref{lem:diff-char}, \( \rho_\epsilon \) is \( \mtp \).

	Finally, for the log-likelihood, we obtain
	\begin{align}
	 \sum_{i = 1}^{N} \log (\rho_\epsilon(X_i)) 
	 = {} & \sum_{i = 1}^{N} \log \left( \frac 1N \sum_{j = 1}^{N} \frac{1}{\epsilon} f \left( \frac{(X_i)_1 - (X_j)_1}{\epsilon} \right) \right)\\
	 = {} & \sum_{i = 1}^{N} \log \left( \frac{f_0}{\epsilon} \right) \label{eq:ke}\\
	 = {} & N \log \left( \frac{f_0}{\epsilon} \right) \to \infty, \quad \text{for } \epsilon \to 0,
 \end{align}
 where \eqref{eq:ke} follows because by the definition of \( \mathcal{A} \) and of \( \rho_\epsilon \), the individual bumps centered at the observations \( X_i \) do not intersect.
 Combined, by choosing \( \epsilon \) arbitrarily small, we can obtain an arbitrarily large log-likelihood.
 In turn, the MLE does not exist.
\end{proof}

\begin{remark}
	\label{rem:ill-posedness}
	Even if the MLE is not defined, there could potentially exist a different estimator over the whole class of \( \mtp \) with good estimation properties.
	However, the estimation problem over the whole \( \mtp \) class bears other signs of ill-posedness:
	Since
	\begin{equation}
		\label{eq:kf}
		\bigcup_{\beta \in (0, 1)} \classholder(\expholder, R) \subseteq \{ \rho : \rho \text{ is } \mtp \},
	\end{equation}
	the lower bound in Theorem \ref{thm:lower-cts} suggests that no estimator \( \hat \rho \) can attain a polynomial estimation rate of
	\begin{equation}
		\label{eq:kg}
		\hel^2(\hat \rho, \rho^\ast) \lesssim N^{-\alpha},
	\end{equation}
	for any \( \alpha > 0 \) over the whole \( \mtp \) class.
	While this does not explicitly exclude possibly slower rates of convergence such as \( \log(N)^{-1} \), this still serves to show that the estimation problem without further regularity assumptions is ill-posed in the sense of not admitting polynomially fast rates.
\end{remark}

\section{Existing results}
\label{sec:tools}

We state and prove some results that are known or follow easily from existing ones. 

\subsection{Concentration of multinomial random variables}

The following is a standard tail bound for a binomial random variable.


\begin{lemma} \label{lem:bin-tail}
Suppose that $Y$ has the binomial distribution $\Bin(N, x)$, where $N$ is a positive integer and $x \in (0,1)$. 
Then for $y \in [0, 1]$, we have
$
|Y - N x| \le N y 
$
with probability at least $1 - 2 \exp \big( -N \frac{y^2}{2 (x+y)} \big)$.
\end{lemma}

\begin{proof}
This follows immediately from Lemma~6 of \cite{MaoWeeRig18} by taking $r = (x-y) \lor 0$ and $s = (x+y) \land 1$. 
\end{proof}

Next, we present a lemma that follows from Bernstein's inequality. 
Recall that for a vector $a \in \R^m$ and an entrywise positive vector $b \in \R^m$, we denote the $b$-weighted $\ell_2$-norm of $a$ by $\|a\|_b = ( \sum_{i=1}^m b_i a_i^2 )^{1/2}$.

\begin{lemma}
\label{lem:multi-bern}
Suppose that $Y$ is a random vector in $\R^m$ having the multinomial distribution $\Multi ( N, p )$, where $N$ is a positive integer and $p = (p_1, \dots, p_m)^\top$ is a vector in $(0, 1)^m$ with $\sum_{i=1}^m p_i = 1$. Then, for any vector $a\in\R^m$,
$$
\p \Big\{ \big| \langle Y - Np, a \rangle \big| \geq t \Big\} \leq 2 \exp \Big( \frac{ - 3 t^2 }{ 6 N \|a\|_p^2 + 4 \|a\|_\infty t } \Big) . 
$$
\end{lemma}

\begin{proof}
Let $I_1, \dots, I_N$ be i.i.d. $\Multi(1, p)$ random variables. That is, we have $I_j = i$ with probability $p_i$ for each $i \in [m]$ and $j \in [N]$. Then we have $Y_i = \sum_{j=1}^N \1 \{I_j = i\}$, and thus 
\begin{align*}
\langle Y - Np, a \rangle &= \sum_{i=1}^m (Y_i - N p_i) a_i 
= \sum_{i=1}^m \Big( \sum_{j=1}^N \1 \{I_j = i\} - N p_i \Big) a_i \\
&= \sum_{j=1}^N \sum_{i=1}^m ( \1 \{I_j = i\} - p_i ) a_i 
= \sum_{j=1}^N \Big( a_{I_j} - \sum_{i=1}^m p_i a_i \Big) 
= \sum_{j=1}^N \Big( a_{I_j} - \E [a_{I_j}] \Big) . 
\end{align*}
Since this is a sum of i.i.d. zero-mean random variables with absolute values bounded by $2 \|a\|_\infty$, Bernstein's inequality (Theorem~2.8.4 of \cite{Ver18}) implies that 
$$
\p \Big\{ \big| \langle Y - Np, a \rangle \big| \geq t \Big\} \leq 2 \exp \Big( \frac{ - t^2/2 }{ \sigma^2 + 2 \|a\|_\infty t/3 } \Big) , 
$$
where $\sigma^2 = N \E (a_{I_j} - \E [a_{I_j}])^2 \le N \E [ a_{I_j}^2 ] = N \sum_{i=1}^m p_i a_i^2 = N \|a\|_p^2.$
\end{proof}

The following lemma is concerned with projections of a multinomial random vector. 

\begin{lemma}
\label{lem:multi-max}
Suppose that $Y$ is a random vector in $\R^m$ having the multinomial distribution $\Multi ( N, p )$, where $N$ is a positive integer and $p = (p_1, \dots, p_m)^\top$ is a vector in $(0, 1)^m$ with $\sum_{i=1}^m p_i = 1$.  Given vectors $v_1, \ldots, v_\ell \in \R^m$, for any $\delta \in (0, 1]$, it holds with probability at least \( 1 - \delta \) that
$$
\max_{j \in [\ell]} \big| \langle Y - N p, v_j \rangle \big| \lesssim  \Big( \max_{j \in [\ell]} \| v_j \|_p \Big) \sqrt{ N  \log(\ell / \delta) } + \Big( \max_{j \in [\ell]} \| v_j \|_\infty \Big) \log(\ell / \delta). 
$$
\end{lemma}

\begin{proof}
The result follows from Lemma~\ref{lem:multi-bern} and a union bound, with the choice of $t$ equal to a constant times the right-hand side of the above inequality.
\end{proof}

\subsection{MLE for \texorpdfstring{$\mtp$}{MTP2} distributions on a grid} \label{sec:exist}

Given the observation $Y$ defined by \eqref{eq:jy}, it is well known that the MLE \eqref{eq:mle-1} can be equivalently defined using the following convex program, which can be solved efficiently:
\begin{align} \label{eq:mle-2}
\mle \defn \argmax_{\Done \theta \Dtwo \ge 0} \frac{1}{N} \langle Y, \theta \rangle - \sum_{i, j} e^{\theta_{i, j}} .
\end{align}

\begin{lemma} \label{lem:equivalence}
The two definitions \eqref{eq:mle-1} and \eqref{eq:mle-2} of the MLE $\hat \theta = \mle$ are equivalent.
\end{lemma}

\begin{proof}
It suffices to verify that $\hat \theta$ given by program \eqref{eq:mle-2} always satisfies $\sum_{i, j} e^{\hat \theta_{i, j}} = 1$. Suppose this is not the case, and define $\tilde \theta \in \R^{\dimone \times \dimtwo}$ by $\tilde \theta_{i, j} = \hat \theta_{i, j} - \log \sum_{k, \ell} e^{\hat \theta_{k, \ell}}$ so that $\sum_{i, j} e^{\tilde \theta_{i, j}} = 1$. Then we have
\begin{align}
\frac 1N \langle Y, \tilde \theta \rangle - \sum_{i, j} e^{\tilde \theta_{i, j}} = \frac 1N \langle Y, \hat \theta \rangle - \frac 1N \Big( \sum_{i, j} Y_{i, j} \Big) \log \sum_{i, j} e^{\hat \theta_{i, j}} - 1 > \frac 1N \langle Y, \hat \theta \rangle - \sum_{i, j} e^{\hat \theta_{i, j}} ,
\end{align}
since $\sum_{i, j} Y_{i, j} = N$ and $\log (x) + 1 < x$ for any $x \ne 1$. However, this gives a contradiction.
\end{proof}

\subsection{Rate of convergence of the empirical frequency matrix}
\label{sec:emp}

Let us consider the empirical frequency matrix $Y/N$ in the discrete setting, where $Y$ is defined by \eqref{eq:jy}. Without leveraging the $\mtp$ constraint, it achieves the following trivial rate of estimation. 

\begin{lemma}
\label{lem:emp}
In the setting of Section~\ref{sec:grid}, the empirical frequency matrix $Y/N$ satisfies
$$
\E \big[ \hel^2(p^*, Y/N) \big] \le \frac{\dimone \dimtwo}{ N } . 
$$
\end{lemma}

\begin{proof}
For any $i \in [\dimone]$ and $j \in [\dimtwo]$, we have $Y_{i,j} \sim \Bin(N, p^*_{i,j})$ marginally. 
Thus we have
\begin{align}
\E \big[ \hel^2(p^*, Y/N) \big]
&= \sum_{i=1}^{\dimone} \sum_{j=1}^{\dimtwo} \E \Big( \sqrt{ p^*_{i,j} } - \sqrt{ Y_{i,j}/N } \, \Big)^2 
= \sum_{i=1}^{\dimone} \sum_{j=1}^{\dimtwo} \E \Big( \frac{ p^*_{i,j} - Y_{i,j}/N }{ \sqrt{ p^*_{i,j} } + \sqrt{ Y_{i,j}/N } } \Big)^2 \label{eq:line1} \\
&\le  \sum_{i=1}^{\dimone} \sum_{j=1}^{\dimtwo} \E \frac{ (p^*_{i,j} - Y_{i,j}/N)^2 }{ p^*_{i,j} } 
= \sum_{i=1}^{\dimone} \sum_{j=1}^{\dimtwo} \frac{ p^*_{i,j} (1 - p^*_{i,j}) }{ p^*_{i,j} N } 
\le \frac{\dimone \dimtwo}{ N } . \notag
\end{align}
\end{proof}




\section{Further details on numerical experiments}
\label{sec:num-extended}

\subsection{Implementation details}
\label{sec:implementation-details}

All simulations are run with Julia 1.4.1 \cite{bezanson2017julia}, where, besides the standard library, we use the libraries Cubature (version 1.5.1), Distributions \cite{Distributions.jl-2019, 2019Distributionspaper} (version 0.23.2), StatsBase (version 0.33.0), PyPlot (version 2.9.0), and GLM \cite{JuliaGLM} (version 1.3.9).

Algorithm \ref{alg:fast-proj} is stopped at a relative distance in the Frobenius norm between two consecutive iterates of less than \( 10^{-6} \) or 400,000 iterations, whichever comes first.
Similarly, Algorithm \ref{alg:prox-newton} is stopped at a relative distance of \( 10^{-5} \) or 100 iterations.
The distribution \( p^\ast \) in \eqref{eq:ir} is sampled as a multinomial distribution via the Distributions package, while the distribution corresponding to the density \( \rho^\ast \) in \eqref{eq:ka} is sampled via rejection sampling from the corresponding Gaussian distribution.
For the calculation of Hellinger distances in the continuous case, we use numerical integration with the Cubature package.

For the calculation of the oracle estimator in Figure \ref{fig:cts_oracle}, we computed the corresponding estimators for \( n \in \{4, 7, 10, 15, 23, 36, 55, 84, 130, 201 \} \) and picked the \( n \) achieving the best squared Hellinger distance to the ground truth in each case.

\subsection{Numerical instability of the MLE for small values of \( N \)}
\label{sec:num-instability}

As observed in Section~\ref{sec:numerics}, for small values of \( N \), Algorithm~\ref{alg:prox-newton} can become unstable and values in the iterate \( \theta \) can underflow due to a large number of zeros in the empirical frequency matrix.
To illustrate this, we perform the same experiment as in Figure~\ref{fig:pmf_large_V_var_N} with \( N = 100 \), which leads to a large error for the unconstrained MLE, which we plot in Figure~\ref{fig:pmf_bad}.
However, this behavior can be remedied by introducing an additional constraint of
\begin{equation}
	\label{eq:h}
	\theta \in \mathcal{\tilde C} \defn \{\theta : \exp(\theta_{i, j}) \ge \epsilon, \quad i, j \in [\dimone] \times [\dimtwo] \}.
\end{equation}
in the calculation of the MLE, where \( \epsilon \) is small.
For example, in this experiment, we specify \( \epsilon = e^{-30} \).
This leads to the estimator
\begin{align}
	\label{eq:ab}
	\tilde \theta^{\mathrm{lb}} \defn {} & \argmax_{\substack{ \Done \theta \Dtwo^\top \ge 0 \\ \theta \in \mathcal{\tilde C} }} \frac{1}{\samples} \langle Y, \theta \rangle - \sum_{i \in [\dimone] , \,  j \in [\dimtwo]} e^{\theta_{i, j}}, \\
	\hat \theta^{\mathrm{lb}}_{i, j} \defn {} & \tilde \theta^{\mathrm{lb}}_{i, j} - \log \sum_{r \in [\dimone], \, s \in [\dimtwo]} e^{\tilde \theta^{\mathrm{lb}}_{r, s}} \quad \text{ for } i \in [\dimone], \, j \in [\dimtwo] .
\end{align}
The constraint \eqref{eq:h} can be incorporated into Algorithm~\ref{alg:fast-proj} in the same way as the constraint \( \theta \in \mathcal{C}(Y) \) by iterative projection of each component \( \theta_{i, j} \) onto the corresponding interval \( [\log(\epsilon), \infty) \).
As can be seen in Figure~\ref{fig:pmf_bad}, this modification (``lower-bounded MLE'') is sufficient to overcome the problem of numerical instability when facing a small sample size.

\begin{figure}[H]
		\includegraphics[width=0.45\textwidth]{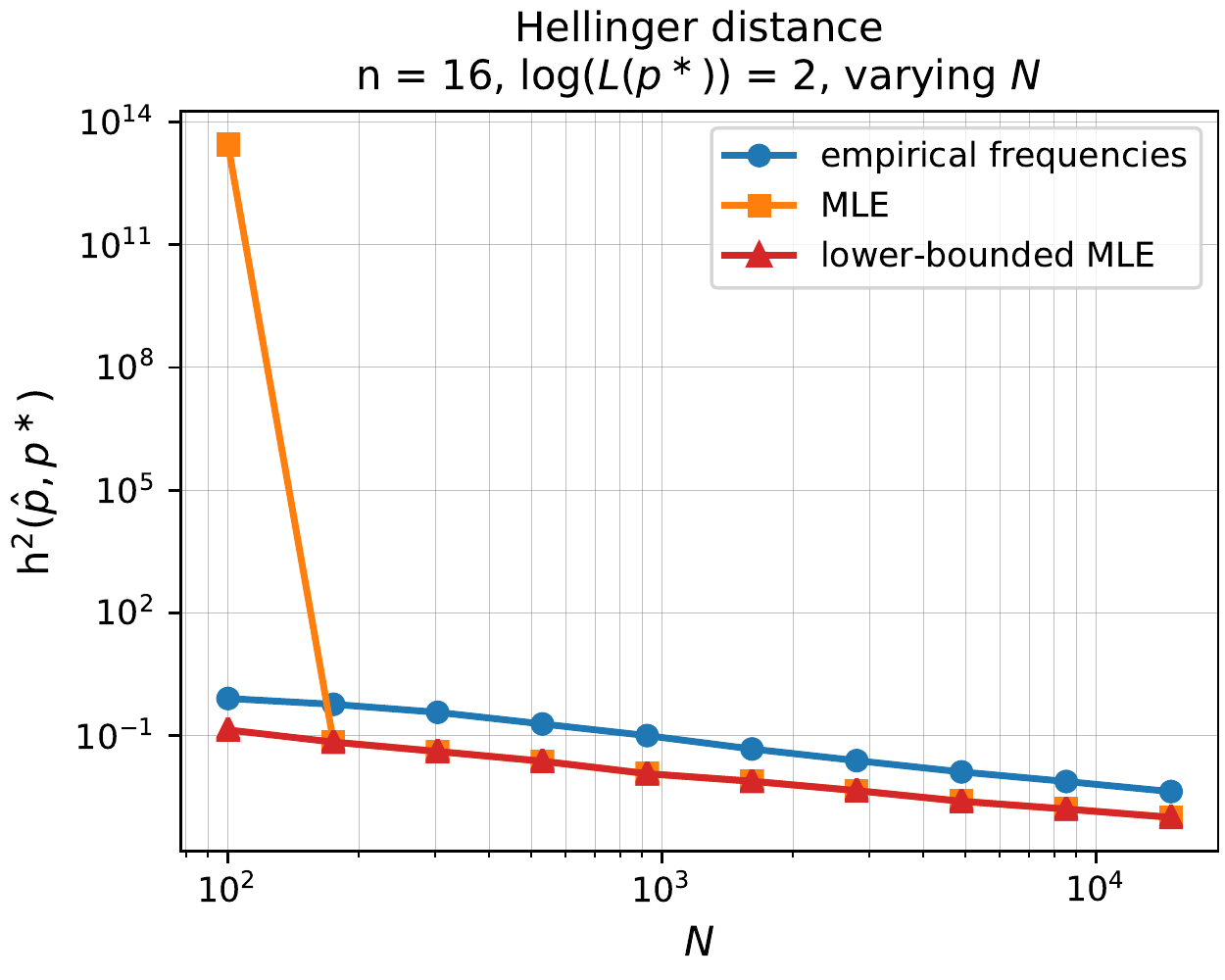}
		\caption{Instability for small sample sizes in the Hellinger distance for varying \( N \) and \( \log(L(p^\ast)) = 2 \)}
		\label{fig:pmf_bad}
\end{figure}

\section*{Acknowledgments}
PR was supported in part by NSF awards IIS-1838071, DMS-1712596 and DMS-TRIPODS-1740751; ONR grant N00014-17- 1-2147 and grant 2018-182642 from the Chan Zuckerberg Initiative DAF. ER was supported in part by an NSF MSPRF DMS-1703821. 
We thank the anonymous reviewers for their constructive comments.

\bibliographystyle{abbrv}
\bibliography{density}

\end{document}